\documentclass[11pt, reqno]{amsart}
\usepackage{hyperref, bbm, amssymb, mathtools, url}
\usepackage{amsmath,stmaryrd}
\usepackage[mathscr]{euscript}
\usepackage{cleveref}
\usepackage{fullpage}
\usepackage{tikz}
\usepackage{enumitem}
\usepackage[alphabetic]{amsrefs}
\usetikzlibrary{matrix,arrows,decorations.pathmorphing, cd}
\usepackage{comment}
\newcommand{\qednow}{\pushQED{\qed}\qedhere\popQED}
\usepackage{microtype}
\usepackage{calc}

\newcommand{\Hom}       {\operatorname{Hom}}
\newcommand{\iHom}      {\operatorname{\underline{\Hom}}}

\newcommand{\tFun}       {\operatorname{\mathbf{Fun}}}

\newcommand{\Map}       {\operatorname{Map}}

\newcommand{\map}       {\operatorname{map}}
\newcommand{\Fun}       {\operatorname{Fun}}

\newcommand{\Cart}       {\operatorname{Cart}}

\newcommand{\op}        {\operatorname{op}}
\newcommand{\opl}        {\operatorname{opl}}
\newcommand{\laxlim}    {\operatorname*{laxlim}}
\newcommand{\laxlimdag}    {\operatorname*{laxlim^\dagger}}

\newcommand{\oplaxcolimdag}    {\operatorname*{oplaxcolim^\dagger}}
\newcommand{\oplaxlim}    {\operatorname*{oplaxlim}}
\newcommand{\oplaxlimdag}    {\operatorname*{oplaxlim^\dagger}}

\newcommand{\gl}	{\mathrm{gl}}

\newcommand{\Tr}	{\mathrm{Tr}}

\newcommand{\Spec}{\operatorname{Spec}}

\newcommand{\Glo}[1]{\operatorname{Glo}_{#1}}   
\newcommand{\Spc}{\mathcal{S}}          
\newcommand{\Spcgl}[1]{\mathcal{S}_{#1\text{-}\gl}}   
\newcommand{\Spcglq}[1]{\Spcgl{\CF}^{\mathrm{fq}}} 
\newcommand{\Spgl}[1]{\Sp_{#1\text{-}\gl}}

\newcommand{\Ar}{\operatorname{Ar}}     

\newcommand{\lax}{\mathrm{lax}}
\newcommand{\lex}{\mathrm{lex}}

\newcommand{\cart}{\mathrm{ct}}

\newcommand{\LocSys}{\mathrm{LocSys}}
\newcommand{\oplaxslice}{\mathbin{\downarrow^{\opl}}}
\newcommand{\laxslice}{\mathbin{\downarrow^{\lax}}}
\newcommand{\ul}[1]{{\underline{#1}}}
\newcommand{\Nat}{\operatorname{Nat}}

\newcommand{\CB}        {{\mathbf{B}}}
\newcommand{\CC}        {{\mathcal{C}}}
\newcommand{\CR}        {{\mathcal{R}}}
\newcommand{\CD}        {{\mathcal{D}}}
\newcommand{\CE}        {{\mathcal{E}}}
\newcommand{\CG}        {{\mathcal{G}}}
\newcommand{\CI}        {{\mathcal{I}}}
\newcommand{\CJ}        {{\mathcal{J}}}

\newcommand{\CL}        {{\mathcal{L}}}
\newcommand{\CM}        {{\mathcal{M}}}
 
\newcommand{\CQ}        {{\mathcal{Q}}}
\newcommand{\CT}        {{\mathcal{T}}}

\newcommand{\CW}        {{\mathcal{W}}}
\newcommand{\CF}        {{\mathcal{F}}}
\newcommand{\CO}        {{\mathcal{O}}}

\renewcommand{\CW}      {{\mathcal{W}}}

\newcommand{\Q}         {{\mathbb{Q}}}
\newcommand{\Z}         {{\mathbb{Z}}}

\newcommand{\T}          {\mathbb{T}}

\newcommand{\CBD}        {{\mathbb{D}}}

\newcommand{\rP}         {\mathbf{P}}
\newcommand{\rL}         {\mathrm{L}}
\newcommand{\rR}         {\mathrm{R}}

\newcommand{\fF}         {\mathrm{F}}

\newcommand{\bG}        {{\mathbf{G}}}

\newcommand{\bS}        {{\mathbf{S}}}
\newcommand{\bE}        {{\mathbf{E}}}
\newcommand{\bC}        {{\mathbf{C}}}

\newcommand{\bGamma}{\mathpalette\makebbGamma\relax}
\newcommand{\makebbGamma}[2]{%
\raisebox{\depth}{\scalebox{1}[-1]{$\mathsurround=0pt#1\mathbb{L}$}}%
}

\newcommand{\sfS}       {\mathsf{S}}

\newcommand{\Mod}{\operatorname{Mod}}   
\newcommand{\CAlg}{\operatorname{CAlg}} 
\newcommand{\QCoh}{\operatorname{QCoh}} 

\newcommand{\Aball}       {{\mathrm{Ab}}}

\newcommand{\Orb}[1]       {\mathrm{Orb}_{#1}}

\newcommand{\Tori}       {\mathrm{Tori}}

\newcommand{\Sp}        {\mathrm{Sp}}

\newcommand{\Cat}       {\mathrm{Cat}}

\newcommand{\tCat}       {\mathrm{\mathbf{Cat}}}

\newcommand{\tC}		{\mathbf{C}}

\newcommand{\rtri}        {\triangleright}

\newcommand{\PrR}       {\mathrm{Pr}^{\rR,\mathrm{lax}}_{\mathrm{st}}}

\newcommand{\PrL}       {\CAlg(\mathrm{Pr}^{\mathrm{L}}_{\mathrm{st}})}

\newcommand{\pt}       {\mathrm{pt}}
\newcommand{\ab}       {\mathrm{ab}}
\newcommand{\fab}       {\mathrm{fab}}

\newcommand{\pr}        {\mathrm{pr}}
\newcommand{\id}        {\mathrm{id}}

\newcommand{\fgt}		{\mathrm{fgt}}

\newcommand{\aug}		{\mathrm{aug}}
\newcommand{\FinSet} 		{\mathrm{FinSet}}
\newcommand{\rex}		{\mathrm{rex}}
\newcommand{\fin}		{\mathrm{fin}}

\newcommand{\ltri}		{\triangleleft}

\newcommand{\1}{\mathbbm{1}}
\newcommand{\oplax}{\mathrm{oplax}}
\newcommand{\ev}{\mathrm{ev}}

\renewcommand{\ell}{\mathrm{ell}}

\newcommand{\quotientmap}{{quotient map}}

\newcommand{\weakcats}{naive global 2-rings}

\newcommand{\decatcoh}[3]{\mathbb{H}_{{#1}}(#2,#3)}
\newcommand{\decatcohsp}[3]{\textup{H}_{{#1}}(#2,#3)}
\newcommand{\genref}[2]{\Gamma_{#1}(#2)}
\newcommand{\pushfor}{\mathbb{P}}

\newcommand{\sX}{\mathcal{X}}
\newcommand{\sY}{\mathcal{Y}}

\newcommand{\TwoGlRingTpregen}[2]{\mathrm{2CAlg}_{#1\text{-}\gl}^{#2\textup{-pre}}}

\newcommand{\TwoGlRingTgen}[2]{\mathrm{2CAlg}_{#1\text{-}\gl}^{#2}}

\newcommand{\NaiveTwoGlRing}[1]{\mathrm{2CAlg}^{\mathrm{naive}}_{#1\text{-}\gl}}

\newcommand{\colim}  		{\operatornamewithlimits{colim}}
\newcommand{\cocolon}{\nobreak \mskip6mu plus1mu \mathpunct{}\nonscript\mkern-\thinmuskip {:}\mskip2mu \relax}
\newcommand{\Unco}[1]       {\mathrm{Un}^\mathrm{co}(#1)}
\newcommand{\Unct}[1]       {\mathrm{Un}^\mathrm{ct}(#1)}

\usetikzlibrary{calc}
\tikzset{curve/.style={settings={#1},to path={(\tikztostart)
        .. controls ($(\tikztostart)!\pv{pos}!(\tikztotarget)!\pv{height}!270:(\tikztotarget)$)
        and ($(\tikztostart)!1-\pv{pos}!(\tikztotarget)!\pv{height}!270:(\tikztotarget)$)
        .. (\tikztotarget)\tikztonodes}},
settings/.code={\tikzset{quiver/.cd,#1}
    \def\pv##1{\pgfkeysvalueof{/tikz/quiver/##1}}},
quiver/.cd,pos/.initial=0.35,height/.initial=0}

\newtheoremstyle{introthms}
{}{}{\itshape}{}{\bfseries }{}{ }
{\thmname{#1} \thmnumber{#2}. \thmnote{\bfseries{(#3)}}}

\theoremstyle{introthms}
\newtheorem{introthm}{Theorem}

\newtheorem{theorem}{Theorem}[section]
\newtheorem*{thm*}{Theorem}

\newtheorem{lemma}[theorem]{Lemma}
\newtheorem{proposition}[theorem]{Proposition}
\newtheorem{corollary}[theorem]{Corollary}

\theoremstyle{definition}
\newtheorem{remark}[theorem]{Remark}

\newtheorem{definition}[theorem]{Definition}
\newtheorem*{definition*}{Definition}
\newtheorem{example}[theorem]{Example}
\newtheorem{warning}[theorem]{Warning}
\newtheorem{non-example}[theorem]{Nonexample}
\newtheorem{construction}[theorem]{Construction}
\newtheorem{notation}[theorem]{Notation}
\numberwithin{equation}{theorem}
\newtheorem*{remark*}{Remark}
\newtheorem*{question*}{Question}
\theoremstyle{definition}

\usepackage{todonotes}

\usetikzlibrary{calc}
\usetikzlibrary{decorations.pathmorphing}

\tikzset{curve/.style={settings={#1},to path={(\tikztostart)
        .. controls ($(\tikztostart)!\pv{pos}!(\tikztotarget)!\pv{height}!270:(\tikztotarget)$)
        and ($(\tikztostart)!1-\pv{pos}!(\tikztotarget)!\pv{height}!270:(\tikztotarget)$)
        .. (\tikztotarget)\tikztonodes}},
settings/.code={\tikzset{quiver/.cd,#1}
    \def\pv##1{\pgfkeysvalueof{/tikz/quiver/##1}}},
quiver/.cd,pos/.initial=0.35,height/.initial=0}

\tikzset{tail reversed/.code={\pgfsetarrowsstart{tikzcd to}}}
\tikzset{2tail/.code={\pgfsetarrowsstart{Implies[reversed]}}}
\tikzset{2tail reversed/.code={\pgfsetarrowsstart{Implies}}}
\tikzset{no body/.style={/tikz/dash pattern=on 0 off 1mm}}

\title{Global 2-rings and genuine refinements}

\author{David Gepner}
\address{D.G.: Kreiger Hall, Johns Hopkins University, Baltimore, MD 21218, United States}
\email{gepner@jhu.edu}
 
\author{Sil Linskens}
\address{S.L.: Fakultät für Mathematik, Universität Regensburg, 93040 Regensburg, Germany}
\email{sil.linskens@mathematik.uni-regensburg.de }

\author{Luca Pol}
\address{L.P.: Max Plank Institute for Mathematics, Vivatsgasse 7, 53111 Bonn, Germany}
\email{pol@mpim-bonn.mpg.de}

\renewcommand\sslash{/\!\!/}

\begin{document}

\begin{abstract}
We introduce the notion of a naive global 2-ring: a functor from the opposite of the $\infty$-category of global spaces to presentably symmetric monoidal stable $\infty$-categories. By passing to global sections, every naive global 2-ring decategorifies to a multiplicative cohomology theory on global spaces, i.e.~a naive global ring. We suggest when a naive global 2-ring deserves to be called \emph{genuine}. As evidence, we associate to such a global 2-ring a family of equivariant cohomology theories which satisfy a version of the change of group axioms introduced in \cite{GKV95}. We further show that the decategorified multiplicative global cohomology theory associated to a genuine global $2$-ring canonically refines to an $\mathbb{E}_\infty$-ring object in global spectra. As we show, two interesting examples of genuine global 2-rings are given by quasi-coherent sheaves on the torsion points of an oriented spectral elliptic curve and Lurie's theory of tempered local systems. In particular, we obtain global spectra representing equivariant elliptic cohomology and tempered cohomology.
\end{abstract}


\maketitle
\setcounter{tocdepth}{1}

\tableofcontents

\section{Introduction}

The study of cohomology theories for equivariant spaces has long been a key aspect of algebraic topology. However it has, in recent years, become increasingly clear that many of the most interesting and important cohomology theories defined on equivariant spaces are more properly understood as cohomology theories on topological stacks. Often this perspective highlights the geometric or algebraic structure underlying such a cohomology theory in a way which the restriction to any group does not. As some examples, we note that Borel cohomology, complex K-theory, (stable) bordism and stable cohomotopy all admit the additional functoriality which characterizes a cohomology theory defined on topological stacks. More recent additions to the list are tempered cohomology and global elliptic cohomology, as defined by \cite{Ell3} and \cite{GM20} respectively. Such cohomology theories have come to be known as \emph{global cohomology theories}. Before we can make concrete definitions we first have to introduce a suitable homotopy theory of topological stacks.

\subsection*{The homotopy theory of topological stacks}	We follow \cite{GH} in considering the $\infty$-category $\Spc_{\gl}$ of \emph{global spaces} (there called \emph{orbispaces}). Global spaces are defined in analogy to (Bredon) $G$-spaces, the $\infty$-category of which, by the theorem of Elmendorf, is equivalent to the presheaf $\infty$-category on the $G$-orbit category. Similarly, global spaces are defined to be a presheaf $\infty$-category on the $\infty$-category $\Glo{}$, which one should interpret as an $\infty$-category of orbit stacks, with objects given by $\CB G$ for a compact Lie group $G$. Up to homotopy, morphisms in $\Glo{}$ from $\CB H\rightarrow \CB G$  are given by conjugacy classes of group homomorphisms $\alpha\colon H\rightarrow G$. A map $f\colon \CB H\rightarrow \CB G$ is called \emph{faithful} if it is represented by a monomorphism of groups. The $\infty$-categories of $G$-spaces and global spaces are closely related: There is a fully faithful colimit preserving functor 
\[
-\sslash G\colon \Spc_G \to{\Spc_{\gl}}_{/\CB G}
\]
whose essential image is given by those global spaces over $\CB G$ whose reference map is faithful. 

It is often useful to only consider global spaces with isotropy restricted to some family $\CE$ of compact Lie groups, which one denotes $\Spcgl{\CE}$ and calls $\CE$-global spaces. For example tempered cohomology and global elliptic cohomology are most naturally defined as a cohomology theory for global spaces with isotropy in finite abelian groups and compact abelian Lie groups respectively. 

\subsection*{Multiplicative global cohomology theories}	
Having introduced our homotopy theory of stacks, one can simply define a \emph{multiplicative $\CE$-global cohomology theory} as a limit preserving functor
\[
E\colon \Spcgl{\CE}^{\op}\rightarrow \CAlg.
\]
Taking homotopy groups we obtain a functor $E^\ast$ taking values in graded commutative rings and satisfying analogues of the Eilenberg-Steenrod axioms. By higher Brown representability, the $\infty$-category of cohomology theories are equivalent to \emph{(commutative) naive global rings}, i.e., commutative algebra objects in the $\infty$-category of spectrum objects in $\CE$-global spaces.
However, as is typically the case in equivariant homotopy theory, we are more interested in ``genuine'' multiplicative global cohomology theories which are represented by\emph{(commutative) global rings}, i.e., commutative algebra objects in the $\infty$-category of \emph{$\CE$-global spectra} in the sense~\cite{Schwede18}. Global spectra can be organized into a symmetric monoidal stable $\infty$-category $\Spgl{\CE}$ which admits a suspension spectrum functor $\Sigma_+^\infty\colon \Spcgl{\CE}\rightarrow \Spgl{\CE}$. In particular, any commutative global ring $X \in \CAlg(\Spgl{\CE})$ defines a multiplicative global cohomology theory (and hence a commutative naive global ring) via the assignment
\[
\Spcgl{\CE}^{\op} \to \CAlg, \qquad  Y\mapsto \map_{\Spgl{\CE}}(\Sigma_+^\infty Y, X)
\]
However not all multiplicative global cohomology theories $E$ arise in this way. In fact a global ring contains significantly more structure. For this reason we say that a global ring $X$ is a \emph{genuine refinement} of the multiplicative global cohomology theory $E$ if there exists a natural equivalence  
\[
E(-)\simeq \map_{\Spgl{\CE}}(\Sigma_+^\infty(-), X).
\]

\begin{remark*}\label{rem:extra_str_gen_ref}
    As noted before, a genuine refinement endows $E$ with significantly more structure. For example for any compact Lie group $G$, and any $G$-space $X$ one obtains an $RO(G)$-graded abelian group $E^\star(X)$ together with dimension shifting transfer maps
    \[
    \mathrm{tr}_H^G\colon E_H^{\star+L}(X)\rightarrow E_G^{\star}(X),
    \] 
    for any subgroup $H\subset G$, where $L$ is the tangent representation of $G/H$. Moreover this structure is highly compatible as you change the group $G$, the subgroup $H$, and the space $X$.
\end{remark*}

Most of the cohomology theories recalled above are all canonically ``genuine". For example, global spectra representing Borel cohomology, K-theory, (stable) bordism and cohomotopy were constructed  by Schwede in \cite{Schwede18}. However neither tempered cohomology nor elliptic cohomology has so far been given a genuine refinement, despite the fact that considerable evidence for such a refinement is contained in \cite{Ell3} and \cites{Survey, GM20} for tempered and elliptic cohomology respectively.

\subsection*{Naive global $2$-rings}	
Therefore there is a use for general procedures which construct genuine refinements of multiplicative global cohomology theories. As such, the main concern of this paper is:
\begin{question*}
    How can one construct genuine refinements of a multiplicative global cohomology theory?
\end{question*}
We provide an answer via the process of categorification. 
One key definition of this paper is:
\begin{definition*}
    Let $\CE$ be a family of compact Lie groups. A \emph{naive ($\CE$-)global 2-ring} is a functor
    \[
    \CR \colon \Spcgl{\CE}^{\op}\rightarrow \PrL, \quad \sX \mapsto \CR_\sX, \; f \mapsto f^*
    \] from global spaces to the $\infty$-category of presentably symmetric monoidal stable $\infty$-categories, which satisfies two additional conditions that we will explain throughout this introduction. We write $f_*$ for the right adjoint to the functor $f^*$.
\end{definition*}

Given a functor $\CR$ as above, we can produce a functor $\decatcohsp{\gl}{-}{\CR}\colon \Spcgl{\CE}^{\op}\rightarrow \CAlg$ by taking endomorphism rings of the various unit objects $\1_{\CR_X}$.  The first condition which makes $\CR$ a naive $\CE$-global 2-ring is that $\decatcohsp{\gl}{-}{\CR}$ is a multiplicative global cohomology theory. Thus, we think of $\CR$ as a \emph{categorification} of $\decatcohsp{\gl}{-}{\CR}$. We summarise the situation in the following diagram:
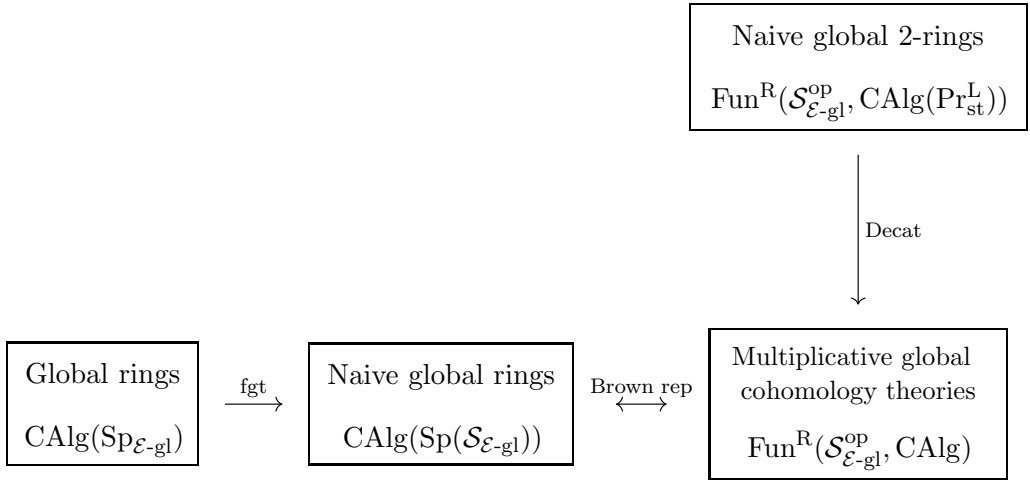
\[\begin{tikzcd}[column sep = 0.4cm]
	&& {} &&& \fbox{\parbox{\widthof{\text{sssNaive global 2-ringssss}}}{\centering\vspace{.4em}\text{Naive global 2-rings}\vspace{1em} 
$\subset \Fun(\Spcgl{\CE}^{\op},\PrL)$}\vspace{.4em}} \\
	\\
	& {} \\
	\fbox{\parbox{\widthof{ Global rings }}{\centering\vspace{.4em}Global rings\vspace{1em} 
$\CAlg(\Spgl{\CE})$}\vspace{.4em}} && \fbox{\parbox{\widthof{ Naive global rings }}{\centering\vspace{.4em}Naive global rings\vspace{1em} 
$\mathrm{CAlg}(\mathrm{Sp}(\mathcal{S}_{\mathcal{E}\text{-}\mathrm{gl}}))$}\vspace{.4em}} &&& \fbox{\parbox{\widthof{\text{ Multiplicative global }}}{\centering\vspace{.3em}\text{\small Multiplicative global}\newline{\small cohomology theories}\vspace{.8em} 
$\Fun^{\rR}(\Spcgl{\CE}^{\op},\CAlg)$}\vspace{.3em}}  \\
	\arrow["{\text{Decat}}", shorten <=5pt, shorten >=5pt, from=1-6, to=4-6]
	\arrow["{\text{fgt}}", shorten <=5pt, shorten >=5pt,  from=4-1, to=4-3]
	\arrow["{\text{Brown rep}}", shorten <=11pt, shorten >=11pt,  tail reversed, from=4-3, to=4-6]
\end{tikzcd}\]
Our interest in considering this categorification of a multiplicative global cohomology theory comes from the fact that it endows its decategorification with significantly more structure, as we explain now.

\subsection*{Unraveling into families of equivariant cohomology theories}	
Recall that in the literature on equivariant elliptic cohomology theories associated to an elliptic curve $\mathsf{E}$, one rarely views equivariant elliptic cohomology as a global cohomology theory valued in spectra. Instead one restricts to a cohomology theory on $G$-spaces for some $G$. This has the benefit that $G$-equivariant elliptic cohomology then canonically lifts to a functor 
\[
\mathcal{E}ll_G(-)\colon \Spc_G^{\op}\rightarrow \QCoh(\mathsf{E}[\hat{G}])
\] 
valued in the stable $\infty$-category of quasi-coherent sheaves on the $\hat{G}$-torsion points of the elliptic curve (here $\hat{G}$ denotes the Pontryagin dual of the abelian compact Lie group $G$). However apriori there is a downside to this perspective: it fails to capture the global nature of equivariant elliptic cohomology. 

There is a way however to obtain the best of both worlds. One can capture the global nature of equivariant elliptic cohomology by equipping the family of functors $\mathcal{E}ll_G(-)$ with a suitable collection of \emph{change of group transformations}. Given a group homomorphism $\alpha\colon H\rightarrow G$ and a $G$-space $X$, these relate the value of $\mathcal{E}ll_H$ on $\alpha^* X$, the restriction of $X$ to an $H$-space, and the value of $\mathcal{E}ll_G$ on $X$. One can find this structure emphasised in the approach to equivariant elliptic cohomology suggested by \cite{GKV95}. As the first step of our approach, we show that any naive global $2$-ring canonically induces this data in an extremely coherent way. The next result combines \Cref{thm-E} with the results of \Cref{sec:equivariantcoh}.

\begin{introthm}\label{thm-intro1}
    Let $\CR \colon \Spcgl{\CE}^{\op} \to \PrL$ be a functor. Then for all $G\in \CE$, there exists a lax symmetric monoidal functor
    \[
    \decatcoh{G}{-}{\CR}\colon \Spc_G^{\op}\rightarrow \CR_{\CB G}
    \] 
    whose composite with the functor $\CR_{\CB G}\to \Sp$, $X \mapsto\map_{\CR_{\CB G}}(\1,X)$ agrees with $\decatcohsp{\gl}{-\sslash G}{\CR}$. Moreover, for every group homomorphism $\alpha \colon  H \to G$ there exists a natural transformation filling the square
    \begin{equation}
        \begin{tikzcd}
            {\Spc_G^{\op}} &  {\Spc_H^{\op}}\\
            {\CR_{\CB G}} & {\CR_{\CB H}.}
            \arrow["{\alpha^*}"', from=2-1, to=2-2]
            \arrow["{\decatcoh{G}{-}{\,\CR}}"', from=1-1, to=2-1]
            \arrow["{\decatcoh{H}{-}{\,\CR}}", from=1-2, to=2-2]
            \arrow["\alpha^*", from=1-1, to=1-2]
            \arrow[shorten <=10pt, shorten >=6pt, Rightarrow, from=2-1, to=1-2, "Q_{\alpha}"']
        \end{tikzcd}
    \end{equation}
    Furthermore this data is coherently functorial: This is encoded in the existence of a functor
    \[
    \decatcoh{\bullet}{-}{\CR}\colon \Glo{\CE}^{\op}\rightarrow \Fun^{\mathrm{oplax}}([1],\tCat^{\otimes,\lax})
    \] extending the assignment above.
\end{introthm}

In fact, the second condition of being a naive global 2-ring is precisely that the functors $\decatcoh{G}{-}{\CR}\colon \Spc_G^{\op}\rightarrow \CR_{\CB G}$ are all limit preserving. In this way we see that naive global 2-rings also categorify families of equivariant cohomology theories.

\subsection*{The Ginzburg--Kapranov--Vasserot axioms}

The previous result is only useful to the extent to which we are able to control the family of cohomology theories one obtains by unraveling. To isolate a case where this is possible we introduce the notion of a \emph{genuine} global 2-ring. Let $\CT\subset\CE$ be some subset of groups in the family $\CE$.

\begin{definition*}
We say that a naive global 2-ring $\CR \colon \Spcgl{\CE}^{\op} \to \PrL$ is \emph{$\CT$-genuine} if 
\begin{itemize}
    \item[(1)] $\CR$ satisfies base-change with respect to faithful morphisms $\CB H \rightarrow \CB G$ with target in $\CT$.
    \item[(2)] For any faithful morphism $f \in \Glo{\CE}$, the adjunction $(f^*,f_*)$ satisfies the projection formula. 
    \item[(3)] For all $G\in \CE$ and every irreducible $G$-representation $V$, the object $\decatcoh{G}{S^V}{\CR}\in \CR_{\CB G}$ is invertible.
\end{itemize}
\end{definition*}

In the previous definition $\CT\subset \CE$ will typically be some subset of \emph{enough injective objects}, a concept which abstracts the properties of the collection of tori inside all abelian compact Lie groups, see \Cref{def:enough-injectives}. This concept ensures that one can effectively control the values at all groups in $\CE$ by only remembering the value of $G$-equivariant cohomology theories for $G\in \CT$. The additional flexibility of specifying such a $\CT$ is crucial for applications. Specifically, the example of equivariant elliptic cohomology will only be genuine with respect to the tori.

As mentioned, the axioms of a genuine global 2-ring ensure that the unraveling $\decatcoh{\bullet}{-}{\CR}$ of a genuine global 2-ring is particularly well-behaved. In short, points (1) and (2) of the definition above ensure that the family $\decatcoh{\bullet}{-}{\CR}$ of equivariant cohomology theories transforms correctly as we vary the group. We then combine this with point (3) to show that we can extend $\decatcoh{G}{-}{\CR}$ to a family of $\mathrm{RO}(G)$-graded theories. More precisely, we show using (1) and (2) that $\decatcoh{\bullet}{-}{\CR}$ satisfies the following analogs of the axiom of \cite{GKV95}. 

\begin{introthm}\label{introthm:GKV-axioms}
Let $\CR$ be a $\CT$-genuine global 2-ring. Then the unraveling $\decatcoh{\bullet}{-}{\CR}$ satisfies the following axioms:
\begin{enumerate}[itemsep = 5pt]
    \item \emph{Induction:} Let $\alpha\colon G\rightarrow G/N$ be a surjective group homomorphism and let $X$ be a $G$-space such that the action of $N$ on $X$ is free. Then there is a natural equivalence
    \[\decatcoh{G/N}{\alpha_!(X/N)}{\CR}\xrightarrow{\sim} \alpha_*\decatcoh{G}{X}{\CR};\]
    \item \emph{Base-change:} Let $\CB \alpha \colon \CB H \to \CB G$ be a map in $\Glo{\CE}$ such that $\CB G\in \CT$. Then the natural transformation
    \[
    Q_\alpha\colon \alpha^* \decatcoh{G}{X}{\CR}\to\decatcoh{H}{\alpha^*X}{\CR}
    \] from \Cref{thm-intro1} is an equivalence for all compact $G$-spaces $X$;
    \item \emph{K\"unneth:} Let $G$ and $H$ be two groups in $\CT$, $X$ a compact $G$-space and $Y$ a compact $H$-space. Then there is an equivalence 
    \[
    \pi_G^*\decatcoh{G}{X}{\CR}\otimes \pi_H^*\decatcoh{H}{Y}{\CR} \simeq \decatcoh{G\times H}{X \times Y}{\CR},
    \]
    where $\pi_H$ and $\pi_G$ denote the two projections $G\times H\rightarrow H,G$. 
\end{enumerate}
\end{introthm}

\subsection*{A genuine refinement of $\decatcohsp{\gl}{-}{\CR}$}

Using the previous result and the universal property of $G$-spectra, we are then able to coherently refine the equivariant cohomology theories $\decatcoh{G}{-}{\CR}$ to cohomology theories on genuine $G$-spectra.
Taking global sections we obtain a family of equivariant cohomology theories with values in $\Sp$. An application of higher Brown representability yields a compatible collection of $G$-spectra, which in turn define a global spectrum by the main theorem of \cite{LNP}. Finally we check that this global spectrum is a genuine refinement of $\decatcohsp{\gl}{-}{\CR}\colon \Spcgl{\CE}^{\op}\rightarrow \Sp$. This gives the main theorem of this paper, see \Cref{thm:Represented-theorem} and \Cref{prop:funct_gen_ref}..

\begin{introthm}\label{introthm:gen_ref}
Let $\CE$ be a multiplicative global family of compact Lie groups and let $\CR\colon \Spcgl{\CE}^{\op} \to \PrL$ be a \emph{genuine} global 2-ring. Then its associated decategorification $\decatcohsp{\gl}{-}{\CR}\colon \Spcgl{\CE}^{\op} \to \CAlg$ admits a canonical genuine refinement $\genref{\gl}{\CR} \in\CAlg(\Spgl{\CE})$. Moreover, this assignment improves to a functor from a suitable category of genuine global 2-rings to global rings.
\end{introthm}	

One benefit of categorification which is often highlighted is that it has the ability to turn structure into a property. We note that another example of this phenomenon is provided by naive global 2-rings: while a genuine refinement of a multiplicative global cohomology theory is structure, it is a property for a naive global 2-ring to be genuine. We can therefore complete our diagram as follows:
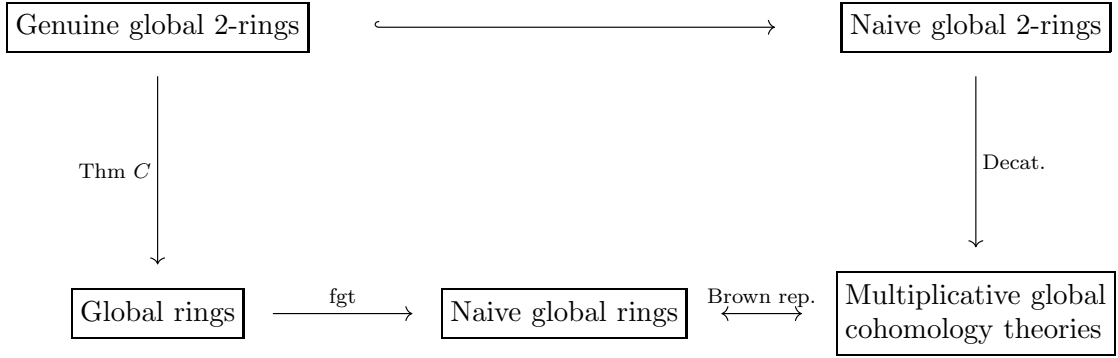
\[
\begin{tikzcd}[column sep = 0.55cm]
	{\fbox{\parbox{\widthof{\text{Genuine global 2-rings}}}{\vspace{.1em}\text{Genuine global 2-rings}\vspace{.1em}}}} &&&&& 
 {\fbox{\parbox{\widthof{\text{Naive global 2-rings}}}{\vspace{.1em}\text{Naive global 2-rings}\vspace{.1em}}}} \\
	&&&&& {} \\
	&&&&& {} \\
	{\fbox{\parbox{\widthof{\text{Global rings}}}{\vspace{.1em}\text{Global rings}\vspace{.1em}}}} & {} & {\fbox{\parbox{\widthof{\text{Naive global rings}}}{\vspace{.1em}\text{Naive global rings}\vspace{.1em}}}} &&& {\fbox{\parbox{\widthof{\text{Multiplicative global}}}{\vspace{.1em}\text{Multiplicative global}\newline\text{cohomology theories}\vspace{.1em}}}}
	\arrow[shorten <=20pt, shorten >=20pt, hook, from=1-1, to=1-6]
	\arrow["{\text{Thm } \ref{introthm:gen_ref}}"', shorten <=5pt, shorten >=5pt, from=1-1, to=4-1]
	\arrow["{\text{Decat.}}", shorten <=5pt, shorten >=5pt, from=1-6, to=4-6]
	\arrow["{\mathrm{fgt}}", shorten <=6pt, shorten >=6pt, from=4-1, to=4-3]
	\arrow["{\text{Brown rep.}}", shorten <=8pt, shorten >=8pt, tail reversed, from=4-3, to=4-6]
\end{tikzcd}
\]
In the body of the paper we improve this picture to an actual commutative diagram of categories and functors. 

\subsection*{Applications: elliptic and tempered cohomology}
Our first application is to equivariant elliptic cohomology. By work of \cite{GM20}, a preoriented strict abelian group object $\bG$ in a suitable $\infty$-category $\mathcal{X}$ canonically induces a functor $\Glo{\ab}\rightarrow \mathcal{X}$ from the global orbit $\infty$-category with isotropy in all compact abelian Lie groups. If the objects of $\mathcal{X}$ admit a sufficiently well-behaved notion of quasi-coherent sheaves one can obtain a naive global $2$-ring by post-composition. Applying this to an oriented elliptic curve $\bE$ in spectral Deligne-Mumford stacks over $\sfS$ we obtain a naive global 2-ring
\[
\CQ^\bE\colon \Spcgl{\ab}^{\op}\rightarrow \PrL,\quad X\mapsto \CQ^\bE_X.
\]
Applying $\CQ^\bE$ to $\CB A$ for a compact abelian group $A$ gives $\QCoh(\bE[\widehat{A}])$, the $\infty$-category of quasi-coherent sheaves on the $\widehat{A}$-torsion points of $\bE$, where $\widehat{A}$ is the Pontryagin dual of $A$. Heavily relying on \cite{GM20}, we prove:
\begin{introthm}
    Let $\bE$ be an oriented elliptic curve in spectral Deligne-Mumford stacks over $\sfS$. Then $\CQ^\bE$ is a genuine global $2$-ring and so $\decatcohsp{\gl}{-}{\CQ^\bE}$ admits a genuine refinement $\Gamma_{\bE}(\sfS,\CO_{\sfS})\coloneqq \genref{\gl}{\CQ^\bE} \in \CAlg(\Spgl{\ab})$.
\end{introthm}

By construction the underlying spectrum of $\Gamma_{\bE}(\sfS,\CO_{\sfS})$ is simply $\Gamma(\sfS,\CO_{\sfS})$, the global sections of $\sfS$. In particular we obtain an integral globally equivariant elliptic cohomology spectrum. Further specializing, when $\sfS$ is $\CM_{\mathrm{ell}}^{\mathrm{or}}$, the moduli stack of oriented elliptic curves, we obtain a global spectrum of topological modular forms $\mathrm{TMF}_{\gl}$.

Our second application is to tempered cohomology \cite{Ell3}. In this case we restrict to the family of \emph{finite abelian} groups and write $\Glo{\fab}$ and $\Spcgl{\fab}$ for the associated global orbit $\infty$-category and $\infty$-category of global spaces. Given an oriented $\rP$-divisible group $\bG$ over a commutative ring spectrum $R$, Lurie constructs a multiplicative global cohomology theory 
\[
R^\bullet_\bG\colon \Spcgl{\fab}^{\op}\rightarrow \CAlg,
\] 
referred to as \emph{tempered cohomology}. Furthermore Lurie in \cite{Ell3} introduces the notion of tempered $\bG$-local systems on a $\fab$-global space. As shown there, this functions as an extremely well-behaved categorification of tempered cohomology. We capture some of this richness by the following theorem.

\begin{introthm}
    Suppose $\bG$ is an oriented $\rP$-divisible group over $R$. Then the functor
    \[\LocSys_{\bG}(-)\colon \Spcgl{\fab}^{\op}\rightarrow \PrL,\quad X\mapsto \LocSys_{\bG}(X)\] is a genuine global $2$-ring and so admits a genuine refinement $\ul{R}_{\bG}^{\gl}\coloneqq \genref{\gl}{\LocSys_{\bG}}\in\CAlg(\Spgl{\fab})$.
\end{introthm}

In particular, restricting to any finite abelian group $A$, one obtains a genuine $A$-spectrum $\ul{R}_{\bG}^{A}$. Using the technology developed in this article, we can in fact go further and actually identify $\LocSys_{\bG}(\CB A)$ with the $\infty$-category of $\ul{R}_{\bG}^{A}$-modules in $A$-spectra. The following theorem combines \Cref{thm-localsym-as-spectra} and \Cref{prop:temp_eq_natural}.

\begin{introthm}
There exists a natural equivalence 
\[
\bGamma_\bullet\colon \LocSys_{\bG}(\bullet)\rightarrow \Mod_{\ul{R}_{\bG}^{\bullet}}(\Sp_\bullet)
\]
of functors $\Glo{\fab}^{\op}\rightarrow \PrL$.
\end{introthm}

\subsection{Conventions and notation}\label{convention}

\begin{enumerate}
    \item We use capital letters to refer to large categories, and bold font to refer to $(\infty,2)$-categories.
    \item We use $\Cat$ for the $\infty$-category of $\infty$-categories and $\Cat_2$ for the $\infty$-category of $(\infty,2)$-categories.
    \item Given a symmetric monoidal $\infty$-category $\CC$, we write $\CAlg(\CC)$ for the $\infty$-category of commutative algebras ($\mathbb{E}_\infty$-algebras) in $\CC$. When $\CC$ is the $\infty$-category of spectra we abbreviate this by $\CAlg$. As another key example we obtain $\PrL$, the $\infty$-category of presentably symmetric monoidal stable $\infty$-categories and symmetric monoidal left adjoint functors between them.
    \item Given an $\infty$-category $\CC$ we write $\Map_{\CC}(-,-)$ for the mapping spaces in $\CC$. If $\CC$ is stable then we write $\map_{\CC}(-,-)$ for the mapping spectra in $\CC$. Finally if $\CC$ is closed symmetric monoidal, then we write $\iHom_{\CC}(-,-)$ for the internal Hom in $\CC$.
    \item We write $\Delta(\CC)\colon \CI\rightarrow \Cat$ for the constant functor on an $\infty$-category $\CC$.
    \item We will assume all (nonconnective) spectral Deligne–Mumford stacks to be locally noetherian,
i.e. they are \'{e}tale locally of the form $\Spec(A)$ with $\pi_0(A)$ noetherian and $\pi_i(A)$ a finitely generated $\pi_0(A)$-module for $i\geq 0$. Moreover, we assume (nonconnective) spectral Deligne–Mumford stacks to be quasi-separated, i.e. the fiber product of any two affines over such a stack
is quasi-compact again, and that all iterated diagonals are quasi-separated (i.e. $n$-quasiseparated in the sense of \cite{GM20}*{Definition D.1} for all $n \geq 1$).
\end{enumerate}

\subsection*{Acknowledgements} 
The second author is an associate member of the Hausdorff Center for Mathematics at the University of Bonn, supported by the DFG Schwerpunktprogramm 1786 ``Homotopy Theory and Algebraic Geometry'' (project ID SCHW 860/1-1). The third author was supported by the SFB 1085 Higher Invariants in Regensburg. 

The authors would like to thank the Hausdorff Research Institute for Mathematics for the hospitality in the context of the Trimester program ''Spectral Methods in Algebra, Geometry, and Topology'', funded by the Deutsche Forschungsgemeinschaft (DFG, German Research Foundation) under Germany’s Excellence Strategy – EXC-2047/1 – 390685813.

The authors would like to thank Robert Burklund, Tim Campion, Bastiaan Cnossen, Jack Davies, Tobias Lenz, Lennart Meier and Stefan Schwede for helpful discussions. We would also like to especially thank Markus Hausmann for various helpful conversations and for suggesting the counterexample in \Cref{rem-uniqueness}, and Denis Nardin for his contributions to this project at an early stage of the collaboration. Finally, we thank Fabio Neugebauer for pointing out a mistake in the previous construction of the naive global $2$-ring associated to elliptic cohomology. 

\part{Recollections on global homotopy theory}\label{part:global_homotopy_theory}

In this part we provide some definitions and results in global homotopy theory. We emphasize the perspective on global homotopy provided by the results of \cite{LNP}, which utilizes the notion of a partially lax limit. Therefore we begin with a section on lax natural transformations and provide a definition of partially lax limits. In the second section, we define the $\infty$-category of global spaces and explain its close relationship to equivariant homotopy theory. In the third section we discuss the $\infty$-category of global spectra.

\section{Lax natural transformations}\label{sec:recall_lax}

In this section we recall the $\infty$-categories of functors and (op)lax natural transformations, discuss (op)lax slice categories and the mate equivalence.

\begin{definition}
    Let $\CI$ be an $\infty$-category and $\tC$ an $(\infty,2)$-category. We define the $(\infty,2)$-categories 
    \[
    \mathbf{Fun}^{\lax}(\CI,\tC)\quad \mathrm{and}\quad\mathbf{Fun}^{\oplax}(\CI,\tC)
    \]
    as the right adjoint objects for the left and right Gray tensor product by $\CI$ respectively, in the sense of \cite{GHL21}, so that 
    \begin{align*}
        \Map_{\Cat_2}(\mathbf{A} \boxtimes \CI, \tC)&\simeq \Map_{\Cat_2}(\mathbf{A}, \mathbf{Fun}^{\lax}(\CI,\tC)),\\
        \Map_{\Cat_2}(\CI\boxtimes \mathbf{A}  , \tC)&\simeq \Map_{\Cat_2}(\mathbf{A}, \mathbf{Fun}^{\oplax}(\CI,\tC)).
    \end{align*} 
    for every $(\infty,2)$-category $\mathbf{A}$. The underlying $\infty$-categories will be denoted $\Fun^{\lax}(\CI, \tC)$ and $\Fun^{\oplax}(\CI,\tC)$ respectively.
\end{definition}

\begin{example}
    The objects of $\Fun^{\lax}(\CI,\tC)$ agree with the objects of $\Fun(\CI,\tC)$. A morphism in $\Fun^{\lax}(\CI,\tC)$ is precisely the data of a lax natural transformation. Informally, a lax natural transformation $\eta \colon F \Rightarrow G$ between functors $F,G\colon \CI \to \tC$ is a coherent collection of squares
    \[
    \begin{tikzcd}
        {F(i)} &  {F(j)} \\
        {G(i)} & {G(j)}
        \arrow["\eta_j",from=1-2, to=2-2]
        \arrow["F(f)", from=1-1, to=1-2]
        \arrow["G(f)"', from=2-1, to=2-2]
        \arrow["\eta_i"', from=1-1, to=2-1]
        \arrow["\epsilon", shorten <=7pt, shorten >=7pt, Rightarrow, from=2-1, to=1-2]
    \end{tikzcd}
    \]
    together with 2-morphisms $\epsilon\colon G(f) \circ \eta_i \Rightarrow \eta_j\circ  F(f) $. If the 2-morphisms $\epsilon$ are all equivalences, then $\eta$ defines a natural transformation. Dually, an oplax natural transformation $\eta \colon F \Rightarrow G$ is informally a coherent collection of oplaxly commuting squares
    \[
    \begin{tikzcd}
        {F(i)} & {F(j)} \\
        {G(i)} & {G(j).}
        \arrow["\eta_j",from=1-2, to=2-2]
        \arrow["F(f)", from=1-1, to=1-2]
        \arrow["G(f)"', from=2-1, to=2-2]
        \arrow["\eta_i"', from=1-1, to=2-1]
        \arrow["\epsilon", shorten <=7pt, shorten >=7pt, Rightarrow, from=1-2, to=2-1]
    \end{tikzcd}
    \]
\end{example}

\begin{definition}
Consider two functors $F,G\colon \CI\rightarrow \Cat$. We write $\Nat^{\lax}(F,G)$ for the mapping categories in $\tFun^{\lax}(\CI,\tCat)$ from $F$ to $G$. Dually we write $\Nat^{\oplax}(F,G)$ for the mapping $\infty$-category in $\tFun^{\oplax}(\CI,\tCat)$.
\end{definition}

We now give a definition of partially (op)lax limits, first introduced in \cite{Berman}.

\begin{definition}
Let $F\colon \CI\rightarrow \Cat$ be a functor. We define the \emph{lax limit} of $F$
\[
\laxlim_{\CI} F \coloneqq \Nat^{\lax}(\Delta(\ast),F)
\] to be the $\infty$-category of \emph{lax cones}: lax natural transformations from the constant functor on the terminal $\infty$-category to $F$. Here $\Delta\colon\Cat\to\tFun^{\lax}(\CI,\Cat)$ denotes the diagonal, which under the identification $\Cat\simeq\tFun^{\lax}(\ast,\Cat)$ is precomposition with the projection $\CI\to\ast$.

Suppose $\CI$ is marked by a subcategory $\CW$. Then we define the \emph{partially lax limit} of $F$
\[
\laxlimdag_{\CI,\CW} F \subset \laxlim_{\CI} F 
\]
to be the full subcategory of $\laxlim F$ spanned by those lax cones whose restriction to $\CW\subset \CI$ is strictly natural.

Dually we define the \emph{oplax limit} $\oplaxlim F$ of $F$ to be the $\infty$-category $\Nat^{\oplax}(\Delta(\ast),F)$ of \emph{oplax cones} and $\oplaxlimdag F\subset \oplaxlim F$ to be the full subcategory spanned by those oplax cones whose restriction to $\CW$ is strictly natural.
\end{definition}

We now introduce an (un)straightening equivalence for the $\infty$-category of functors and lax natural transformations. This will be a crucial tool for working with partially lax limits, and will also allow us to connect the definition given above to the definitions of \cite{Berman} and \cite{LNP}.

\begin{definition}
Write $\mathbf{Cocart}^{\lax}(\CI)$ and $\mathbf{Cart}^{\oplax}(\CI)$ for the full subcategories of the $(\infty,2)$-category $\tCat_{/\CI}$ spanned by the cocartesian and cartesian fibrations respectively.
\end{definition}

\begin{theorem}[\cite{HHLN1}*{Theorem E}]
The (un)straightening equivalences of Lurie extend to natural equivalences
\begin{equation}\label{eq:lax_unstraightening}
\mathbf{Cocart}^{\lax}(\CI) \simeq   \mathbf{Fun}^{\lax}(\CI,\tCat) \quad \text{and} \quad 
\mathbf{Cart}^{\oplax}(\CI) \simeq \mathbf{Fun}^{\oplax}(\CI^{\op},\tCat). \qednow
\end{equation}
\end{theorem}

\begin{remark}\label{rem:laxlim_sect}
The previous equivalence gives an equivalence 
\[
\Nat^{\lax}(F,G)\simeq \Fun_{\CI}(\Unco{F},\Unco{G}).
\]
when restricted to the mapping $\infty$-category from $F$ to $G$. In particular, because the identity functor on $\CI$ is the unstraightening of $\Delta(\ast)$, we obtain an equivalence 
\[
\laxlim F \simeq \Fun_{\CI}(\CI,\Unco{F}).
\]
That is, the lax limit of $F$ is equivalent to the $\infty$-category of sections of the cocartesian unstraightening of $F$. Furthermore, if $\CI$ is marked by $\CW$, then a lax cone is in $\laxlimdag F$ if and only if the associated section $\CI\rightarrow \Unco{F}$ sends maps in $\CW$ to cocartesian edges in $\Unco{F}$. This follows immediately from the fact that \Cref{eq:lax_unstraightening} restricts to the usual (un)straightening equivalence. 

We conclude by \cite{Berman}*{Theorem 4.4} that our definition agrees with the definition of partially lax limits given there, and used in \cite{LNP}. Finally we note that all of the observations above dualize and imply analogous statements for oplax limits.
\end{remark}

\begin{example}\label{ex:functors_are_lax_cones}
A particularly degenerate consequence of the previous observations is already interesting. Recall that the cocartesian unstraightening of the constant functor $\Delta(\CC)\colon \CI\rightarrow \Cat$ is given by the projection $\CC\times \CI\rightarrow \CI$. Therefore we may compute that 
\[
\laxlim \Delta(\CC) \simeq \Fun_{\CI}(\CI,\CI\times \CC) \simeq \Fun(\CI,\CC).
\]
We conclude that functors $F\colon \CI\rightarrow \CC$ are equivalent to lax cones over the constant functor on $\CC$. Dually functors $F\colon \CI^{\op}\rightarrow \CC$ are equivalent to oplax cones over the constant functor on $\CC$.
\end{example}

While not immediately clear from the description above, we note that if $F\colon \CI\rightarrow \Cat$ lifts to a diagram in $\Cat^{\otimes}$, the $\infty$-category of symmetric monoidal $\infty$-categories and symmetric monoidal functors, then the partially lax limit $\laxlimdag F$ is canonically symmetric monoidal.

\begin{proposition}\label{calg-laxlim}
Let $F\colon \CI\rightarrow \Cat^{\otimes}$ be a functor. Then $\laxlimdag F$ admits a symmetric monoidal structure such that for any symmetric monoidal $\infty$-category $\CC$ we have a natural equivalence 
\[
\Fun^{\otimes\mathrm{-\lax}}(\CC,\laxlimdag F) \simeq \laxlimdag \Fun^{\otimes\mathrm{-\lax}}(\CC,F(i)),
\]
between lax symmetric monoidal functors.
In particular,
\[
\CAlg(\laxlimdag F) \simeq \laxlimdag \CAlg(F(i)).
\]
\end{proposition}

\begin{proof}
This follows from \cite{LNP}*{Remark 5.1}.
\end{proof}
In the remainder of this section we recall some technical material which we will use in later sections.

\subsection{Lax slice categories}
In this subsection we introduce the definition of (op)lax slice categories, and use this to make some identifications of lax limits which will be useful later.

\begin{definition}
    Given an $(\infty,2)$-category $\tC$ with underlying $\infty$-category $\CC$ and an object $X$ in $\tC$, we define $\tC \oplaxslice X $, the \textit{oplax slice category over $X$}, via the pullback
    \[\begin{tikzcd}
        {{\tC \oplaxslice X}} & {\Fun^{\oplax}([1],\tC)} \\
        {\CC\times \{X\}} & {\CC\times \CC}.
        \arrow[from=1-2, to=2-2]
        \arrow[from=2-1, to=2-2]
        \arrow[from=1-1, to=2-1]
        \arrow[from=1-1, to=1-2]
    \end{tikzcd}\]
    Analogously we can define the notion of lax slice category $\bC \laxslice X$.
\end{definition}

Our interest in oplax slice categories comes from the following theorem. 

\begin{theorem}\label{cor-lax-slice}
    Let $\bC$ be an $(\infty,2)$-category with underlying $\infty$-category $\CC$, and fix $X\in \bC$. 
    \begin{enumerate}
        \item The forgetful functor $\fgt\colon \bC \oplaxslice X\rightarrow \CC$ is a cartesian fibration and classifies the mapping functor $\Map_{\bC}(-,X)\colon \CC^{\op}\rightarrow \Cat$.
        \item The forgetful functor $\fgt \colon (\bC \laxslice X)^{\op} \to \CC^{\op}$ is a cocartesian fibration and classifies the same functor as in (a).
    \end{enumerate}
\end{theorem}
\begin{proof}
    Part (1) is proved in~\cite{HHLN2}*{Theorem 7.21}. Part (2) follows by applying (1) to $\bC^{2\text{-op}}$, the $(\infty,2)$-category obtained from $\bC$ by passing to opposites on hom-categories.
\end{proof}

Applying this to $\bC = \tCat$ we obtain another formulation of the (un)straightening equivalence. We write $\Cart$ for the subcategory of $\Ar(\Cat)$ spanned by cartesian fibrations and maps of cartesian fibrations, see \Cref{def:Cart}.

\begin{corollary}\label{slice-cart}
    There exists an equivalence 
    \[
    \mathrm{Un} \colon \tCat\oplaxslice \Cat \xrightarrow{\sim} \Cart.
    \]
\end{corollary}

\begin{proof}
The equivalence is induced by unstraightening the natural equivalence 
\[
\Fun(-,\Cat) \simeq \Cart((-)^{\op}).\qedhere
\]
\end{proof}

\begin{construction}\label{const:functors_into_lax_slice}
Suppose we have a functor \[F\colon \CI\rightarrow {\tCat\oplaxslice {\CC}},\quad i\mapsto [\phi_i\colon F(i)\rightarrow \CC].\] We abuse notation and write $F$ also for the functor $F\colon \CI\rightarrow \Cat$ obtaining by forgetting the functor to $\CC$. By \Cref{cor-lax-slice}(a), the forgetful functor ${\tCat\oplaxslice {\CC}} \to \Cat$ is the cartesian unstraightening of the functor $\Fun(-,\CC)\colon \Cat^{\op} \to \Cat$. By pulling back along $F\colon \CI\rightarrow \Cat$, this implies that we can interpret $F$ as a section of the cartesian fibration
\[q\colon \Unct{\Fun(F(-), \CC)}\rightarrow \CI^{\op}.\] By the dual of \Cref{rem:laxlim_sect}, this equivalently defines an object of $\oplaxlim_{\CI^{\op}} \Fun(F(-),\CC)$. Summarizing, we have exhibited for any $F\colon \CI\rightarrow \Cat$, an equivalence between $\oplaxlim \Fun(F(-),\CC)$ and $\Fun_{\Cat}(\CI, \tCat\oplaxslice {\CC})$. The previous argument in fact applies to any $(\infty,2)$-category $\bC$ and an object $X\in \bC$, and gives an identification of 
\[
F\colon \CI\rightarrow \bC \oplaxslice X
\]
and objects of $\oplaxlim \bC(F(-),X)$. We will later apply this also to $\tCat^{\otimes,\lax}$.
\end{construction}

\subsection{The mate equivalence}
We now introduce the mate equivalence, which in fact takes two forms. To state them we introduce some notation:

\begin{definition}
    Let $\CI$ be an $\infty$-category and let $\tCat$ be the (very large) $(\infty,2)$-category of large $\infty$-categories. We write 
    \[\tFun^{\lax}_\mathrm{L}(\CI, \tCat) \quad \text{and} \quad \tFun^{\oplax}_\rR(\CI, \tCat)\] for the full subcategory of $\tFun^{\lax}(\CI, \tCat)$ and $\tFun^{\oplax}(\CI, \tCat)$ respectively spanned by those functors $\CI \to \Cat$ which sends each morphism in $\CI$ to a left and right adjoint respectively.
\end{definition}

\begin{theorem}[\cite{HHLN1}*{Theorem 5.3.6}]\label{first-mate-equivalence}
    Let $\CI$ be an $\infty$-category. Then there is an equivalence
    \[
    \tFun^{\lax}_\rL(\CI, \tCat)\simeq \tFun^{\oplax}_\rR(\CI^{\op},\tCat)
    \] 
    which sends a diagram $F\colon \CI \to \Cat^L$ to the corresponding diagram $G\colon \CI^{\mathrm{op}}\to \Cat^R$ of right adjoints.
\end{theorem}

\begin{example}
Suppose $F,F'\colon \CI\rightarrow \Cat$ are two objects of $\Fun_\rL^{\lax}(\CI,\Cat)$, and write $G$ and $G'$ respectively for the associated $\CI^{\op}$-shaped diagrams of right adjoints. Consider a lax natural transformation $\gamma \colon F\Rightarrow F'$, given at a map $f\colon i\to j$ in $\CI$ by the laxly commuting square
\begin{equation*}\label{lax-square}
    \begin{tikzcd}
        {F(i)} &  {F(j)} \\
        {F'(i)} & {F'(j).}
        \arrow["\gamma_j",from=1-2, to=2-2]
        \arrow["F(f)", from=1-1, to=1-2]
        \arrow["F'(f)"', from=2-1, to=2-2]
        \arrow["\gamma_i"', from=1-1, to=2-1]
        \arrow["\alpha", shorten <=7pt, shorten >=7pt, Rightarrow, from=2-1, to=1-2]
    \end{tikzcd}
\end{equation*}
    By \cite{HHLN1}*{Proposition 3.2.7}, the first mate equivalence sends $\gamma$ to the oplax natural transformation $\gamma \colon G\rightarrow G'$ such that given a map $f\colon i\to j$ in $\CI$ the oplaxly commuting square 
    \[
    \begin{tikzcd}
        {G(j)} &  {G(i)} \\
        {G'(j)} & {G'(i).}
        \arrow["\gamma_i",from=1-2, to=2-2]
        \arrow["G(f)", from=1-1, to=1-2]
        \arrow["G'(f)"', from=2-1, to=2-2]
        \arrow["\gamma_j"', from=1-1, to=2-1]
        \arrow["\beta", shorten <=7pt, shorten >=7pt, Rightarrow, from=1-2, to=2-1]
    \end{tikzcd}
    \] is filled by the \emph{Beck--Chevalley transformation} $\beta$ of $\alpha$, defined to be the composite 
    \[
    \gamma_j G(f) \xRightarrow{\eta} G'(f)F'(f) \gamma_i G(f) \xRightarrow{\alpha} G'(f) \gamma_j F(f) G(f) \xRightarrow{\epsilon} G'(f)\gamma_i.
    \]
\end{example}

For the second mate equivalence we again introduce some notation:
\begin{definition}\label{def:lax_fun_adj_morp}
Let $\CI$ be an $\infty$-category. We write 
\[\Fun^{\rR,\lax}(\CI, \tCat) \quad \text{and} \quad \Fun^{\rL,\oplax}(\CI, \tCat)\] for the wide subcategories of $\Fun^{\lax}(\CI, \tCat)$ and $\Fun^{\oplax}(\CI, \tCat)$ respectively spanned on morphisms by those (op)lax natural transformations $\eta\colon F \Rightarrow G$ such that each component $\eta_i$ is a right/left adjoint respectively. 
\end{definition}

\begin{theorem}[\cite{HHLN1}*{Theorem 5.3.5}]\label{second-mate-equivalence}
    Let $\CI$ be an $\infty$-category. Then there is an equivalence
    \[
    \tFun^{\rR,\lax}(\CI, \tCat)\simeq \tFun^{\rL,\oplax}(\CI,\tCat)^{(1,2)\textup{-op}},
    \]
    where $(-)^{(1,2)\textup{-op}}$ denotes the functor obtained by passing to opposites on the level of both 1 and 2-morphisms.
\end{theorem}

\begin{example}
Given the square (\ref{lax-square}), let $\gamma_i^\rL$ and $\gamma_j^\rL$ denote the left adjoints of $\gamma_i$ and $\gamma_j$ respectively. Then the second mate equivalence sends the square (\ref{lax-square}) to the oplaxly commuting square 
    \[
    \begin{tikzcd}
        {F'(i)} &  {F'(j)} \\
        {F(i)} & {F(j).}
        \arrow["\gamma_j^\rL",from=1-2, to=2-2]
        \arrow["F'(f)", from=1-1, to=1-2]
        \arrow["F(f)"', from=2-1, to=2-2]
        \arrow["\gamma_i^\rL"', from=1-1, to=2-1]
        \arrow["\beta", shorten <=7pt, shorten >=7pt, Rightarrow, from=1-2, to=2-1]
    \end{tikzcd}
    \] 
    which is filled by the Beck--Chevalley transformation $\beta$ of $\alpha$, defined to be the composite 
    \[
    \gamma_j^\rL F'(f) \xRightarrow{\eta} \gamma_j^\rL F'(f)\gamma_i \gamma_i^\rL \xRightarrow{\alpha} \gamma_j^\rL\gamma_j F(f)\gamma_i^\rL \xRightarrow{\epsilon} F(f)\gamma_i^\rL.
    \]
\end{example}

\begin{example}\label{ex:flipping_adjoints}
Specializing the second mate equivalence to the case $\CI \simeq \ast$ we obtain an equivalence $\tCat^\rR\simeq (\tCat^\rL)^{(1,2)\text{-op}}$.
\end{example}

As a first example of the mate equivalence, we observe that one can identify certain lax and oplax limits taken over adjoint diagrams. 

\begin{proposition}\label{prop:adjoint_diag_lax_oplax}
Let $G\colon \CI \rightarrow \Cat^\rR$ be a diagram of right adjoints. Write $F\colon \CI^{\op}\rightarrow \Cat^\rL$ for the diagram of left adjoints. Then there exists an equivalence $\oplaxlim_{\CI} G\simeq \laxlim_{\CI^{\op}} F$.
\end{proposition}

\begin{proof}
    On categories this follows immediately from the following chain of equivalences
    \[
    \oplaxlim G = \Nat^{\oplax}(\Delta(\ast),G) \simeq \Nat^{\lax}(\Delta(\ast),F) = \laxlim F,
    \] where the first and the third are by definition and the middle one follows from the definition of adjoint diagrams in higher category theory.
\end{proof}

\section{Unstable global homotopy theory}
In this section we recall the necessary background on global homotopy theory following~\cites{GH, Rezk, LNP}.

To begin recall that a \emph{global family} is a collection of compact Lie groups $\CE$ which is closed under isomorphisms, passage to subgroups and quotients. A global family is said to be \emph{multiplicative} if in addition it is closed under finite products.

\begin{definition}
    Let $\CE$ be a global family of compact Lie group.
    \begin{enumerate} 
        \item We let $\Glo{\CE}$ denote the \emph{global orbit} $\infty$-\emph{category} of \cite{GH} with isotropy in $\CE$, whose objects are given by $\CB G$ for $G\in\CE$. Up to homotopy, morphisms $\CB H\rightarrow \CB G$ in $\Glo{\CE}$ are given by conjugacy classes of continuous group homomorphisms $\alpha\colon H\rightarrow G$. More precisely, by \cite{LNP}*{Proposition 6.3}, we have
        \begin{equation}\label{mapping-spaces-glo}
        \Glo{\CE}(\CB H, \CB G) \simeq \coprod\limits_{[\alpha]} BC(\alpha)
        \end{equation}
	where $[\alpha]$ runs through the set of conjugacy classes of continuous group homomorphisms from $H$ to $G$, and $C(\alpha)$ denotes the centraliser of the image of $\alpha$.
        \item We denote by $\Orb{\CE}\subseteq \Glo{\CE}$ the wide subcategory spanned by those morphisms which are represented by an injective group homomorphism.	
        \item The $\infty$-category of \emph{$\CE$-global spaces} $\Spcgl{\CE} \coloneqq \Fun(\Glo{\CE}^{\op},\Spc)$ is the presheaf $\infty$-category on $\Glo{\CE}$. We identify $\CB G$ with its image under the Yoneda embedding.
    \end{enumerate}
\end{definition}

\begin{notation}
    When $\CE$ is the global family of all compact Lie groups, we will omit the $\CE$ and simply write $\Glo{}, \Orb{}$ and $\Spc_{\gl}$ for these $\infty$-categories. When $\CE$ is the global family of abelian compact Lie groups, we will simplify the notation to $\Glo{\mathrm{ab}}$, $\Orb{\mathrm{ab}}$ and $\Spc_{\ab}$. Similarly when we restrict to finite abelian groups, we will write $\Glo{\fab}$, $\Orb{\fab}$ and $\Spc_{\fab}$. 
\end{notation}

\begin{notation}
    When it is clear from the context, we will often simply call an object of $\Spcgl{\CE}$ a global space, leaving the family $\CE$ implicit. We will also do this with all other ``global" objects we consider in this article.
\end{notation}

\begin{remark}\label{rem-pointed_glo}
    By \cite{GM20}*{Remark 2.14} the $\infty$-category $\mathrm{Glo}_{\ast}$ of pointed objects in $\Glo{}$ is equivalent to $\mathrm{CptLie}$, the topologically enriched category of compact Lie groups. In particular we obtain a functor $\CB(-)\colon \mathrm{CptLie}\rightarrow \Glo{}$ which ``forgets the basepoint''.
\end{remark}

\begin{remark}\label{rem:Glo_fin_prods}
If $\CE$ is a multiplicative global family, then $\Glo{\CE}$ has finite products inherited from $\mathrm{CptLie}$. This follows for example from a simple computation using \cite{LNP}*{Proposition 6.4}.
\end{remark}

Crucial to the story of global homotopy theory is its close relationship to equivariant homotopy theory. It follows from \cite{LNP}*{Lemma 6.13} that for any $G\in \CE$ the slice $\infty$-category $(\Orb{\CE})_{/{\CB G}}$ can be identified with $\mathrm{Orb}_G$, the orbit $\infty$-category of the compact Lie group $G$. Restricting along the functor $(\Orb{\CE})_{/{\CB G}} \to \Glo{\CE}$, $(\CB H \to \CB G)\mapsto
\CB H$ and using Elmendorf's theorem yields a \emph{restriction} functor 
\[
\mathrm{res}_G \colon \Spcgl{\CE} \to \Spc_G
\]
into the $\infty$-category of $G$-spaces. 
The right Kan extension along ${\Orb{\CE}}_{/{\CB G}} \to {\Glo{\CE}}$ defines a right adjoint to $\mathrm{res}_G$. We record the following special case.

\begin{notation}\label{not-righ-adj-eval}
   We denote the right adjoint of $\mathrm{res}_e\colon \Spcgl{\CE} \to \Spc$ by $X \mapsto X^{(-)}$. One verifies that this right adjoint is in fact fully faithful.
\end{notation}

\begin{example}\label{ex-classifying-space-right-induced}
    For any compact Lie abelian group $K$ we have $\CB K=BK^{(-)}$, see the discussion in \cite{Schwede20}*{Theorem 1.2.32}.
\end{example}

The restriction functor $\mathrm{res}_G$ also admits a left adjoint $-{\sslash} G\colon \Spc_G \to {\Spcgl{\CE}}$ which is given by left Kan extension. It will be important to record the following special case. 

\begin{notation}\label{not-constant-global-space}
     The functor $-\sslash e \colon \Spc \to \Spcgl{\CE}$ is fully faithful and we refer to a global space in its image as a \emph{constant} global space. Given a space $X$, we will identify $X$ with its constant global space $X \sslash e$. 
\end{notation}

One verifies that $\pt \sslash G \simeq \CB G$ so we obtain an induced functor $-\sslash G \colon \Spc_G \to {\Spcgl{\CE}}_{/\CB G}$. To describe the essential image of this functor we introduce the following definition: 

\begin{definition}
    A morphism of global spaces $f\colon \sX \to \sY$ is \emph{faithful} (or \emph{representable}) if for every $\CB G\in\Glo{\CE}$ and normal subgroup $N \triangleleft G$ the diagram 
    \[
    \begin{tikzcd}
        \sX(\CB G/N)\arrow[r]\arrow[d] & \sX(\CB G)\arrow[d]\\
        \sY(\CB G/N)\arrow[r] & \sY(\CB G)
    \end{tikzcd}
    \]
    is a pullback. We will distinguish faithful maps by writing $\sX \hookrightarrow \sY $.
\end{definition}

\begin{example}
    A morphism $\CB H \to \CB G$ in $\Glo{}$ is faithful if and only if it can be represented by an injective continuous group homomorphism $H \to G$.
\end{example}

\begin{example}\label{ex-nu-map}
   By \Cref{mapping-spaces-glo} we see that $\CB e\simeq \pt$, and $\Map_{\Glo{\CE}}(\pt, \CB G)\simeq BG$. Therefore we obtain a map $\nu\colon BG \rightarrow \CB G$, given by the counit of the adjunction between spaces and global spaces. Using~\Cref{mapping-spaces-glo} again is not hard to see that $\nu$ is faithful.
\end{example}

\begin{remark}\label{rem-faithful-pullback}
    By an application of the Yoneda Lemma, a map $f\colon \sX\rightarrow \sY$ is faithful if and only if it is right orthogonal to every map $\CB p\colon \CB G\rightarrow \CB G/N$ between representables such that $p$ is a surjective group homomorphism. In particular faithful morphisms are closed under pullback.
\end{remark}

\begin{definition}\label{def-quotient-map}
    By the dual of~\cite{HTT}*{Proposition 5.5.5.7}, the faithful maps are the right class of a factorization system on $\Spcgl{\CE}$. We call morphisms in the associated left class \emph{\quotientmap s} and denote them by $f\colon \sX\twoheadrightarrow \sY$.
\end{definition}

\begin{example}
By \Cref{rem-faithful-pullback}, a map $f\colon \CB H \to \CB G$ is a \quotientmap{} if and only if it represented by a surjective group homomorphism.
\end{example}

\begin{theorem}\label{Rezk}
    Let $\CE$ be a global family and consider $G\in \CE$. The functor 
    \[
    -{\sslash} G\colon \Spc_G\rightarrow {\Spcgl{\CE}}_{/\CB G}
    \]
    is fully faithful and an object $f\colon \sX\rightarrow \CB G$ is in its essential image
    if and only if $f$ is faithful.
\end{theorem}

\begin{proof}
    The claim for the global family of all compact Lie groups is proved in \cite{GM20}*{Proposition 2.19}. We can obtain the general case from this one by noting that there is a factorization 
    \[
    -{\sslash} G\colon \Spc_G \to (\Spcgl{\CE})_{/\CB G} \to (\Spc_{\gl})_{/\CB G}
    \]
    where both functors are fully faithful. 
\end{proof}

\begin{remark}
As observed in \cite{Rezk}*{Section 5.3}, the functor $-\sslash G\colon \Spc_G\rightarrow {\Spcgl{\CE}}_{/\CB G}$ admits a left adjoint, and so the $\infty$-category of $G$-spaces is a Bousfield localization of global spaces over $\CB G$. We may describe this left adjoint as follows: Given a map of global spaces $f\colon \sX\rightarrow \CB G$, we may factor $f$ as a \quotientmap{} followed by a faithful map $\sX\twoheadrightarrow \sX'\hookrightarrow \CB G$. The left adjoint applied to $f$ is the map $\sX'\hookrightarrow \CB G$, while the map  $\sX\twoheadrightarrow \sX'$ gives the unit of the adjunction.
\end{remark}

\begin{remark}\label{rem:restriction_functoriality} 
We can also give an interpretation of the functoriality of equivariant spaces using the previous result. Given a group homomorphism $\alpha\colon H\rightarrow G$, pullback and postcomposition by $B\alpha\colon \CB H\rightarrow \CB G$ induces an adjunction
\[
\CB \alpha_!\colon {\Spcgl{\CE}}_{/\CB H} \rightleftarrows {\Spcgl{\CE}}_{/\CB G}\cocolon \CB \alpha^*.
\]
As observed in \Cref{rem-faithful-pullback}, pulling back preserves faithful maps. Therefore the right adjoint above restricts to a functor
\[
(\CB \alpha)^*\colon (\Spcgl{\CE})_{/\CB G}^{\mathrm{fth}} \rightarrow (\Spcgl{\CE})_{/\CB H}^{\mathrm{fth}}
\]
between faithful morphisms.
By \cite{LNP}*{Proposition 6.17}, this functor agrees under the equivalences of \Cref{Rezk} with the standard restriction functoriality $\alpha^*\colon \Spc_G\rightarrow \Spc_H$ of equivariant spaces.

It follows by standard arguments that the left adjoint of $\CB \alpha^*$, restricted to global spaces over $\CB G$ and $\CB H$ with faithful structure map, is computed by first applying $\CB \alpha_!$ and then reflecting back into $(\Spcgl{\CE})_{/\CB G}^{\mathrm{fth}}$. This agrees by uniqueness of adjoints with the induction functor on equivariant spaces. In particular the counit of the adjunction $\alpha_!\dashv \alpha^*$ is given by factoring the counit $\alpha^* X\sslash H \simeq (\CB \alpha)^* (X\sslash G)\rightarrow X\sslash G$ of the pullback-postcomposition adjunction into a \quotientmap{} followed by a faithful map as in the following diagram 
\[\begin{tikzcd}
	{\alpha^* X\sslash H} & {\alpha_!\alpha^* X\sslash G} & {X\sslash G} \\
	{\CB H} && {\CB G.}
	\arrow[hook, from=1-1, to=2-1]
	\arrow[two heads, from=1-1, to=1-2]
	\arrow["\epsilon", hook, from=1-2, to=1-3]
	\arrow[hook, from=1-3, to=2-3]
	\arrow[""{name=0, anchor=center, inner sep=0}, "{\CB \alpha}", from=2-1, to=2-3]
	\arrow["\lrcorner"{anchor=center, pos=0.125}, draw=none, from=1-1, to=0]
\end{tikzcd}\]

\end{remark}

We also briefly discuss the identification of free $G$-spaces under the equivalence of \Cref{Rezk}. To establish notation we recall that $\mathrm{End}_{\Orb{G}} (G/e) \simeq G$, and so we obtain a restriction functor
\[
\Spc_G\coloneqq \Fun(\Orb{G},\Spc)\xrightarrow{\ev_{G/e}} \Fun(BG,\Spc)\simeq \Spc_{/BG}.
\] 
We call $G$-spaces in the image of the fully faithful left adjoint \emph{free} $G$-spaces.

\begin{proposition}\label{prop:free_G_spaces_as_global_spaces}
A map $\sX\rightarrow BG$ of global spaces is faithful if and only if it is in the image of the constant functor. Now consider the map $\nu\colon BG\rightarrow \CB G$ from \Cref{ex-nu-map}. Both functors in the adjunction 
    \[
    \nu_!\colon {\Spcgl{\CE}}_{/BG} \rightleftarrows {\Spcgl{\CE}}_{/\CB G} \cocolon \nu^*
    \]
    preserve objects with a faithful structure map, and the following square commutes 
    \[\begin{tikzcd}
        {(\Spcgl{\CE})_{/\CB G}^{\mathrm{fth}}} & {(\Spcgl{\CE})_{/BG}^{\mathrm{fth}}} \\
        {\Spc_{G}} & {\Spc_{/ BG}.}
        \arrow["{\ev_{G/e}}", from=2-1, to=2-2]
        \arrow["{\nu^*}", from=1-1, to=1-2]
        \arrow["\sim"', "\mathrm{res_e}", from=1-2, to=2-2]
        \arrow["\sim", "\mathrm{res}_G"', from=1-1, to=2-1]
    \end{tikzcd}\] In particular passing to left adjoints, we conclude that an object $\sX\rightarrow \CB G$ of ${\Spcgl{\CE}}_{/\CB G}$ corresponds to a free $G$-space if and only if $\sX$ is a constant global space and the map to $\CB G$ is faithful.
\end{proposition}

\begin{proof}
    Suppose $f\colon \sX\rightarrow BG$ is a faithful map of global spaces. For any $\CB H$ the square 	
    \[
    \begin{tikzcd}
        \sX(\pt)\arrow[r]\arrow[d] & \sX(\CB H)\arrow[d]\\
        BG \arrow[r, "\sim"] &  BG
    \end{tikzcd}
    \] is a pullback square. So we conclude that the map $\sX(\pt)\rightarrow \sX(\CB H)$ is an equivalence, and so $\sX$ is constant. For the second statement we first note that faithful maps are closed under pullback, and so $\nu^*$ clearly preserves faithful maps. Since composition of faithful maps is again faithful and $\nu_!$ is given by postcomposing with $\nu$, it follows that $\nu_!$ preserves faithful maps. The commutativity of the diagram follows from the observation that $\mathrm{res}_e$ (which agrees with evaluation at $\pt$) preserves pullbacks square so we obtain the first equivalence in the following sequence: 
    \[
    \mathrm{res}_e \nu^*\sX\simeq \mathrm{res}_e X \simeq \mathrm{ev}_{G/e}\circ \mathrm{res}_G \sX.\qedhere
    \]
\end{proof}

The connection between global and equivariant homotopy theory expressed by \Cref{Rezk} can be extended to give a different perspective on global homotopy theory.
Recall from \cite{LNP}*{Section 6}, that there exists a functor 
$\Spc_{\bullet}\colon \Glo{\CE}^{\mathrm{op}} \to \Cat $, which sends $\CB G$ to $\Spc_G$, and sends a morphism $\CB\alpha \colon \CB H\to \CB G$ to the restriction-inflation functor $\alpha^* \colon \Spc_G \to \Spc_H$. We recall the following result.

\begin{theorem}[\cite{LNP}*{Theorem 6.18}]\label{thm:__lax_lim}
    Let $\CE$ be a global family of compact Lie groups. Then the restriction functors induce an equivalence of $\infty$-categories
    \[
    \Spcgl{\CE}\simeq \laxlimdag_{\Glo{\CE}^{\op},\Orb{\CE}^{\op}} \Spc_\bullet
    \]
    where we view $\Glo{\CE}^{\op}$ as a marked $\infty$-category via the inclusion $\Orb{\CE}^{\op}\subset \Glo{\CE}^{\op}$. 
\end{theorem}

\begin{remark}\label{rem:marking_Glo}
In this article we have many occasions to take a partially (op)lax limits over the marked $\infty$-category $(\Glo{\CE},\Orb{\CE})$, as well as subcategories and opposite categories thereof. We therefore make the global\footnote{Pun intended.} convention that any category derived from $\Glo{}$ is marked by the collection of faithful edges.
\end{remark}

\begin{remark}
    Using the previous remark, we can informally summarize \Cref{thm:__lax_lim} as follows: a global space $\sX$ is equivalent to the data of
    \begin{itemize}
        \item a $G$-space $\mathrm{res}_G \sX$ for each group $G\in \CE$;
        \item an $H$-equivariant map $f_\alpha\colon \alpha^*\mathrm{res}_G \sX\to \mathrm{res}_H \sX$ for each continuous group homomorphism $\alpha\colon H \to G$;
         \item the maps $f_\alpha$ are functorial, in the sense that there are given homotopies $f_{\beta\circ \alpha}\simeq f_\beta\circ \beta^*(f_{\alpha})$ for all composable maps $\alpha$ and $\beta$, and $f_{\mathrm{id}}\simeq\mathrm{id}$;
        \item the map $f_\alpha$ is an equivalence for every continuous \emph{injective} homomorphism $\alpha$.
        \item a homotopy between  the map $f_{c_g}$ induced by the conjugation isomorphism and the map $l_g \colon c_g^* \mathrm{res}_G X \to \mathrm{res}_G X$ given by left multiplication by $g$;
        \item higher coherences for the homotopies.
    \end{itemize}
\end{remark}	

Here we record a consequence of this description for constructing functors out of global spaces.

\begin{proposition}\label{prop:glquotient-oplaxlim}
Let $\mathcal{C}$ be an $\infty$-category. Restriction along the functors $-{\sslash}G\colon\Spc_G\rightarrow \Spcgl{\CE}$ induces a functor 
\[
\Phi\colon \Fun(\Spcgl{\CE},\mathcal{C})\to \oplaxlimdag_{\Glo{\CE}^{\op}} \Fun(\Spc_\bullet,\mathcal{C}),
\] 
where the diagram $\Fun(\Spc_\bullet,\mathcal{C})$ is functorial in $\Glo{\CE}^{\op}$ as follows: a morphism $\CB \alpha \colon \CB H \to \CB G$ in $\Glo{\CE}$ is sent to restriction along the induction functor $\alpha_!\colon \Spc_H \rightarrow \Spc_G$. Furthermore suppose $\mathcal{C}$ admits all small colimits. Then $\Phi$ restricts to an equivalence
\[
\Fun^\mathrm{L}(\Spcgl{\CE},\mathcal{C})\simeq \oplaxlimdag\limits_{\Glo{\CE}^{\op}} \Fun^\mathrm{L}(\Spc_\bullet,\mathcal{C}).
\]
\end{proposition}

\begin{proof}
Recall that $\Spcgl{\CE}\simeq \laxlimdag \Spc_G$. In particular we obtain a universal partially lax cone
\[\begin{tikzcd}
        & {\Spc_H} \\
        {\Spcgl{\CE}} \\
        & {\Spc_G.}
        \arrow[""{name=0, anchor=center, inner sep=0}, "{\mathrm{res}_H}", from=2-1, to=1-2]
        \arrow["{\mathrm{res}_G}"', from=2-1, to=3-2]
        \arrow["{\alpha^*}"', from=3-2, to=1-2]
        \arrow[shorten <=19pt, shorten >=12pt, Rightarrow, from=3-2, to=0]
    \end{tikzcd}\]
Passing to left adjoints (by which we mean using the equivalence $\tCat^\rR \simeq (\tCat^\rL)^{(1,2)\text{-}\op}$, see \Cref{ex:flipping_adjoints}) we obtain a partially oplax cocone
\begin{equation}\label{oplaxcone}
    \begin{tikzcd}
            {\Spc_H} \\
            & {\Spcgl{\CE}} \\
            {\Spc_G.} &
            \arrow[""{name=0, anchor=center, inner sep=0}, "{- \sslash H}", from=1-1, to=2-2]
            \arrow["{ - \sslash G}"', from=3-1, to=2-2]
            \arrow["{\alpha_!}"', from=1-1, to=3-1]
            \arrow[shorten <=12pt, shorten >=19pt, Rightarrow, from=0, to=3-1]
    \end{tikzcd}
\end{equation}
Given a functor $F\colon \Spcgl{\CE}\to \mathcal{C}$ we may precompose by this cocone to obtain a partially oplax cocone with target $\mathcal{C}$, encoded by a diagram $\Glo{\CE}\rightarrow \tCat \oplaxslice \mathcal{C}$. This in turn gives an object of $\oplaxlimdag \Fun(\Spc_G,\mathcal{C})$, see \Cref{const:functors_into_lax_slice}.
    
For the final statement we first observe that because each functor $-\sslash G$ preserves colimits, the functor constructed above restricts appropriately. To see that it is an equivalence we compute
\begin{align*}
    \Fun^{\rL}(\Spcgl{\CE},\mathcal{C}) \simeq \Fun(\Glo{\CE},\mathcal{C}) &\simeq \Fun(\oplaxcolimdag \mathrm{Orb}_\bullet, \mathcal{C}) \\ &\simeq \oplaxlimdag \Fun(\mathrm{Orb}_\bullet,\mathcal{C}) \simeq \oplaxlimdag \Fun^{\rL}(\Spc_\bullet,\mathcal{C}),
\end{align*}
where the second and third equivalence are justified by the proof of Theorem 6.18 and Proposition 4.15 of \cite{LNP} respectively.
\end{proof}

\begin{remark}
Unwinding the proof of the previous proposition, we find that the image of $F\colon \Spcgl{\CE} \rightarrow \CC$ under the functor of \Cref{prop:glquotient-oplaxlim} and the equivalence of \Cref{Rezk} is given by the family of functors $\{F_G(X ) = F(X\sslash G)\}$. The oplax structure map $F_G(\alpha_!(-))\rightarrow F_H(-)$ is given at an object $X$ by applying $F$ to the unique \quotientmap{} $X\sslash G\twoheadrightarrow \alpha_! X\sslash H$ for which the square
\[
\begin{tikzcd}
    X\sslash G & {\alpha_!X\sslash H} \\
    {\CB G} & {\CB H}
    \arrow["f", from=2-1, to=2-2]
    \arrow["i", hook, from=1-1, to=2-1]
    \arrow["\alpha_!(i)",hook, from=1-2, to=2-2]
    \arrow[two heads, from=1-1, to=1-2]
\end{tikzcd}
\]
commutes.
\end{remark}

\section{Stable global homotopy theory}

Global spectra were defined in \cite{Schwede18} as a $\infty$-category of representing objects for genuine cohomology theories on global spaces. Just as in the unstable setting, one can describe global spectra as a laxly compatible collection of genuine equivariant spectra.

\begin{theorem}[\cite{LNP}*{Theorem 11.10}]\label{thm:Sp_gl_lax_lim}
Let $\CE$ be a multiplicative global family of compact Lie groups. Then there exists a functor 
\[
\Sp_\bullet\colon \Glo{\CE}^{\op}\rightarrow \PrL, \quad \CB G\mapsto \Sp_G
\] 
which sends $\CB G$ to the $\infty$-category of \textit{genuine $G$-spectra} and an equivalence of symmetric monoidal $\infty$-categories
\[
\Spgl{\CE} \simeq \laxlimdag \Sp_\bullet.
\]
\end{theorem}

In this section we refine this result, and show that in certain cases a simpler description of global spectra is possible.

\begin{definition}\label{def:enough-injectives}
    Let $\CE$ be a multiplicative global family. We say a full subcategory $\CT\subset \Glo{\CE}$ is a subcategory of \emph{enough injectives} if
    \begin{enumerate}
        \item $\CT$ is closed under products in $\Glo{\CE}$.
        \item Given any object $\CB G$, there exists a faithful map $\CB G\hookrightarrow \CB K$ for some $\CB K\in \CT$.
    \end{enumerate}
    We view $\CT$ as a marked $\infty$-category by once again marking the subcategory of faithful maps.
\end{definition}

\begin{remark}
    Recall from \Cref{rem:Glo_fin_prods} that $\Glo{\CE}$ has finite products, so condition (1) of the previous definition is well-defined.
\end{remark}

\begin{example}
    The full subcategory $\Sigma \subset \Glo{\fin}$ of the global orbit $\infty$-category of finite groups spanned by those $\CB G$ such that $G$ is a finite product of symmetric groups is a subcategory of enough injective objects. This is a consequence of Cayley's theorem.
\end{example}

\begin{example}
    The full subcategory $\Tori\subset \Glo{\ab}$ of the global orbit $\infty$-category of abelian compact Lie groups spanned by those $\CB G$ such that $G$ is a torus is a subcategory of enough injective objects. 
\end{example}

\begin{example}
  The full subcategory $\mathrm{U}\subset \Glo{}$ of the global orbit $\infty$-category of compact Lie groups spanned by those $\CB G$ such that $G$ is a finite product of unitary groups is a subcategory of enough injective objects. This follows from the fact that any compact Lie group admits a faithful unitary representation, by an application of the Peter--Weyl theorem.
\end{example}

The variant of \Cref{thm:Sp_gl_lax_lim} which we will prove is the following:

\begin{theorem}\label{thm:P_gl_sp_from_tori}
    Let $\CE$ be a multiplicative global family of compact Lie groups, and suppose $\CT\subset \Glo{\CE}$ is a family of enough injective objects. Then there exists an equivalence of symmetric monoidal $\infty$-categories
    \[
    \Spgl{\CE} \simeq \laxlimdag_{\CT^{\op}} \Sp_\bullet.
    \]
\end{theorem}

These equivalences will follow from the fact that the inclusion $\CT^{\op}\subset \Glo{\CE}^{\op}$ is marked final, a concept from \cite{AG20} which we recall now.

\begin{proposition}\label{prop:markedfinal}
    Let $F\colon \CI\rightarrow \CJ$ be a functor of marked $\infty$-categories. The following are equivalent:
    \begin{enumerate}
        \item Given any functor $G\colon \CJ\rightarrow \Cat$, $\laxlimdag G\simeq \laxlimdag GF$ whenever either exist.
        \item Given any functor $G\colon \CJ\rightarrow \Cat$, $\operatorname*{oplaxlim^\dagger} G\simeq \operatorname*{oplaxlim^\dagger} GF$ whenever either exist.
    \end{enumerate}
\end{proposition}

\begin{proof}
    This follows from the fact that $(-)^{\op}$ is an auto-equivalence of $\Cat$ which sends lax limits to oplax limits and vice-versa. 
\end{proof}

\begin{definition}
    We say a functor $F\colon \CI\rightarrow \CJ$ of marked $\infty$-categories is \emph{marked final} if it satisfies the equivalent conditions of the previous proposition. We say $F$ is \emph{marked cofinal} if $F^{\op}$ is marked final.
\end{definition}

\begin{remark}
    $F$ is marked cofinal if and only if it preserves partially (op)lax colimits, in the sense of Proposition \ref{prop:markedfinal}. Therefore our definition agrees with \cite{AG20}*{Definition 5.4}, and we may freely cite their results.
\end{remark}

\begin{notation}
    Given a functor $F\colon\CI\rightarrow \CJ$ such that $\CJ$ is a marked $\infty$-category and $j\in \CJ$, we write $\CI_{j/}$ for the comma category $\{j\}\downarrow F$. We consider this as a marked $\infty$-category by marking all the edges whose projection to $\CJ$ is marked. 
\end{notation}

Given a marked $\infty$-category $\CC$, we write $\mathcal{L}(\CC)$ for the localization of $\CC$ at the marked edges in $\CC$.

\begin{lemma}\label{lemma:productsinloc}
    Let $\CC$ be a marked $\infty$-category which admits finite products. Suppose that the functor $X\times (-)\colon \CC\rightarrow \CC$ is a marked functor for every object $X\in \CC$. Then $\mathcal{L}(\CC)$ admits finite products, and the localization functor $\CC\rightarrow \mathcal{L}(\CC)$ preserves them.
\end{lemma}

\begin{proof}
    The product functor $-\times-\colon \CC\times\CC\rightarrow \CC$ is determined by being right adjoint to the diagonal functor $\Delta_{\CC}\colon \CC\rightarrow \CC\times\CC$. Our assumptions imply that the product derives to a functor $\mathcal{L}(-\times-)\colon \mathcal{L}(\CC\times \CC)\rightarrow \mathcal{L}(\CC)$, which is right adjoint to $\mathcal{L}(\Delta_{\CC})$. Here we consider $\CC\times \CC$ as a marked $\infty$-category by taking the product in marked $\infty$-categories, explicitly this is given by marking an edge $(f,g)$ whenever both $f$ and $g$ are marked. One can compute that $\mathcal{L}(\CC\times \CC)\simeq \mathcal{L}(\CC)\times \mathcal{L}(\CC)$, that is $\mathcal{L}$ preserves products. 
    We therefore have a commutative diagram 
    \[
    \begin{tikzcd}
        \CC \times \CC \arrow[r,shift right, "-\times -"'] \arrow[d] & \CC \arrow[d]\arrow[l, shift right, "\Delta_{\CC}"']\\
        \mathcal{L}(\CC \times \CC)\arrow[d,"\sim"'] \arrow[r,shift right, "\mathcal{L}(-\times -)"'] & \mathcal{L}(\CC)\arrow[l, shift right, "\mathcal{L}(\Delta_{\CC})"']\arrow[d,"="]\\
        \mathcal{L}(\CC)\times \mathcal{L}(\CC) \arrow[r, shift right,dotted, "-\times -"'] & \mathcal{L}(\CC)\arrow[l, shift right, "\Delta_{\mathcal{L}(\CC)}"'].
    \end{tikzcd}
    \]  
    and the dotted arrow defines a product in $\mathcal{L}(\CC)$. The commutativity of the above diagram also shows that the localization functor preserves products.
\end{proof}

It will be helpful to keep the following example in mind. 

\begin{example}\label{ex-products-inj}
    Let $\CT \subset \Glo{\CE}$ be a subcategory of enough injectives. For any $\CB G\in \Glo{\CE}$, we can form the comma category $\CT_{\CB G/}$. A morphism 
    \[
    \begin{tikzcd}
        & \CB G \arrow[rd]\arrow[ld] & \\
        \CB K \arrow[rr, "f"] & & \CB J
    \end{tikzcd}
    \]
    in this comma category is marked precisely when $f$ is represented by an injective group homomorphism. Note that $\CT_{\CB G/}$ has finite products since $\CT$ is closed in $\Glo{\CE}$ under finite products. Furthermore, taking products with a fixed object of $\CT_{\CB G/}$ preserves faithful maps, and therefore Lemma \ref{lemma:productsinloc} implies that $\mathcal{L}(\CT_{\CB G/})$ has finite products and that the localization $\CT_{\CB G/}\rightarrow \mathcal{L}(\CT_{\CB G/})$ preserve them.
    
\end{example}

The key ingredient for the proof of \Cref{thm:P_gl_sp_from_tori} is the following criterion.

\begin{theorem}[\cite{AG20}*{Theorem 5.10}] \label{thm:markedcofinality}
A functor $F\colon I\rightarrow J$ is marked cofinal if and only if 
    \begin{enumerate}
        \item For every $j\in J$, there exists an object $g\colon j\rightarrow F(i)$ of $I_{j/}$ which is marked when viewed as a morphism in $J$;
        \item Any object of $I_{j/}$ of the form above is an initial object in $\mathcal{L}(I_{j/})$;
        \item Given a marked edge $j\rightarrow j'$ of $J$, the induced map $\mathcal{L}(I_{j'/})\rightarrow \mathcal{L}(I_{j/})$ preserves initial objects.   
    \end{enumerate}	
\end{theorem}

In order to verify condition (2) of the previous criterion we will need the following consequence of the axioms of a subcategory of enough injectives.
\begin{lemma}\label{lem-condition3-enough-inj}
    Let $\CT\subset \Glo{\CE}$ be a subcategory of enough injective objects. Given any map $f\colon \CB H\rightarrow \CB G$ in $\Glo{\CE}$ there is a commutative square
        \[
        \begin{tikzcd}
            \CB K & \CB J \\
            {\CB H} & {\CB G}
            \arrow["{\tilde{f}}", from=1-1, to=1-2]
            \arrow["f", from=2-1, to=2-2]
            \arrow[hook, from=2-1, to=1-1]
            \arrow[hook, from=2-2, to=1-2]
        \end{tikzcd}
        \] 
        in $\Glo{\CE}$ such that the vertical maps are faithful and $\CB K,\CB J \in\CT$.
\end{lemma}

\begin{proof}
    By condition (2) for $\CT$ there are faithful maps $i \colon \CB H \hookrightarrow \CB H'$ and $j \colon \CB G \hookrightarrow \CB G'$ with $\CB H',\CB G' \in \CT $. We can then form the commutative square
    \[
    \begin{tikzcd}
        \CB H' \times \CB G' \arrow[r,"\pr_2"] & \CB G' \\
        \CB H \arrow[u, hook,"i \times( j\circ f)"] \arrow[r,"f"] & \CB G \arrow[u,hook,"j"']\\.
    \end{tikzcd}
    \]
    It is only left to note that $\CB H' \times \CB G'\in \CT$ by condition (1), and that the map $i \times (j \circ f)$ is faithful since $i$ is so.
\end{proof}

\begin{theorem}\label{thm-tori-marked-final}
    Let $\CT\subset \Glo{\CE}$ be a subcategory of enough injective objects. Then the inclusion $\CT^{\op} \rightarrow \Glo{\CE}^{\op}$ is marked final.
\end{theorem}

\begin{proof}
    We apply Theorem \ref{thm:markedcofinality} to the inclusion $\CT\rightarrow \Glo{\CE}$. Condition (1) follows from \Cref{def:enough-injectives}(2). Note that if we assume condition (2), then condition (3) is immediate from the fact that a composite of faithful maps is faithful. Therefore all that remains is condition (2). Fix an object $\CB G\in \Glo{\CE}$, and consider the $\infty$-category $\CT_{\CB G/}$. As discussed in \Cref{ex-products-inj}, the comma category $\mathcal{L}(\CT_{\CB G/})$ has finite products and the localization $\CT_{\CB G/}\rightarrow \mathcal{L}(\CT_{\CB G/})$ preserves them. However given an object $\alpha\colon \CB G\rightarrow X$ in $\CT_{\CB G/}$, the composite 
    \[
    \begin{tikzcd}
        & \CB G \arrow[ld, "\alpha"'] \arrow[d, "\alpha\times\alpha" description] \arrow[rd, "\alpha"] &      \\
        \CB K \arrow[r, hook, "\mathrm{\Delta}"'] \arrow[rr, bend right, "\mathrm{id}"] & \CB K\times K \arrow[r, "\mathrm{pr}_1"']                                              & \CB K
    \end{tikzcd}
    \] 
    in $\CT_{\CB G/}$ exhibits $\mathrm{pr_1}$ as an equivalence in $\mathcal{L}(\CT_{\CB G/})$; i.e., every object is equivalent to its product. On mapping spaces we obtain that 
    \[
    \mathrm{pr}_1\colon \Map_{\mathcal{L}(\CT_{\CB G/})}(\beta, \alpha)\times \Map_{\mathcal{L}(\CT_{\CB G/})}(\beta, \alpha) \rightarrow \Map_{\mathcal{L}(\CT_{\CB G/})}(\beta, \alpha)
    \] 
    is an equivalence for every two objects $\alpha,\beta$ of $\CT_{\CB G/}$, which implies that $\Map_{\mathcal{L}(\CT_{\CB G/})}(\beta, \alpha)$ is always either empty or contractible. Condition (2) of Theorem \ref{thm:markedcofinality} requires that any faithful map $i\colon\CB G\hookrightarrow \CB J$ is initial in $\mathcal{L}(\CT_{\CB G/})$. First we note that the commutative diagram
    \[\begin{tikzcd}
	& {\CB G} \\
	{\CB J} & {\CB J\times K} & {\CB K}
	\arrow["i"', hook', from=1-2, to=2-1]
	\arrow["{i\times j}", hook, from=1-2, to=2-2]
	\arrow["j", hook, from=1-2, to=2-3]
	\arrow[hook, from=2-1, to=2-2]
	\arrow[hook', from=2-3, to=2-2]
\end{tikzcd}\]
    shows that any two faithful maps $i\colon\CB G\hookrightarrow \CB J$ and $j\colon\CB G\hookrightarrow \CB K$ are isomorphic in $\CL(\CT_{/\CB G})$. Combining this with the observation before, we see that it suffices to show that for any $\beta\in \CT_{\CB G/}$, there exists a faithful map $i\colon \CB G\hookrightarrow \CB J$ such that $\Map_{\mathcal{L}(\CT_{\CB G/})}(i, \beta)$ is non-empty. For this we apply \Cref{lem-condition3-enough-inj} to the map $\beta$ to obtain the diagram
    \[\begin{tikzcd}
	{\CB J} & {\CB L} \\
	{\CB G} & {\CB K.}
	\arrow["{\bar{\beta}}", from=1-1, to=1-2]
	\arrow["i", hook, from=2-1, to=1-1]
	\arrow["\beta"', from=2-1, to=2-2]
	\arrow["j"', hook, from=2-2, to=1-2]
\end{tikzcd}\]
Observe that we may interpret this as a zig--zag 
    \[\begin{tikzcd}
	i & {\bar{\beta}i} & \beta
	\arrow["{\bar{\beta}}", from=1-1, to=1-2]
	\arrow["j"', hook', from=1-3, to=1-2]
\end{tikzcd}\]
    in $\CT_{\CB G/}$ where the second map is faithful. This proves that $\Map_{\mathcal{L}(\CT_{\CB G/})}(i, \beta)$ is non-empty.
\end{proof}

\begin{proof}[Proof of \Cref{thm:P_gl_sp_from_tori}]
    Since the forgetful functor $\Cat^\otimes_\infty\to \Cat$ is conservative and preserves partially lax limits (see \cite{LNP}*{Remark 5.1}), it suffices to verify that the induced functors on underlying 
    $\infty$-categories are equivalences. This now follows from combining \Cref{thm:Sp_gl_lax_lim} and \Cref{thm-tori-marked-final}.
\end{proof}

\part{Naive global 2-rings}

In this part we introduce naive global and equivariant ring spectra. We then give a precise definition of a genuine refinements for such a ring. We then introduce the notion of a naive global $2$-ring, the central concept of this work. We discuss how a naive global $2$-ring decategorifies to give a naive global ring and list some examples. Finally, we spend the remainder of the part on the proof of \Cref{thm-intro1} and some related material.

\section{Naive global rings and genuine refinements}\label{sec:genuineref}

In this section we introduce the naive analogs of global and equivariant ring spectra. We then discuss multiplicative cohomology theories and introduce the notion of genuine refinements.

Before diving into the main definitions of this section we review some background on spectrum objects in an $\infty$-category. Recall that given any presentable $\infty$-category $\mathcal{B}$, one can define the category $\Sp(\mathcal{B})$ of spectrum objects in $\mathcal{B}$. Combining \cite{HA}*{Proposition 4.8.1.17 and Example 4.8.1.23} we deduce that 
\[
\Sp(\mathcal{B}) \simeq \mathcal{B} \otimes \Sp \simeq \Fun^{\rR}(\mathcal{B}^{\op},\Sp).
\]
If $\mathcal{B}$ is presentably symmetric monoidal, then $\Sp(\mathcal{B})$ acquires a symmetric monoidal structure uniquely determined by the following universal property: for any stable and presentably symmmetric monoidal $\infty$-category $\CC$, precomposition with the suspension functor $\Sigma_+^\infty \colon \mathcal{B} \to \Sp(\mathcal{B})$ induces an equivalence
\[
\Fun^{\rL,\otimes}(\Sp(\mathcal{B}), \CC)\simeq \Fun^{\rL,\otimes}(\mathcal{B}, \CC)
\]
see \cite{GGN}*{Theorem 5.1}.

\begin{remark}\label{def:spectral_yoneda}
For a stable $\infty$-category $\CC$ and two objects $X,Y\in \CC$, we write $\map_\CC(X,Y)$ for the spectrum of maps from $X$ to $Y$. Moreover we let
\[
y \colon \CC \rightarrow \Fun^{\rR}(\CC^{\op}, \Sp), \quad X\mapsto \map_{\CC}(-,X)
\]
denote the spectral Yoneda embedding. For an arbitrary presentably symmetric monoidal $\infty$-category $\mathcal{B}$, the equivalence $\Sp(\mathcal{B}) \simeq \Fun^{\rR}(\mathcal{B}^{\op},\Sp)$ is concretely given by the assignment $X \mapsto \map_{\Sp(\mathcal{B})}(\Sigma_+^\infty(-),X)$. 
\end{remark}
We now specialize to our cases of interest. 
\begin{definition}
 We will call $\Sp(\Spcgl{\CE})$ and $\Sp(\Spc_{G})$ the $\infty$-categories of \emph{naive global spectra} and \emph{naive $G$-spectra} respectively. 
Recall that $\Sp(\Spcgl{\CE})$ and $\Sp(\Spc_G)$ both inherit symmetric monoidal structures from the cartesian monoidal structure on $\Spcgl{\CE}$ and $\Spc_G$. We call objects in $\CAlg(\Sp(\Spcgl{\CE}))$ and $\CAlg(\Sp(\Spc_G))$ \emph{(commutative) naive global rings} and \emph{(commutative) naive $G$-rings} respectively.
\end{definition}

\begin{remark}
Since in this paper we only consider commutative (i.e.~$\mathbb{E}_\infty$-)rings, we will typically drop this from the terminology. 
\end{remark}

The terminology above is intended to distinguish such objects from their genuine analogs, which contain substantially more structure.

\begin{definition}
We call an object of $\CAlg(\Sp_G)$ a \emph{(commutative) $G$-ring} and an object of $\CAlg(\Spgl{\CE})$ a \emph{(commutative) global ring}. If we wish to emphasise the distinction with naive equivariant/global rings, we may call them genuine.
\end{definition}

\begin{definition}
Let $G$ be a compact Lie group and recall that there exists a symmetric monoidal colimit preserving suspension functor $\Sigma_G^\infty \colon\Spc_G\to \Sp_G$. By the universal property of stabilization this lifts to a strong monoidal left adjoint $\Sp(\Spc_G)\rightarrow \Sp_G$, whose right adjoint we denote by $\mathbb{U}_G \colon \Sp_G\rightarrow \Sp(\Spc_G)$. As the right adjoint of a symmetric monodial functor, $\mathbb{U}_G$ is canonically lax symmetric monoidal. In particular we obtain a functor $\mathbb{U}_G\colon \CAlg(\Sp_G)\rightarrow \CAlg(\Sp(\Spc_G))$ from $G$-rings to naive $G$-rings.
\end{definition}

\begin{remark}\label{rem:non-refinement}
Intuitively, the functor $\mathbb{U}_G\colon \CAlg(\Sp_G)\rightarrow \CAlg(\Sp(\Spc_G))$ forgets the additional deloopings for representation spheres with non-trivial $G$-action encoded by the original naive $G$-ring. Alternatively, once restricted to the heart of $\Sp_G$ this functor forgets from Green functors to coefficient systems rings. Therefore, in general one may think of a naive G-ring as lacking the additive transfers contained in a (genuine) $G$-ring. 
\end{remark}

It has been crucial in equivariant homotopy theory to build and exploit the additional structure contained in a genuine $G$-spectrum, over and above that of a naive $G$-spectrum. To systematically consider this, we make the following definition:

\begin{definition}
Let $X$ be a naive $G$-ring. We say a $G$-ring $\tilde{X}\in \CAlg(\Sp_G)$ is a \emph{genuine refinement} of $X$ if there is an equivalence $\mathbb{U}_G(\tilde{X}) \simeq X$. The space of genuine refinements of $X$ is the following pullback 
\[\begin{tikzcd}
	{\mathrm{GenRef}(X)} & {\CAlg(\Sp_G)} \\
	\ast & {\CAlg(\Sp(\Spc_G))}
	\arrow[from=1-1, to=1-2]
	\arrow[from=1-1, to=2-1]
	\arrow["\lrcorner"{anchor=center, pos=0.125}, draw=none, from=1-1, to=2-2]
	\arrow[from=1-2, to=2-2]
	\arrow["{\{X\}}", from=2-1, to=2-2]
\end{tikzcd}\]
in $\Cat$. Because $\mathbb{U}_G\colon \CAlg(\Sp_G) \rightarrow \CAlg(\Sp(\Spc_G))$ is conservative, $\mathrm{GenRef}(X)$ is a space.
\end{definition}

\begin{remark}\label{rem-uniqueness}
We emphasize again that the space $\mathrm{GenRef}(X)$ of genuine refinements of a naive $G$-ring is not necessarily contractible. For a concrete example, we note that from the discussion in \Cref{rem:non-refinement}, it suffices to give two different Green functor structures on the same coefficient system of rings. consider the following $C_2$-coefficient system $R$:
\[
\begin{tikzcd}
    R(C_2/C_2) \arrow[d, bend right, "\mathrm{res}"'] \\
    R(C_2/1) \arrow[loop,in=-120,out=-60,looseness=3, "\gamma"]\arrow[u, bend right, "\mathrm{tr}"',dotted]
\end{tikzcd}
\quad =\quad 
\begin{tikzcd}
    \mathbb{F}_2[x]/x^2 \arrow[d, bend right, "x \mapsto 0"'] \\
    \mathbb{F}_2.\arrow[loop,in=-120,out=-60,looseness=3, "\mathrm{id}"]\arrow[u, bend right, "\mathrm{tr}"', dotted]
\end{tikzcd}
\]
Note that there are different choices of transfer maps $\mathrm{tr}\colon \mathbb{F}_2\to \mathbb{F}_2[x]/x^2$ such that $R$ is a Green functor; for examples $a \mapsto 0$ and $a \mapsto ax$.
\end{remark}

We may also consider a global analog of the previous definition. To do so we first note that the equivariant suspension spectrum functors assemble to give a natural transformation $\Sigma_\bullet^\infty \colon \Spc_\bullet \Rightarrow \Sp_\bullet$ of symmetric monoidal left adjoint functors defined over $\Glo{\CE}^{\op}$, see~\cite{LNP}*{Proposition 10.5}.

\begin{definition}
We define the \emph{global suspension} functor
\[
\Sigma_\gl^{\infty}\colon \Spcgl{\CE} \to \Spgl{\CE}
\] 
as the functor induced on partially lax limits by the natural transformation $\Sigma_\bullet^\infty\colon \Spc_\bullet\rightarrow \Sp_\bullet$.
\end{definition}

\begin{warning}
In \cite{Schwede18}*{Construction 4.1.7}, Schwede also constructs a suspension spectrum functor $\Sigma_+^\infty\colon \Spcgl{\CE}\rightarrow \Spgl{\CE}$. 
We warn the reader that we have not shown that this functor agrees with the functor defined above, but we do expect this to be the case. In fact we strongly suspect that the methods of \cite{LNP} suffice to prove this statement. We use the definition above because its connection to the equivariant suspension spectrum functor is more immediate and, importantly for us, coherent by definition.
\end{warning}

\begin{definition}
We write $\Omega^\infty_\bullet\colon \Sp_\bullet \Rightarrow \Spc_\bullet$ for the lax natural transformation associated to $\Sigma^\infty_\bullet$ under the second mate equivalence (\ref{second-mate-equivalence}). In particular the component $\Omega^\infty_G\colon \Sp_G\rightarrow \Spc_G$ of $\Omega^\infty_\bullet$ at $G$ is the right adjoint of $\Sigma^\infty_G\colon \Spc_G \rightarrow \Sp_G$.
\end{definition}

\begin{lemma}\label{lem:global_omega}
The global suspension functor $\Sigma_\gl^\infty$ is strong monoidal, and admits a right adjoint $\Omega_\gl^\infty\coloneqq \laxlimdag \Omega_\bullet^\infty$.
\end{lemma}

\begin{proof}
As a partially lax limit of strong monoidal functors, $\Sigma_\gl^\infty$ is automatically strong monoidal. To construct $\Omega_\gl^\infty$ and exhibit it as the right adjoint of $\Sigma_\gl^\infty$ we apply the criteria of \cite{LinskensGlobalization}*{Proposition 3.22}. For this it suffices to show that for any faithful morphism $f \colon \CB H \to \CB G$, the Beck--Chevalley natural transformation
\[
f^* \Omega^\infty_G \Rightarrow \Omega^\infty_H f^*
\]
is an equivalence. In other word, we have to check that $\Omega^\infty_\bullet$ is a natural transformation when restricted to $\Orb{\CE}^{\op}\subset \Glo{\CE}^{\op}$. To check this we can pass to total mates and instead verify that taking suspension spectra commutes with induction. This is a well-known fact, which could for instance be checked in ones preferred model of equivariant spectra.
\end{proof}

\begin{definition}
As before, the strong monoidal left adjoint $\Sigma^\infty_\gl\colon \Spcgl{\CE}\rightarrow \Spgl{\CE}$ lifts to a symmetric monoidal functor out of the stabilization of global spaces, and we write $\mathbb{U}_{\gl}\colon \Spgl{\CE}\rightarrow \Sp(\Spcgl{\CE})$ for its right adjoint. This is again lax symmetric monoidal, and so induces a functor from global rings to naive global rings.
\end{definition}

Given this we can discuss the notion of a genuine refinement of a naive global ring.

\begin{definition}\label{def-genuine-ref}
Let $X$ be a naive global ring. We say a global ring $\tilde{X}\in \CAlg(\Spcgl{\CE})$ is a \emph{genuine refinement} of $X$ if there is an equivalence $\mathbb{U}_{\gl}(\tilde{X})\simeq X$. Once again we can define the space of genuine refinements.
\end{definition}

\subsection{Multiplicative cohomology theories}

We will typically construct and compare naive global and equivariant rings via the cohomology theories they represent.

\begin{definition}
Let $\mathcal{B}$ be a presentably symmetric monoidal $\infty$-category and let $\mathcal{C}$ be a presentably symmetric monoidal stable $\infty$-category. Then a \emph{multiplicative $\mathcal{C}$-valued cohomology theory} on $\mathcal{B}$ is a lax symmetric monoidal limit preserving functor $F\colon \mathcal{B}^{\op}\rightarrow \mathcal{C}$. We write $\Fun^{\rR,\otimes\mathrm{-lax}}(\mathcal{B}^{\op},\mathcal{C})$ for the category of multiplicative cohomology theories. 
We are most interested in the following two cases:
\begin{enumerate}
\item If $\CE$ is a global family and $\mathcal{B} = \Spcgl{\CE}$, we will refer to the above functor as a multiplicative $\CE$-global $\CC$-valued cohomology theory.
\item If $G$ is a compact Lie group and $\mathcal{B}= \Spc_G$, we will refer to the above functor as a multiplicative $G$-equivariant $\CC$-valued cohomology theory.
\end{enumerate}
\end{definition}

\begin{proposition}\label{rem:mult-cohom-alg-in-spectra}
Let $\mathcal{B}$ be a presentably symmetric monoidal $\infty$-category. The assignment $X \mapsto \map_{\Sp(\mathcal{B})}(\Sigma_+^\infty(-), X)$ defines an equivalence 
\[
\CAlg(\Sp(\mathcal{B})) \simeq \Fun^{\rR,\otimes\mathrm{-\lax}}(\mathcal{B}^{\op},\Sp).
\]
\end{proposition}

\begin{proof}
By \Cref{thm:Day-convolution-is-mon-struc-on-PrL-tensor}, there is a symmetric monoidal equivalence $\mathcal{B} \otimes \Sp \simeq \Fun^\rR(\mathcal{B}^{\op},\Sp)$. Here the right hand side is symmetric monoidal by localizing the Day convolution symmetric monoidal structure, and the left hand side is symmetric monoidal as it is the tensor product of two  presentably symmetric monoidal $\infty$-categories. As mentioned there is an equivalence $\Sp(\mathcal{B})\simeq \mathcal{B} \otimes \Sp$. We observe that after precomposing with the suspension functor it agrees with $\mathcal{B}\otimes \Sigma_+^\infty \colon \mathcal{B}\otimes \Spc \to\mathcal{B}\otimes\Sp$, and so is canonically symmetric monoidal. Therefore by the universal property of $\Sp(\mathcal{B})$, the equivalence above is also symmetric monoidal. All in all we have constructed a symmetric monoidal equivalence 
\[
\Sp(\mathcal{B}) \xrightarrow{\sim} \Fun^{\rR}(\mathcal{B}^{\op}, \Sp),
\]
which by \Cref{def:spectral_yoneda} is given as in the proposition. In particular we may compute:
\[
\CAlg(\Sp(\mathcal{B})) \simeq \CAlg(\Fun^\rR(\mathcal{B}^{\op},\Sp)) \simeq \Fun^{\rR,\otimes\mathrm{-lax}}(\mathcal{B}^{\op},\Sp),
\]
where the second equivalence utilises the universal property of Day convolution (\cite{HA}*{Example 2.2.6.9}), and the fact that the category of commutative algebras in a symmetric monoidal Bousfield localization is equivalent to commutative algebras whose underlying object is local.
\end{proof}

Suppose $\mathcal{B}$ is equipped with the cartesian symmetric monoidal structure. By \cite{HA}*{Theorem 2.4.3.18} we obtain an equivalence
\[
\Fun^{\otimes\mathrm{-lax}}(\mathcal{B}^{\op},\CC) \simeq \Fun(\mathcal{B}^{\op},\CAlg(\CC))
\]
between lax symmetric monoidal functors $F\colon \mathcal{B}^{\op}\rightarrow \CC$ and ordinary functors $\tilde{F}\colon \mathcal{B}^{\op}\rightarrow \CAlg(\CC)$. Moreover, because the forgetful functor $\CAlg(\CC)\rightarrow \CC$ preserves and detects limits, $F$ is limit preserving if and only if $\tilde{F}$ is. Therefore the equivalence above restricts to an equivalence
\[
\Fun^{\rR,\otimes\mathrm{-lax}}(\mathcal{B}^{\op},\CC) \simeq \Fun^\rR(\mathcal{B}^{\op},\CAlg(\CC)),
\]
which gives an alternative definitions of multiplicative $\mathcal{C}$-valued cohomology theories on cartesian monoidal $\mathcal{B}$. 
When $\mathcal{B}$ is the $\infty$-category of $G$-spaces or of global spaces and $\CC$ is the $\infty$-category of spectra, we may combine this with the previous proposition to deduce the following result, where we write $\CAlg$ for $\CAlg(\Sp)$.

\begin{proposition}\label{prop-naive-global-rings}
For any compact Lie group $G$, the assignment $X \mapsto \map_{\Sp(\Spc_G)}(\Sigma_+^\infty(-), X)$ defines an equivalence
\[
\CAlg(\Sp(\Spc_G))\simeq \Fun^{\rR}(\Spc_G^{\op}, \CAlg).
\]
Similarly for any global family $\CE$, the assignment $X \mapsto \map_{\Sp(\Spcgl{\CE})}(\Sigma_+^\infty(-), X)$ 
defines an equivalence
\[
\CAlg(\Sp(\Spcgl{\CE}))\simeq \Fun^{\rR}(\Spcgl{\CE}^{\op},\CAlg).
\] 
\end{proposition}
Now that we have discussed naive rings and multiplicative cohomology theories we can further investigate the consequences of having a genuine refinement. 

\begin{proposition}\label{prop:family-of-eq-vs-global}
\leavevmode 
\begin{enumerate}
\item The composite 
\[
\mathbb{U}_G\colon \CAlg(\Sp_G)\rightarrow \CAlg(\Sp(\Spc_G)) \simeq \Fun^\rR(\Spc_G^{\op},\CAlg)
\] corresponds to the functor which sends a $G$-ring $X$ to the functor $\map(\Sigma^\infty_{G}(-),X)$.
The analogous statement holds in the global case. 
\item Precomposition by the lax cocone $-\sslash G\colon \Spc_G^{\op}\rightarrow \Spcgl{\CE}^{\op}$ defines an equivalence 
\[
\Phi \colon \Fun^{\rR}(\Spcgl{\CE}^{\op}, \CAlg) \xrightarrow{\simeq}\laxlimdag \Fun^\rR(\Spc_G^{\op},\CAlg).
\]
\item The restriction functor induces an equivalence $F\colon  \CAlg(\Sp(\Spcgl{\CE}))\xrightarrow{\sim} \laxlimdag \CAlg(\Sp(\Spc_G))$ which fits into the following commutative squares
\[
\begin{tikzcd}
 {\CAlg(\Spgl{\CE})}\arrow[r, "\sim"]\arrow[d, "\mathbb{U}_\gl"']& {\laxlimdag \CAlg(\Sp_G)}\arrow[d, "\laxlimdag \mathbb{U}_G"]\\
{\CAlg(\Sp(\Spcgl{\CE}))} \arrow[d, "\sim","\ref{prop-naive-global-rings}"']\arrow[r,"\sim"', "F"]& {\laxlimdag \CAlg(\Sp(\Spc_G))} \arrow[d,"\sim"',"\ref{prop-naive-global-rings}"]\\
{\Fun^\rR(\Spcgl{\CE}^{\op},\CAlg)} \arrow[r,"\sim"', "\Phi"]& {\laxlimdag \Fun^\rR(\Spc_G^{\op},\CAlg),}
\end{tikzcd}\]
where the top horizontal equivalence uses \Cref{calg-laxlim}.
\end{enumerate}
\end{proposition}

\begin{proof}
Part (1) follows from \Cref{prop-naive-global-rings} combined with the adjunction equivalence 
\[
\map_{\Sp(\Spc_{G})}(\Sigma^\infty_+(-),\mathbb{U}_G(X)) \simeq \map_{\Sp_G}(\Sigma^\infty_G(-),X).
\]
The proof of the global statement is completely analogous. The dual argument of \Cref{prop:glquotient-oplaxlim} gives part (2). 
For part (3) we start by constructing a lax symmetric monoidal functor $F\colon \Sp(\Spcgl{\CE}) \to \laxlimdag \Sp(\Spc_G)$. Using the universal property of the partially lax limits (see \cite{LNP}*{Remark 5.1}), it suffices to define a partially lax family of lax symmetric monoidal functors $F_G \colon\Sp(\Spcgl{\CE})\to \Sp(\Spc_G)$. We can then invoke \cite{Nikoluas} to see that there exists a unique lax symmetric monoidal functor $F_G$ making the following diagram commute
\[
\begin{tikzcd}
    \Sp(\Spcgl{\CE})\arrow[d,"\Omega^\infty"'] \arrow[r,"F_G"] & \Sp(\Spc_G)\arrow[d,"\Omega^\infty"] \\
    \Spcgl{\CE} \arrow[r,"\mathrm{res}_G"] & \Spc_G.
\end{tikzcd}
\]
Passing to commutative algebras and using \Cref{calg-laxlim} we obtain a functor $$F=\{F_G\}\colon\CAlg(\Sp(\Spcgl{\CE}))\to \laxlimdag \CAlg(\Sp(\Spc_G)).$$ We claim that $F$ makes the bottom square in part (3) commute. This in particular shows that $F$ is an equivalence. The commutativity follows from the series of equivalences 
\begin{align*}
\Map_{\Sp(\Spcgl{\CE})}(\Sigma_+^\infty(-\sslash G), X) & \simeq \Map_{\Spc_G}(-, \mathrm{res}_G \Omega_{-}^\infty X)\\
& \simeq \Map_{\Spc_G}(-, \Omega_{-}^\infty F_G X) \\
& \simeq \Map_{\Sp(\Spc_G)}(\Sigma_+^\infty(-), F_GX).
\end{align*}
It remains to prove that the top square in part (3) commutes. unraveling the definitions, we see that the commutativity of the square reduces to proving an equivalence $\mathbb{U}_G \mathrm{res}_G X\simeq F_G \mathbb{U}_{\gl}(X)$. Consider the following diagram:
\[
\begin{tikzcd}
 \Spgl{\CE}\arrow[dd, bend right= 110, "\Omega^\infty_\gl"']\arrow[d,"\mathbb{U}_\gl"'] \arrow[r,"\mathrm{res}_G"] & \Sp_G \arrow[d,"\mathbb{U}_G"]\arrow[dd, bend left=110, "\Omega^\infty"] \\
 \Sp(\Spcgl{\CE})\arrow[d,"\Omega^\infty_{\gl}"'] \arrow[r,"F_G"] & \Sp(\Spc_G)\arrow[d,"\Omega^\infty"] \\
 \Spcgl{\CE} \arrow[r,"\mathrm{res}_G"] &\Spc_G .
\end{tikzcd}
\]
We note that the left and right triangles commutes as we can check this after passing to left adjoints, where it holds by definition. Moreover the other square also commutes by the definition of $\Omega_\gl^\infty$. Then the commutativity of the top square follows by the universal property of $\Sp(\Spc_G)$ as we can check this after precomposition with $\Omega^\infty_{-}$.
\end{proof}

\begin{corollary}
Suppose $\tilde{X}\in \Spgl{\CE}$ is a genuine refinement of a naive global ring $X$. Then $\mathrm{res}_G X$ is a genuine refinement of the naive $G$-ring $\mathrm{res}_G X$.\qedhere
\end{corollary}

\begin{proof}
    This follows from \Cref{prop:family-of-eq-vs-global} (3).
\end{proof}

\section{Naive global 2-rings}

In this section we introduce the notion of a naive global 2-ring, which is the central concept of this work, and list some key examples.  

The starting point of our discussion is a functor $\CR\colon \Spcgl{\CE}^{\op}\rightarrow \PrL$. We can now perform two different constructions to $\CR$. 

\begin{definition}
    The \emph{decategorification} of $\CR$ is the composite 
    \[
    \decatcohsp{\gl}{-}{\CR} \colon \Spcgl{\CE}^{\op} \xrightarrow{\CR}\PrL \xrightarrow{\mathrm{End}_\bullet(\1)} \CAlg
    \]
    where the global section functor is given by $\mathcal{C}\mapsto \map_{\mathcal{C}}(\1,\1)$.
\end{definition}

The second construction that we can perform is given by the following result, which we state here but defer proving until later sections, as its proof requires substantial additional work.
\begin{theorem}\label{thm:Unraveling}
    Consider a functor $\CR \colon \Spcgl{\CE}^{\op} \to \PrL$. Then for all $G\in \CE$, there exists a functor as described in \Cref{def-unravel-incoh}
    \[
    \decatcoh{G}{-}{\CR}\colon \Spc_G^{\op}\rightarrow \CAlg(\CR_{\CB G})
    \] 
     such that for every group homomorphism $\alpha \colon  H \to G$ there exists a natural transformation filling the square
    \begin{equation}
        \begin{tikzcd}[column sep = huge]
            {\Spc_G^{\op}} &  {\CAlg(\CR_{\CB G})}\\
            {\Spc_H^{\op}} & {\CAlg(\CR_{\CB H})}
            \arrow["{\alpha^*}"', from=1-1, to=2-1]
            \arrow["{\decatcoh{G}{-}{\,\CR}}", from=1-1, to=1-2]
            \arrow["{\decatcoh{H}{-}{\,\CR}}"', from=2-1, to=2-2]
            \arrow["{\alpha^*}", from=1-2, to=2-2]
            \arrow[shorten <=18pt, shorten >=18pt, Rightarrow, from=2-1, to=1-2, "Q_{\alpha}"]
        \end{tikzcd}
    \end{equation}
    This data is coherently functorial: this is encoded in the existence of a functor
    \[
    \decatcoh{\bullet}{-}{\CR}\colon \Glo{\CE}^{\op}\rightarrow \Fun^{\mathrm{oplax}}([1],\tCat^{\otimes,\lax})
    \] extending the assignment above. 
\end{theorem}
 We call the output $\decatcoh{\bullet}{-}{\CR}$ of the previous theorem the \emph{unraveling} of $\CR$. Our consideration of such families is inspired by the perspective on equivariant elliptic cohomology given in \cite{GKV95}.
\begin{definition}\label{def-unravel-incoh}
The construction of this data is simple to explain in an incoherent manner: 
\begin{enumerate}
    \item\label{item-unravel-obj} Given a group $G\in \CE$ and a $G$-space $X$, which we regard as a global space $f\colon X\sslash G \rightarrow \CB G$ equipped with a faithful map to $\CB G$, we define $\decatcoh{G}{X}{\CR} \coloneqq f_* \1_{X\sslash G}$ to be the pushforward of the unit of $\CR_{X\sslash G}$ to $\CR_{\CB G}$. 
    \item\label{item-unravel-morph} Given a map of $G$-spaces $X\to Y$, encoded by a commutative diagram 
\[\begin{tikzcd}	X\sslash G \arrow[rr, "h"] \arrow[rd, hook, "f"']& &  Y\sslash G \arrow[ld, hook, "g"] \\
            & \CB G &
        \end{tikzcd} \]
in $\Spcgl{\CE}$, we define $\decatcoh{G}{h}{\CR}$ to be the map        
\[
g_*\1_{\CR_{Y\sslash G}}\xrightarrow{g_*(\lax_{h})} g_* h_* \1_{\CR_{X\sslash G}} \simeq f_* \1_{\CR_{X\sslash G}},
\]
where $\lax_{h}$ is the lax unit map of the lax monoidal functor $h_*$. 
\item\label{item-unravel-nat-tran} Finally, given a group homomorphism $\alpha\colon H\rightarrow G$, we define the transformation
\[
Q_\alpha\colon \alpha^*\decatcoh{G}{-}{\CR} \Rightarrow \decatcoh{H}{\alpha^*(-)}{\CR}
\] to be the mate of a natural transformation 
\[
T_\alpha\colon \decatcoh{G}{\alpha_!(-)}{\CR}\Rightarrow \alpha_*\decatcoh{H}{-}{\CR}.
\] 
We define $T_\alpha$ at an $H$-space $Z$, corresponding to a faithful map $f\colon Z\sslash H\hookrightarrow \CB H$, by consider the factorization of $\alpha f$ into a \quotientmap{} $h$ followed by an injection $\alpha_! f$:
    \[
    \begin{tikzcd}
            Z\sslash H & {\alpha_!Z\sslash H} \\
            {\CB H} & {\CB G,}
            \arrow["\CB \alpha", from=2-1, to=2-2]
            \arrow["f"', hook, from=1-1, to=2-1]
            \arrow["\alpha_! f", hook, from=1-2, to=2-2]
            \arrow["h",two heads, from=1-1, to=1-2]
        \end{tikzcd}
        \]
        and then defining $T_\alpha$ as the map
        \[
        \decatcoh{G}{\alpha_! Z\sslash H}{\CR} \simeq (\alpha_! f)_*(\1_{\CR(\alpha_! Z\sslash H)}) \xrightarrow{(\alpha_! f)_* (\lax_{h})} (\alpha_! f)_* h_* \1_{\CR_{Z\sslash H})} \simeq \alpha_* f_*\1_{\CR_{Z\sslash H}}  \simeq \alpha_* \decatcoh{H}{Z}{\CR}.
        \]
        \end{enumerate}
\end{definition}

It will not surprise the reader that carefully dealing with all of the coherences implicit in this construction is a nontrivial task, which we carry out in \Cref{sec:decategorify} and \Cref{sec:unraveling}. 

We want to think of the functor $\CR$ as the categorification of the naive global ring $\decatcohsp{\gl}{-}{\CR}$, and we want to think of the unraveling $\decatcoh{\bullet}{-}{\CR}$ as a family of equivariant cohomology theories equipped with change of group transformations which are canonically associated to $\CR$. However, in this generality, the functors $\decatcoh{G}{-}{\CR}$ and $\decatcohsp{\gl}{-}{\CR}$ will not be limit-preserving, so we can fix this with the following definition.

\begin{definition}
A \emph{naive global 2-ring} is a functor $\CR\colon \Spcgl{\CE}^{\op}\rightarrow \PrL$ whose associated decategorications $\decatcohsp{\gl}{-}{\CR}$ and $\decatcoh{G}{-}{\CR}$ are limit-preserving for all $G \in \CE$. A morphism of naive global $2$-rings is a symmetric monoidal natural transformation $F \colon \CR \to \CR'$. The $\infty$-category of naive global $2$-rings is denoted by $\NaiveTwoGlRing{\CE}$.
\end{definition}

\begin{warning}
We note that the terminology 2-ring is often reserved for small stable categories, see for example \cite{Ma16}*{Section 2.2}. We find it simpler to work in a context with arbitrary colimits, but we make no essential use of this fact. For example we could always index cohomology theories over finite (equivariant or global) spaces, removing the necessity for filtered colimits in the target entirely.
\end{warning}
\begin{example}\label{ex:decat}
By definition the decategorication of a naive global $2$-ring defines a functor 
\[
\NaiveTwoGlRing{\CE} \to \Fun^{\rR}(\Spcgl{\CE}^{\op},\CAlg) \simeq \CAlg(\Sp(\Spcgl{\CE})), \quad \CR \mapsto \decatcohsp{\gl}{-}{\CR}
\]
to naive global rings, where we used \Cref{prop-naive-global-rings}.
\end{example}
A good source of examples is given by the following result.
\begin{lemma}\label{lem-limit_naive}
  Any limit-preserving functor $\CR\colon \Spcgl{\CE}^{\op}\rightarrow \PrL$ defines a naive global 2-ring.  
\end{lemma}
\begin{proof}
    Since the global section functor preserves limits, it is clear that the decategorication $\decatcohsp{\gl}{-}{\CR}$ is limit-preseving. Let us now argue that $\decatcoh{G}{-}{\CR}$ preserves limits, which we can do after forgetting commutative algebra structures. Now consider $X$ in $\Spc_G$ and suppose $X$ is equivalent to the colimit of G-spaces $X_i$. Because $-\sslash G\colon \Spc_G\rightarrow \Spcgl{\CE}$ preserves colimits we find that $Y:=X\sslash G$ is a colimit of the global spaces $Y_i:=X_i\sslash G$. By our assumption on $\CR$ we have an equivalence $\CR_{Y}\simeq \lim \CR_{Y_i}.$
In particular, writing $\pi_i\colon Y_i\rightarrow Y$ for the components of the universal cocone, given any object $B$ in $\CR_{X\sslash G}$ we obtain an equivalence
\[B\simeq \lim_i (\pi_i)_*\pi_i^* B\] in $\CR_Y$. To see this we may for example compare universal properties:
\begin{align*}
    \Map_{\CR_Y}(A,B) &\simeq \lim_i\Map_{\CR_{Y_i}}(\pi_i^* A,\pi_i^*B) \\
    &\simeq \lim_i\Map_{\CR_{Y}}(A, (\pi_i)_*\pi_i^* B) \\ 
    &\simeq \Map_{\CR_{Y}}(A, \lim_i (\pi_i)_*\pi_i^* B). 
\end{align*} 
Finally note that the canonical faithful map $f\colon Y \rightarrow \CB G$ is induced by the universal property of a colimit by the maps $f_i\colon Y_i \rightarrow \CB G$. We can now calculate
\begin{align*}
\decatcoh{G}{X}{\CR} = f_* \1_{\CR_{X\sslash G}} \simeq f_*  \lim_i (\pi_i)_*\pi_i^* \1_{\CR_{X\sslash G}} &\simeq \lim_i f_* (\pi_i)_* \1_{\CR_{{X_i\sslash G}}} \\
&\simeq \lim_i (f_i)_* \1_{\CR_{{X_i\sslash G}}} \simeq \lim_i \decatcoh{G}{X_i}{\CR}.\qedhere
\end{align*}
\end{proof}

As the main theme of this work, we will see that a naive global $2$-ring in fact endows its decategorification with significantly more structure. Before we dive into this, let us first give some examples of \weakcats{}.

\begin{example}\label{ex:borel}
Given any presentably symmetric monoidal stable $\infty$-category $\CC$, we may define a functor \[\CC_{\mathrm{bor}}\colon {\Glo{}}^{\op}\rightarrow \PrL,\quad \CB G\mapsto \CC^{B G}.\] The limit extension of this functor defines a naive global 2-ring, which we call a \emph{Borel} 2-ring.
\end{example}

\begin{example}\label{ex:genuine_spectra}
Consider the functor ${\Glo{}}^{\op} \to \PrL$ which sends $\CB G$ to the $\infty$-category of genuine $G$-spectra $\Sp_G$, and a morphism $\CB\alpha \colon \CB H \to \CB G$ to the restriction-inflation functor $\alpha^* \colon \Sp_G \to \Sp_H$. Limit extending we obtain a naive global 2-ring 
\[\Sp_\bullet \colon \Spc_{\gl}^{\op} \to \PrL.\] The resulting decategorified naive global ring is equivalent to the underlying naive global ring associated to the global sphere spectrum $\mathbb{S}_\gl$.
\end{example}

\begin{example}\label{ex:global-spectra-cat-cohom}
We may also consider the limit extension of the functor
\[
\Sp_{\bullet\text{-}\gl}\colon \Glo{\mathrm{fin}}^{\op}\to \PrL, \quad \CB G\mapsto \Sp_{G\text{-}\gl}
\]
defined for the family of finite groups in \cite{CLL_Global}. This functor sends $\CB G$ to the $\infty$-category of $G$-global spectra in the sense of \cite{LenzGglobal}. Once again the associated decategorified global ring is equivalent to the global sphere spectrum.
\end{example}

\begin{example}\label{ex:cat_gl_cohom_of_gl_spectrum}
Generalizing \Cref{ex:genuine_spectra}, one can associate to any $R\in \CAlg(\Spgl{\CE})$ the functor 
\[
\Mod_{R_\bullet}\colon \Glo{\CE}^{\op}\rightarrow \PrL, \quad \CB G \mapsto \Mod_{\mathrm{res}_G R}(\Sp_G).
\] 
We refer to~\cite{LNP}*{Theorem 5.10} for a construction of this functor. Taking the limit extension we obtain a naive global 2-ring. Its decategorification agrees with the underlying naive global ring of $R$.
\end{example}

The following two examples form the motivation for this paper. We will introduce two multiplicative global cohomology theories, and then categorifications thereof. Here we will be brief; we refer the reader to \Cref{examples} for a more extensive account of these examples.

\begin{example}\label{tempered cohomology}
Let $\bG$ be a preoriented $\rP$-divisible group over a commutative ring spectrum $R$, in the sense of \cite{Ell3}*{Definition 2.6.1}. Given this data Lurie constructs~\cite{Ell3}*{Construction 4.0.3} a multiplicative global cohomology theory
\[
R_{\bG}^\bullet\colon \Spcgl{\fab}^{\op}\to \CAlg,
    \]
    called $\bG$-tempered cohomology. Here $\fab$ refers to the family of finite abelian groups. Given a finite abelian group $H$, there is a non-canonical equivalence 
    \[
    \mathrm{Spec}(R_{\bG}^{\CB H})\simeq \bG[\widehat{H}]
    \]
    between the spectrum of $R_{\bG}^{\CB H}$ and the $\widehat{H}$-torsion points in $\bG$, where $\widehat{H}$ denotes the Pontryagin dual of $H$.
\end{example}

Tempered cohomology is categorified by the notion of tempered local systems.

\begin{example}
Let $\bG$  be a preoriented $\rP$-divisible group over $R$. As discussed in ~\cite{Ell3}*{Remark 5.2.11}, there exists a limit preserving functor
\[
\LocSys_\bG \colon \Spcgl{\fab}^{\op}\rightarrow \widehat{\Cat}, \quad \sX \mapsto  \LocSys_\bG(\sX)
\] 
which sends a global space $\sX$ to the $\infty$-category of $\bG$-tempered local systems on $\sX$, and a morphism of global spaces $f\colon \sX \to \sY$ to the pullback functor $f^* \colon \LocSys_{\bG}(\sY) \to \LocSys_{\bG}(\sX)$, see \Cref{sec:tempered} for more details. If in addition $\bG$ is \emph{oriented} in the sense of \cite{Ell2}*{Proposition 4.3.23}, then \Cref{megaprop} below implies that $\LocSys_\bG$ is a naive global 2-ring.
\end{example}

So far we have only given examples of naive global $2$-rings coming from limit-preserving functors. The following example motivates our more general definition. 

\begin{example}\label{ex:cat_ell_coh}
Let $\sfS$ be a nonconnective spectral Deligne-Mumford stack and let $\bE$ be an oriented elliptic curve in nonconnective spectral Deligne-Mumford stacks over $\sfS$, in the sense of~\cite{GM20}*{Definition 3.1}. Following \cite{GM20}, we can associate to $\bE$ a global cohomology theory
\[ 
\decatcohsp{\gl}{-}{\bE} \colon \Spcgl{\ab}^{\op}\rightarrow \CAlg(\Sp),\quad \CB A \mapsto \Gamma(\bE[\widehat{A}], \CO_{\bE[\widehat{A}]}).
\]
Similarly to the case of tempered cohomology, if $A$ is a compact abelian Lie group, we denote by $\bE[\widehat{A}]$ the $\widehat{A}$-torsion points of $\bE$.
\end{example}

\begin{example}\label{ex:QCoh_cohom}
As we will explain in \Cref{sec:oriented_ab_groups}, there is also a way to associate to $\bE\to \sfS$ a naive global 2-ring:
\[
\CQ_\bullet^\bE \colon \Spcgl{\ab}^{\op}\rightarrow \PrL.
\] 
Morally this comes about by taking the categories of quasi-coherent sheaves on a suitable geometric extension of $\bE$ to compact $\ab$-global spaces. This assignment will however not be limit preserving, and so this is an example of a naive global 2-ring not coming from \Cref{lem-limit_naive}. Considering the natural equivalence $\mathrm{End}_{\QCoh(X)}(\1)\simeq \Gamma(X,\mathcal{O}_X)$, we see that $\CQ_\bullet^\bE$ categorifies $\decatcohsp{\gl}{-}{\bE}$. 
\end{example}

\section{Relative global sections}\label{sec:decategorify}

In this section we construct the functoriality which is required to prove \Cref{thm:Unraveling}. To motivate the steps of our construction we first sketch a construction of the functor \[\decatcoh{G}{-}{\CR}\colon \Spc_G^{\op}\rightarrow \CAlg(\CR_{\CB G}), \quad [X\sslash G \xrightarrow{p} \CB G] \mapsto p_*\1_{X\sslash G}.\] 
\begin{enumerate}
\item First we claim that there exists an oplax cone
\[\begin{tikzcd}
& {\CAlg(\CR_{X\sslash G})} \\
\ast \\
& {\CAlg(\CR_{Y\sslash G)}}
\arrow["{f_*}", from=1-2, to=3-2]
\arrow["{\1_{X\sslash G}}", from=2-1, to=1-2]
\arrow[""{name=0, anchor=center, inner sep=0}, "{\1_{Y\sslash G}}"', from=2-1, to=3-2]
\arrow[shorten <=12pt, shorten >=15pt, Rightarrow, from=0, to=1-2]
\end{tikzcd}\] which sends a $G$-space $X$ to the unit $\1_{X\sslash G}$ of the symmetric monoidal $\infty$-category $\CR_{X\sslash G}$. The triangles are filled by the map $\1_{Y\sslash G}\rightarrow 
f_*\1_{X\sslash G}$ coming from the lax monoidal structure of $f_*$.
\item Next we observe that because the $\infty$-category of $G$-spaces admits a final object, there exists a unique cone $\CAlg(\CR_{\bullet \sslash G})\Rightarrow \Delta(\CAlg(\CR_{\CB G}))$ which sends a $G$-space $p\colon X\sslash \CB G \rightarrow \CB G$ to the functor $p_*\colon \CAlg(\CR_{X\sslash \CB G})\rightarrow \CAlg(\CR_{\CB G})$. 
\item Given the previous two constructions, we define the functor $\decatcoh{G}{-}{\CR}$ by pasting the previous two oplax natural transformations together
\[
\Delta(\ast)\xRightarrow{\1_{\bullet}} \CAlg(\CR_{\bullet\sslash G}) \xRightarrow{{\color{white}\1_{\bullet}}} \Delta(\CAlg(\CR_{\CB G})),
\]
using the identifications of functors $\Spc_G^{\op}\rightarrow \CAlg(\CR_{\CB G})$ with $\Spc_G$-shaped oplax cones $\Delta(\ast)\Rightarrow \Delta(\CAlg(\CR_{\CB G}))$, see \Cref{ex:functors_are_lax_cones}. 
\end{enumerate}

Clearly the resulting functor will send $p\colon X\sslash G\to \CB G$ to $p_* \1_{X\sslash G}$ and so agrees with (\ref{item-unravel-obj}) of \Cref{def-unravel-incoh}. Similarly, considering the definition one sees that it will agree on morphisms with the desideratum (\ref{item-unravel-morph}) of \Cref{def-unravel-incoh}.

We will construct the oplax transformations sketched in step (1) and step (2) above in \Cref{subsec:unit-sec} and \Cref{sec:pushforward_sections}, respectively. In both cases our task is made more difficult by the fact that we require a construction which is functorial enough to allow us to deduce \Cref{thm:Unraveling}. Finally in \Cref{subsec:pushforward} we combine these two constructions together, as in step (3).

\subsection{Constructing unit sections}\label{subsec:unit-sec}

In this section we construct the oplax cone $\1_\bullet\colon 
\Delta(\ast)\Rightarrow \CR_{-\sslash G}$, in the guise of a section of the cartesian unstraightening of the functor $\CAlg(\CR_{-\sslash G})$. We begin with a simplification of the task: Note that the unit $\1_{\CR_X}$ is the initial object of $\CAlg(\CR_X)$. In particular given $f\colon X\rightarrow Y$ in $\Spcgl{\CE}$ the map $\lax_f\colon \1_{\CR_Y}\rightarrow f_* \1_{\CR_X}$ can be characterized as the unique map of commutative algebras from $\1_{\CR_Y}$ to $f_*\1_{\CR_X}$. Therefore the construction of a unit section as above is subsumed by the construction of initial object sections.

\begin{definition}\label{def:Cart}
    We write $\Cart$ for the subcategory of $\Ar(\Cat)$ spanned on objects by the cartesian fibrations and on morphisms by those squares
    \[
    \begin{tikzcd}
        X \arrow[r, "F"] \arrow[d, "p"] & Y \arrow[d, "q"] \\
        S \arrow[r, "f"]							& T
    \end{tikzcd}
    \] such that $F$ sends $p$-cartesian edges to $q$-cartesian edges. We define $\Cart^{(\emptyset)}$ to be the full subcategory of $\Cart$ on those cartesian fibrations $p\colon X\rightarrow S$ such that for each object $s\in S$ the fiber $X_s$ of $p$ over $s$ admits an initial object.
\end{definition}

\begin{proposition}\label{s-functor}
    The inclusion $\Cart^{(\emptyset)}\subset \Ar(\Cat)$ factors through the subcategory $\Fun_{\rR}^{\oplax}([1],\mathbf{\Cat})$. In particular postcomposing with the first mate equivalence (\Cref{first-mate-equivalence}) we obtain a functor
    \[s_\emptyset\colon \Cart^{(\emptyset)} \rightarrow \Fun_\rL^{\lax}([1],\mathbf{\Cat}).\]
\end{proposition}

\begin{proof}
    Consider a cartesian fibration $p\colon X\rightarrow S$, which classifies a functor $F\colon S^{\op}\rightarrow \Cat$. It will suffices to show that $p\colon X \to S$ has a left adjoint. We claim that the object $(\emptyset,s)$ is a left adjoint object of $s\in S$, where $\emptyset$ is the initial object of $X_s$. To see this consider the map  
    \[
    \phi \colon \Map_X((\emptyset,s),(x,t))\rightarrow \Map_S(s,t)
    \] 
    induced by $p$.	By \cite{HTT}*{Proposition 2.4.4.2} the fiber over $f\colon s\to t$ is equivalent to $$\Map_{X_s}(\emptyset, F(f)(x))\simeq \ast.$$ Thus $\phi$ is an equivalence.
\end{proof}

\begin{remark}\label{s-section}	
    The argument given in the proof of \Cref{s-functor} also shows that the left adjoint $s$ to $p\colon X\rightarrow S$ is fully faithful. In particular the composite $ps$ is canonically equivalent to the identity on $S$ via the unit of the adjunction $s\dashv p$, and so $s$ is canonically a section of $p$.
\end{remark}

\begin{example}\label{ex:unit-sections}
    Consider a cartesian fibration $p\in \Cart^{(\emptyset)}$, which classifies a functor $F\colon S^{\op}\rightarrow \Cat$, and let $s$ be a section constructed by this procedure, that is $s=s_{\emptyset}(p)$. Then $s$ may be interpreted as oplax cones $\ast\Rightarrow F$, see dual of \Cref{rem:laxlim_sect}. The component at $i\in \CI$ is given by the functor $\{\emptyset_{F(i)}\}\colon \ast \rightarrow F(i)$. If we begin with a functor $F\colon \CI\rightarrow \Cat^{\otimes,\lax}$, then applying the previous construction to the functor $\CAlg(F)\colon \CI \rightarrow \Cat^{(\emptyset)}$ we obtain a lax cocone which we will denote by $\1_\bullet\colon \ast\Rightarrow \CAlg(F)$.
\end{example}

\subsection{Cartesian transport to the initial object}\label{sec:pushforward_sections}
Having constructed the unit section, we now build the natural transformation $\CAlg(\CR_{\bullet\sslash G})\Rightarrow \Delta(\CAlg(\CR_{\CB G}))$. Note that such a natural transformation unstraightens to a functor $\Unct{\CAlg(\CR_{\bullet\sslash G})}\rightarrow \CAlg(\CR_{\CB G})\times \Spc_G^{\op}$, which by adjunction is in turn equivalent to a functor $\Unct{\CAlg(\CR_{\bullet\sslash G})}\rightarrow \CAlg(\CR_{\CB G})$. It in this guise that we will build the natural transformation above. Of course we will define it more generally for any functor $F\colon \CI\rightarrow \Cat$ such that $\CI$ admits a final object. In this case it will be a functor with signature $\Unct{F}\rightarrow F(\ast)$.

To be able to deduce \Cref{thm:Unraveling} we will require the functor above to itself be functorial in both natural transformations $F\Rightarrow G$ and precomposition by an \emph{arbitrary} functor $\CJ\rightarrow \CI$, not just one which preserves the final object. This is the case because the induction functors generally do not preserve the final object. To make this precise we introduce the following definition:

\begin{definition}
Let $\Cart_{(\emptyset)}$ be the \emph{full} subcategory of $\Cart$ spanned by those cartesian fibrations $p\colon X\rightarrow \CI$ such that $\CI$ has an initial object.
\end{definition}

Note that the base of a cartesian fibration $p\colon X\rightarrow\CI$ is given by the opposite of the source of its straightening $\CI^{\op}\to\Cat$, explaining the switch from final to initial objects. With the previous definition, we may state the goal of this section.

\begin{proposition}\label{prop:transport_to_initial_object}
There exists a functor $\Cart_{(\emptyset)}\rightarrow \Fun([1],\Cat)$ which sends a cartesian fibration $p\colon X\rightarrow \CI$ to the functor
\[
\Tr_\emptyset\colon X\rightarrow X_\emptyset, \quad (x,i)\mapsto F(\emptyset \rightarrow i)(x).
\]
\end{proposition}

To prove this statement we will appeal to the functoriality of cartesian lifts in a cartesian fibration. There is a convenient formulation of this which goes via the notion of free fibrations, in the sense of \cite{GHN}, which we recall now.

\begin{proposition}
The inclusion $\Cart \subseteq \Ar(\Cat)$ admits a left adjoint $\fF^{\cart}\colon \Ar(\Cat) \to \Cart$.
\end{proposition}

\begin{proof}
Both $\Cart$ and $\Ar(\Cat)$ are cartesian fibrations over $\Cat$, and the inclusion is clearly a map of cartesian fibrations. Therefore by \cite{HA}*{Proposition 7.3.2.6} it suffices to prove that the restriction of the inclusion to each fiber 
\[\mathrm{incl} \colon \Cart(\CI)\rightarrow \Cat_{/\CI}\]
has a left adjoint. This is precisely \cite{GHN}*{Theorem 4.5}. 
\end{proof}

\begin{definition}
    Concretely, the functor $\fF^{\cart}$ sends a functor $p\colon X \to \CI$ to the cartesian fibration whose total space is given by the following pullback square 
    \[
    \begin{tikzcd}
        \fF^{\cart}_{\CI}(X) \arrow[r] \arrow[d] & \mathrm{Ar}(\CI) \arrow[d,"\mathrm{ev}_1"]\\
        X \arrow[r,"p"] & \CI.
        \arrow["\lrcorner"{anchor=center, pos=0.125}, draw=none, from=1-1, to=2-2]
    \end{tikzcd}
    \]
    We view $\fF_{\CI}^{\cart}(X)$ as living over $\CI$ via the composite $\fF^\cart_{\CI}(X) \to \mathrm{Ar}(\CI) \xrightarrow{\mathrm{ev}_0} \CI$. When the base $\infty$-category $\CI$ is clear from the context, we will drop it from the notation and simply write $\fF^\cart(X)$. We refer to $\fF^{\cart}(X)\rightarrow \CI$ as the \emph{free fibration} of $p$.
\end{definition}

\begin{definition}
    We write \[\Theta \colon \Cart\rightarrow \Fun([1]\times[1],\Cat)\] for the functor which sends a cartesian fibration $p\colon X \rightarrow S$ to the counit of the adjunction $\fF^{\cart}(X)\dashv \mathrm{incl}$,
    \[\begin{tikzcd}
        {\fF^{\cart}(X)} & X \\
        \CI & \CI.
        \arrow[from=1-1, to=2-1]
        \arrow[Rightarrow, no head, from=2-1, to=2-2]
        \arrow["p", from=1-2, to=2-2]
        \arrow["\Tr", from=1-1, to=1-2]
    \end{tikzcd}\] We write $\Tr\colon \Cart \rightarrow \Ar(\Cat)$ for the projection to the top of the square. Write $G\colon S^{\op}\rightarrow \Cat$ for the straightening of $p\colon X\rightarrow S$. Unwinding the definitions we find that $\Tr$ sends the pair $(x\in X_i, f\colon i\rightarrow j)$ to $G(f)(x)\in X_j\subset X$. The functor $\Tr$ encodes the functoriality of cartesian lifts. Therefore we call it the \emph{cartesian transport functor}.
\end{definition}

We will now apply the cartesian transport functor to construct the functor $\Tr_\emptyset\colon \Unct{F}\rightarrow F(\emptyset)$. Recall that we will require this construction to be functorial in maps in the base which do not necessarily preserve the initial object. Nevertheless we will begin by assuming that it does, and then later reduce to this case.

\begin{construction}
    Let $\Cat_{\emptyset}$ be the subcategory of $\Cat$ spanned by those $\infty$-categories which have initial objects and those functors which preserve the initial object. We build a functor $\Cat_{\emptyset} \rightarrow \Fun([1],\Cat)$ as follows: First consider the functor $\Cat \rightarrow \Fun([1],\Cat)$ which sends an $\infty$-category $\CC$ to the functor $\ev_1 \colon \CC\times [1] \rightarrow \CC$. Now note that when $\CC$ has an initial object this functor admits a right adjoint, given by the unique functor 
    \[
    c_\emptyset \colon \CI\rightarrow \mathrm{Ar}(\CI),\quad x\mapsto [\emptyset \rightarrow x].
    \] Therefore we may pass through the first mate equivalence (\ref{first-mate-equivalence}) to obtain a functor 
    \[
    c_\emptyset \colon \Cat^\ast \rightarrow \Fun^\lax([1],\Cat).
    \] 
    A quick check shows that the Beck--Chevalley transformations are in fact equivalences, and so this functor factors through $\Fun([1],\Cat)$.
\end{construction}

\begin{construction}
    Consider the subcategory $\Cart_{\emptyset} = \Cart\times_{\Cat} \Cat_{\emptyset}$ of $\Cart$ spanned by those cartesian fibrations whose base admits an initial object and those maps of cartesian fibrations such that the functor in the base preserves the initial object. We build a functor 
    $\Cart_{\emptyset} \rightarrow \Fun([1]\times\Lambda^2_2,\Cat)$ by sending an object $p\colon X\rightarrow \CI$ to the diagram
    \[\begin{tikzcd}
        \CI & \CI & X \\
        {\mathrm{Ar}(\CI)} & \CI & X.
        \arrow["{c_\emptyset}", from=1-1, to=2-1]
        \arrow[Rightarrow, no head, from=1-1, to=1-2]
        \arrow["p", from=1-3, to=1-2]
        \arrow["p", from=2-3, to=2-2]
        \arrow[Rightarrow, no head, from=1-2, to=2-2]
        \arrow[Rightarrow, no head, from=1-3, to=2-3]
        \arrow["{\ev_1}"', from=2-1, to=2-2]
    \end{tikzcd}\]
    Taking the limit along $\Lambda^2_2$ results in a functor $\Psi\colon \Cart_{\emptyset} \rightarrow \Fun([1],\Cat)$. Note that the composite $\ev_1\circ \Psi$ is equivalent to the restriction of the functor $\fF^\cart\colon \Cart\rightarrow \Cat$ to $\Cart_{\emptyset}$ and $\ev_0\circ \Psi$ is equivalent to the identity functor on $\Cart_{\emptyset}.$ In particular $\Psi$ defines a natural transformation from the identity of $\Cart_{\emptyset}$ to $\fF^{\cart}$. 
\end{construction}

\begin{construction}
    As observed before, the source of $\Tr$ agrees with the target of $\Psi$, and so we may paste them to obtain a new functor \[\Tr_{\emptyset} \colon \Cart_{\emptyset} \rightarrow \Fun([1],\Cat).\] By definition $\Tr_{\emptyset}$ sends a cartesian fibration $p\colon X \to \CI$ to the functor:
    \[
    \Tr_{\emptyset}(X)\colon X \xrightarrow{(\id,c_\emptyset)} \fF^\cart(X) \xrightarrow{\Tr} X.
    \] 
    So we find that $\Tr_{\emptyset}(X)(x,i) \simeq (\emptyset \to i)_*(x)$. In particular the composite factors through the fiber $X_\emptyset$. We will implicitly restrict the target of the functor to the fiber over the final object from here on forward. Similarly applying this functor to a map 
    \[
    \begin{tikzcd}
        X \arrow[r, "H"] \arrow[d, "p"]& Y \arrow[d,"q"] \\
        \CI	\arrow[r, "h"] &	\CJ
    \end{tikzcd}
    \]in $\Cart_{\emptyset}$ returns a commutative square 
    \[\begin{tikzcd}
        X & {X_\emptyset} \\
        Y & {Y_\emptyset}.
        \arrow["\Tr_{\emptyset}(Y)"', from=2-1, to=2-2]
        \arrow["\Tr_{\emptyset}(X)", from=1-1, to=1-2]
        \arrow["H(\emptyset)", from=1-2, to=2-2]
        \arrow["H"', from=1-1, to=2-1]
    \end{tikzcd}\]
\end{construction}

We have built the functor $\Tr_\emptyset$ of \Cref{prop:transport_to_initial_object}. However we have so far only exhibited its functoriality in natural transformations and maps in the base which preserve the initial object. Next we explain how to proceed in the case that the map in the base does \emph{not} necessarily preserve the final object. 

\begin{construction}\label{cons-colimit-diagram}
    Consider the functor
    \[(-)^{\ltri}\colon \Cat\rightarrow \Cat, \quad \CI\mapsto \CI^{\ltri}\coloneqq \{\pt\}\ast \CI,\] where $\ast$ refers to the join of two $\infty$-categories. Note that including the second component of the join induces a natural transformation $b\colon \id\Rightarrow (-)^{\ltri}$.
    Consider the induced natural transformation 
    \[b^*\colon \Fun((-)^{\ltri},\Cat) \Rightarrow \Fun(-,\Cat)\] and note that each restriction functor admits a right adjoint, given by right Kan extension. Passing through the first mate equivalence (\ref{first-mate-equivalence}) we obtain an oplax natural transformation
    \[\Fun((-),\Cat) \Rightarrow \Fun((-)^{\ltri},\Cat)\] which evaluated at an $\infty$-category $\CI$ is given by right Kan extension along $b\colon \CI\rightarrow \CI^{\ltri}$. Observe that the right Kan extension of a diagram $F\colon \CI\rightarrow \Cat$ along $b$ is given by the universal limit cone $\CI^{\rtri}\rightarrow \Cat$ associated to $F$. Such an oplax natural transformation induces an endofunctor 
    \[\Phi\colon \Cart \rightarrow \Cart\] of the cartesian unstraightening of the functor $\Fun((-)^{\op},\Cat)$.
\end{construction}

\begin{construction}
    Suppose we now restrict $\Phi$ to the \textit{full} subcategory $\Cart_{(\emptyset)}\subseteq \Cart$ spanned by those fibrations whose base has an initial object. We emphasize that the functors do not necessarily preserve the initial object. The crucial observation, however, is that the induced functor $h^{\ltri}\colon \CI^\ltri\rightarrow \CJ^\ltri$ does preserve the initial object. Therefore $\Phi$ restricts to a functor 
    \[\Phi\colon \Cart_{(\emptyset)}\rightarrow \Cart_{\emptyset}.\]
    
    It is worth unraveling the effect of the functor $\Phi$. Note that the limit of a functor $F \colon \CI\rightarrow \Cat$ such that $\CI$ has a initial object is simply given by $F(\emptyset)$, and so $\Phi(F)\colon \CI^{\ltri}\rightarrow \Cat$ agrees with $F$ on $\CI$ and  sends the cone point to $F(\emptyset)$ again. Therefore the effect of $\Phi$ on objects is not so interesting. Nevertheless the effect of $\Phi$ on morphisms is useful. Suppose 
    \[
    \begin{tikzcd}
        X \arrow[r, "H"] \arrow[d, "p"]& Y \arrow[d,"q"] \\
        \CI	\arrow[r, "h"] &	\CJ
    \end{tikzcd}
    \] is a morphism in $\Cart_{(\emptyset)}$ which represents the natural transformation $\eta\colon F\Rightarrow G\circ h$. Then the induced morphism
    \[
    \begin{tikzcd}
        \Phi(X) \arrow[r, "\Phi(H)"] \arrow[d, "p^{\ltri}"]& \Phi(Y) \arrow[d,"q^{\ltri}"] \\
        \CI^\ltri	\arrow[r, "h^\ltri"] &	\CJ^\ltri
    \end{tikzcd}
    \]
    represents the natural transformation 
    \[\Phi(\eta)\colon \Phi(F) \rightarrow \Phi(G)\] which agrees with $\eta$ on every object except at the cone point, where it is given by the composite 
    \[F(\emptyset)\xrightarrow{\eta_\emptyset} G(h(\emptyset))\xrightarrow{G(h(\emptyset)\rightarrow \emptyset)} G(\emptyset).\]
\end{construction}

\begin{proof}[Proof of \Cref{prop:transport_to_initial_object}]
    We may consider the composite 
    \[\Cart_{(\emptyset)}\xrightarrow{\Phi} \Cart_{\emptyset}\xrightarrow{\Tr_{\emptyset}} \Fun([1],\Cat).\] On objects $\Phi \circ \Tr_{\emptyset}$ sends the cartesian fibration $p\colon \Unct{F} \rightarrow \CI$ to the functor $\Tr_{\emptyset}\colon \Unct{F^\ltri}\rightarrow F(\emptyset).$ We will from now on restrict this functor in the source to the subcategory $\Unct{F}$. This is natural in $\CI$ and so constitutes a new functor 
    \[
    \Tr_{(\emptyset)}\colon \Cart_{(\emptyset)} \rightarrow \Fun([1],\Cat).
    \] Note that on objects this functor agrees with $\Tr_{\emptyset}$. However on morphisms $\Tr_{(\emptyset)}$ sends a square
    \[
    \begin{tikzcd}
        \Unct{F} \arrow[r, "H"] \arrow[d, "p"]& \Unct{G} \arrow[d,"q"] \\
        \CI	\arrow[r, "h"] &	\CJ,
    \end{tikzcd}
    \] where now the map $h$ does not necessarily preserve the initial object, to the commutative square
    \[\begin{tikzcd}[column sep = huge]
        {\Unct{F}} &[10] {\Unct{G}} \\
        {F(\emptyset)} & {G(\emptyset).}
        \arrow["H", from=1-1, to=1-2]
        \arrow[from=1-1, to=2-1]
        \arrow[from=1-2, to=2-2]
        \arrow["{G(\emptyset \rightarrow h(\emptyset)) H(\emptyset)}", from=2-1, to=2-2]
    \end{tikzcd}\qedhere\]
\end{proof}

\subsection{Combining the previous two constructions}\label{subsec:pushforward}

We now combine the previous two sections to obtain the required result. To state it, we define $\tCat_{(\ast)}$ to be the \textit{full} subcategory of $\tCat$ spanned by those $\infty$-categories which admit a final object. Given an $\infty$-category $\CI$ with a terminal object and an object $i\in \CI$, we write $p_i\colon i\rightarrow \ast$ for the unique map from $i$ to $\ast$. Combining the constructions of this section we obtain the following functor.

\begin{theorem}\label{thm:Rel_global_sect_functor}
    There exists a functor
    \[
    \pushfor(-)\colon {\tCat_{(\ast)}}\oplaxslice{\Cat^{\otimes,\lax}}\rightarrow \Fun^{\lax}([1],\tCat)
    \] with the following properties. 
    
    On objects, $\pushfor(-)$ sends a functor $F\colon \CI\rightarrow \Cat^{\otimes,\lax}$ to the functor 
    $
    \pushfor(F)\colon \CI^{\op}\rightarrow \CAlg(F(\ast))
    $ 
    such that
    \begin{enumerate}
    \item $\pushfor(F)(i) \simeq F(p_i) \1_{F(i)}$.
    \item Given a map $f\colon X\rightarrow Y$ in $\CI$, the map $\pushfor(F)(f)\colon F(p_j) \1_{F(j)} \rightarrow F(p_i)\1_{F(i)}$ is given by the composite 
    \[
        F(p_j) \1_{F(j)} \xrightarrow{F(p_j) \lax_f} F(p_j) F(f) \1_{F(i)} \simeq F(p_i) \1_{F(i)},
    \] 
    where $\lax_f$ is the lax unit map of the lax monoidal functor $F(f)$.
    \end{enumerate}
    
    On morphisms, $\pushfor(-)$ sends the oplaxly commuting triangle 
    \[\begin{tikzcd}
    \CI && \CJ \\
	& {\makebox{$\Cat^{\otimes,\lax}$}}
	\arrow[""{name=0, anchor=center, inner sep=0}, "H"', from=1-1, to=2-2]
	\arrow[""{name=1, anchor=center, inner sep=0}, "G", from=1-3, to=2-2]
	\arrow["F", from=1-1, to=1-3]
	\arrow["\eta"{description}, shift left=1, shorten <=13pt, shorten >=13pt, Rightarrow, from=0, to=1]
\end{tikzcd}\]
to the square 
\[\begin{tikzcd}[column sep = large]
	{\CI^{\op}} & {\CAlg(H(\ast))} \\
	{\CJ^{\op}} & {\CAlg(G(\ast))}
	\arrow["{\pushfor(H)}", from=1-1, to=1-2]
	\arrow["{\pushfor(G)}"', from=2-1, to=2-2]
	\arrow["F"', from=1-1, to=2-1]
	\arrow["{G(p_{F(\ast)})\circ\eta_\ast}", from=1-2, to=2-2]
	\arrow["{\pushfor(\eta)}", shorten <=12pt, shorten >=12pt, Rightarrow, from=2-1, to=1-2]
\end{tikzcd}\]
where 
\begin{enumerate}
    \item[(3)] $\pushfor(\eta)$ is given at an object $i \in \CI$ by the composite: 
\begin{align*}
G(p_{F(i)})\1_{GF(i)} \xrightarrow{G(p_{F(i)})(\lax_{\eta_i})} G(p_{F(i)}) \eta_i \1_{H(i)} \simeq G(p_{F(\ast)}) GF(p_i)\eta_{i} \1_{H(i)} \simeq G(p_{F(\ast)})\eta_* H(p)(\1_{H(i)}),
\end{align*}
where $\lax_{\eta_i}$ is the lax unit map of $\eta_i$.
\end{enumerate}
\end{theorem}

\begin{proof}
    Consider the composite 
    \[
    {\tCat_{(\ast)}}\oplaxslice {\Cat^{\otimes,\lax}} \xrightarrow{{\CAlg(-)}\circ{-}} \tCat\oplaxslice{\Cat}  \underset{{\vspace{-0pt}\sim}}{\xrightarrow{\Unct{-}}} \Cart
    \]
    where the last equivalence is \Cref{slice-cart}. Because the $\infty$-category $\CAlg(\CD)$ always has an initial object given by the unit $\1_\CD$ and $\CI^{\op}$ has an initial object by assumption this composite lands in the intersection  $\Cart^{(\emptyset)}\cap \Cart_{(\emptyset)}$. We may therefore apply the functor 
    \[
    (s_\emptyset,\Tr_{(\emptyset)})\colon \Cart^{(\emptyset)}\cap \Cart_{(\emptyset)}\rightarrow \Fun^{\lax}([1],\Cat)\times_{\Cat} \Fun^{\lax}([1],\Cat).
    \] 
    Here the pullback is taken along evaluation at source and target respectively. The fact that $(s_\emptyset,\Tr_{(\emptyset)})$ factors through the pullback is justified by the fact that the target of $s_\emptyset$ agrees with the source of $\Tr_{(\emptyset)}$: both are the identity on $\Cart^{(\emptyset)}\cap \Cart_{(\emptyset)}$. Now we may paste these two lax natural transformations to obtain the functor $\pushfor$. More formally, these means we postcompose by the functor 
    \[
    \mathrm{paste}\colon \Fun^{\lax}([1],\Cat)\times_{\Cat} \Fun^{\lax}([1],\Cat)\rightarrow \Fun^{\lax}([1],\Cat).
    \] We define the result to be $\pushfor(-)$. Unwinding the construction one finds it has the properties required.
\end{proof}

\begin{remark}
We note the similarity between the properties of $\pushfor(-)$ and the desiderata of \Cref{def-unravel-incoh}.
\end{remark}

\section{Families of equivariant cohomology theories}\label{sec:unraveling}
After the technical work of the last section, we can cash in and prove \Cref{thm:Unraveling}.

\begin{proposition}\label{thm-E}
    Consider a functor $\CR \colon \Spcgl{\CE}\to\Cat^{\otimes,\lax}$. Then there exists a functor 
    \[
    \decatcoh{\bullet}{-}{\CR}\colon \Glo{\CE} \to \Fun^{\lax}([1], \tCat^{\otimes,\lax})
    \] such that: 
    \begin{enumerate}
        \item $\decatcoh{\bullet}{-}{\CR}$ sends $\CB G$ to a functor 
        \[
        \decatcoh{G}{-}{\CR}\colon \Spc_G^{\op}\rightarrow \CAlg(\CR_{\CB G})
        \] which agrees on objects and morphisms with \textup{(\ref{item-unravel-obj})} and \textup{(\ref{item-unravel-morph})} of \Cref{def-unravel-incoh} respectively.
        \item $\decatcoh{\bullet}{-}{\CR}$ sends a morphism $\CB \alpha \colon \CB H \to \CB G$ to the lax square  
        \begin{equation}
        \begin{tikzcd}[column sep = huge]
            {\Spc_H^{\op}} &  {\CAlg(\CR_{\CB H})}\\
            {\Spc_G^{\op}} & {\CAlg(\CR_{\CB G})}
            \arrow["{\alpha_!}"', from=1-1, to=2-1]
            \arrow["{\decatcoh{H}{-}{\,\CR}}", from=1-1, to=1-2]
            \arrow["{\decatcoh{G}{-}{\,\CR}}"', from=2-1, to=2-2]
            \arrow["{\alpha_*}", from=1-2, to=2-2]
            \arrow[shorten <=18pt, shorten >=18pt, Rightarrow, from=2-1, to=1-2, "T_{\alpha}"]
        \end{tikzcd}
    \end{equation}
        filled by a natural transformation 
        \[T_\alpha\colon \decatcoh{G}{\alpha_!(-)}{\CR} \Rightarrow \alpha_*\decatcoh{H}{-}{\CR}\] which agrees pointwise with \textup{(\ref{item-unravel-nat-tran})} of \Cref{def-unravel-incoh}.
    \end{enumerate}
\end{proposition}

\begin{remark}\label{rem:Talpha-equivalence}
    Note that $T_\alpha$ is an equivalence when $\alpha$ is injective. Indeed, borrowing the notation of \Cref{def-unravel-incoh}(\ref{item-unravel-nat-tran}), in this case $h$ is necessarily an equivalence and so its associated lax structure map $\1_{\CR(\alpha_! Z)} \rightarrow h_* \1_{\CR(Z)}$ is also.
\end{remark}

\begin{remark}\label{rem-cocartesian2}
    Let $\mathbf{Cat}^{\amalg,\lax}$ be the $(\infty,2)$-category of cocartesian monoidal $\infty$-categories and lax monoidal functors. We note that $\mathbf{Cat}^{\amalg,\lax}$ is equivalent to the full subcategory $\mathbf{Cat}^{\amalg}\subseteq \mathbf{Cat}$ spanned by those $\infty$-categories which admit finite coproducts, see \cite{HA}*{Proposition 2.4.3.8}.
\end{remark}

\begin{proof}
    We first pass $\CR$ through the composite 
    \[\Fun(\Spcgl{\CE},\Cat^{\otimes,\lax})\xrightarrow{ \ref{prop:glquotient-oplaxlim}} \oplaxlimdag \Fun(\Spc_G,\Cat^{\otimes,\lax}) \xrightarrow{\mathrm{forget}} \Fun(\Glo{\CE},\tCat\oplaxslice \Cat^{\otimes,\lax}),\] 
    where the second functor is justified by the identification of \Cref{const:functors_into_lax_slice}. Next we observe that resulting object factors through the subcategory $\tCat_{(\ast)}\oplaxslice \Cat^{\otimes,\lax}$. Therefore we may postcompose by the functor 
    \[\pushfor\colon \tCat_{(\ast)}\oplaxslice \Cat^{\otimes,\lax}\rightarrow \Fun^{\lax}([1],\tCat)\]
    from \Cref{thm:Rel_global_sect_functor} to obtain
    \[\decatcoh{\bullet}{-}{\CR}\colon \Glo{\CE} \to \Fun^{\lax}([1], \tCat).\]
    
    Finally we observe that this functor factors through the subcategory $\Fun^{\lax}([1], \tCat^{\amalg})$, and so by \Cref{rem-cocartesian2} we can equivalently view $\decatcoh{\bullet}{-}{\CR}$ as a functor 
    \[\decatcoh{\bullet}{-}{\CR} \colon \Glo{\CE} \to \Fun^{\lax}([1], \tCat^{\otimes,\lax}).\]
    This functor has the properties advertised by \Cref{thm:Rel_global_sect_functor}.
\end{proof}

If we now assume that the diagram $\CR\colon \Spcgl{\CE}\to \Cat^{\otimes,\lax}$ is obtained from a naive global 2-ring by passing to right adjoints, then one can pass the functor $\decatcoh{\bullet}{-}{\CR}$ constructed above through the mate equivalence, a fact which we record here.

\begin{corollary}\label{mainfunctor}
    Let $\CR \colon \Spcgl{\CE}^{\op} \to \PrL$ be a functor. Then there is a functor
    \[
    \decatcoh{\bullet}{-}{\CR}\colon \Glo{\CE}^{\op}\rightarrow \Fun^{\mathrm{oplax}}([1],\tCat^{\otimes,\lax})
    \]
    which agrees with the functor from \Cref{thm-E} on objects, but sends a map $\CB \alpha\colon \CB H\rightarrow \CB G$ in $\Glo{\CE}$ to the oplax square
    \begin{equation}\label{diag-cEll}
        \begin{tikzcd}[column sep = huge]
            {\Spc_G^{\op}} &  {\CAlg(\CR_{\CB G})}\\
            {\Spc_H^{\op}} & {\CAlg(\CR_{\CB H})}
            \arrow["{\alpha^*}"', from=1-1, to=2-1]
            \arrow["{\decatcoh{G}{-}{\,\CR}}", from=1-1, to=1-2]
            \arrow["{\decatcoh{H}{-}{\,\CR}}"', from=2-1, to=2-2]
            \arrow["{\alpha^*}", from=1-2, to=2-2]
            \arrow[shorten <=18pt, shorten >=18pt, Rightarrow, from=2-1, to=1-2, "Q_{\alpha}"]
        \end{tikzcd}
    \end{equation}
    which is filled by the Beck--Chevalley transformation of $T_\alpha$.
\end{corollary}

\begin{proof}
    Consider the functor $\CR \colon \Spcgl{\CE} \to \Cat^{\otimes,\lax}$ given by passing to adjoints. Applying the previous proposition to this we obtain a functor 
    \[
    \decatcoh{\bullet}{-}{\CR}\colon \Glo{\CE}\rightarrow \Fun([1],\tCat)
    \]
    Observe that because $\alpha^*$ is strong monoidal, the adjunction $\alpha^*\dashv \alpha_*$ induces an adjunction on commutative algebra objects. Therefore $\decatcoh{\bullet}{-}{\CR}$ factors through the subcategory $\Fun^{\rR,\lax}([1],\Cat)$ from \Cref{def:lax_fun_adj_morp}. We may therefore pass $\decatcoh{\bullet}{-}{\CR}$ through the second mate equivalence (\ref{second-mate-equivalence}) to obtain the functor
    \[
    \decatcoh{\bullet}{-}{\CR}\colon \Glo{\CE}^{\op}\rightarrow \Fun^{\mathrm{oplax}}([1],\tCat).
    \]
    Once again we observe that the diagram factors through $\tCat^\amalg$, and so lifts uniquely to a functor into $\Fun^{\oplax}([1],\tCat^{\otimes,\lax})$. 
\end{proof}

The unfurling construction is functorial in the following sense.

\begin{proposition}\label{prop:funct_unfurl}
The previous construction constitutes a functor 
\[
\mathbb{H}\colon \NaiveTwoGlRing{\CE} \rightarrow \Fun(\Glo{\CE}^{\op},\Fun^{\oplax}([1],\tCat^{\otimes,\lax}))
\]
\end{proposition}

\begin{proof}
The construction is given by the following composite of functors
\begin{align*}
\NaiveTwoGlRing{\CE} \subset& \Fun(\Spcgl{\CE}^{\op},\PrL)\xrightarrow{\mathrm{adj}} \Fun(\Spcgl{\CE}^{\op}, (\PrR)^{\op}) \simeq \Fun(\Spcgl{\CE}, \PrR)^{\op} \subset 
 \dots \\ 
 \dots \subset& \Fun(\Spcgl{\CE},\Cat^{\otimes,\lax})^{\op} \xrightarrow{\ref{prop:glquotient-oplaxlim}} \bigl(\oplaxlimdag \Fun(\Spc_G,\Cat^{\otimes,\lax}) \bigr)^{\op} \rightarrow \dots \\
 \dots \xrightarrow{\mathrm{forget}}& \Fun(\Glo{\CE},\tCat_{(\ast)}\oplaxslice \Cat^{\otimes,\lax})^{\op} \xrightarrow{\pushfor_*} \Fun(\Glo{\CE},\Fun^{\rR,\lax}([1],\tCat^{\amalg}))^{\op} \rightarrow \dots \\
 \dots \xrightarrow{\ref{second-mate-equivalence}}& \Fun(\Glo{\CE},\Fun^{\rL,\oplax}([1],\tCat^{\amalg})^{\op})^{\op} \simeq \Fun(\Glo{\CE}^{\op},\Fun^{\rL,\oplax}([1],\tCat^{\amalg})) \subset \dots \\
 \dots \subset& \Fun(\Glo{\CE}^{\op},\Fun^{\oplax}([1],\tCat^{\otimes,\lax})),
\end{align*}
and so functorial.
\end{proof}

\begin{remark}\label{rmk:unwind_funct_unraveling}
Unwinding the effect of the construction above, we find that it sends a map $F\colon \CR\rightarrow \CR'$ of naive global rings to a $\Glo{\CE}^{\op}$-indexed collection of oplaxly commuting squares
\[\begin{tikzcd}[column sep = large]
	{\Spc_G^{\op}} & {\CR_{\CB G}} \\
	{\Spc_G^{\op}} & {\CR'_{\CB G}.}
	\arrow["{\decatcoh{G}{-}{\CR}}", from=1-1, to=1-2]
	\arrow[Rightarrow, no head, from=1-1, to=2-1]
	\arrow["{u_{F, G}}"', shorten <=7pt, shorten >=7pt, Rightarrow, from=1-2, to=2-1]
	\arrow["F_{\CB G}", from=1-2, to=2-2]
	\arrow["{\decatcoh{G}{-}{\CR'}}"', from=2-1, to=2-2]
\end{tikzcd}\]
The transformation $u_{F,G}$ is given at a $G$-space $X$ by the 
map 
\[
F\decatcoh{G}{X}{\CR} \coloneqq F_{\CB G} p_*\1_{\CR(X\sslash G)} \xrightarrow{BC} p_* F_{X\sslash G} \1_{\CR(X\sslash G)} \simeq p_* \1_{\CR'(X\sslash G)} \simeq \eqqcolon \decatcoh{G}{X}{\CR}, 
\]
where $BC$ denotes the Beck--Chevalley transformation of the naturality equivalence $ F_{\CB G} p^* \simeq p^* F_{X\sslash G}$. 
\end{remark}

All in all, we have now proven \Cref{thm:Unraveling}.

\begin{proof}[Proof of \Cref{thm:Unraveling}]
The natural transformation constructed in \Cref{mainfunctor} has all of the required properties by \Cref{thm-E}.
\end{proof}

\section{Global sections and equivariant cohomology theories}\label{sec:equivariantcoh}

In this section we show that the family of equivariant cohomology theories constructed before recovers the naive global ring $\decatcohsp{\gl}{-}{\CR}$. More precisely we exhibit an equivalence
\begin{equation}\label{two_unravelings}
    \decatcohsp{\gl}{-\sslash \bullet}{\CR} \simeq \Gamma\decatcoh{\bullet}{-}{\CR}
\end{equation}
as objects of $\laxlimdag \Fun(\Spc_\bullet^{\op},\CAlg)$. Of course, we should begin by explaining the notation used and how precisely we interpret both sides of this expression as objects of the partially lax limit. We begin with the left hand side, which we have previously seen.

\begin{construction}\label{unraveling1}
    We may apply the equivalence of \Cref{prop:family-of-eq-vs-global}(2)
    \[
    \Phi \colon \Fun^\rR(\Spcgl{\CE}^{\op},\CAlg)\rightarrow \laxlimdag \Fun^\rR(\Spc_\bullet^{\op}, \CAlg)
    \]
    to the functor $\decatcohsp{\gl}{-}{\CR}$ to obtain the object $\decatcohsp{\gl}{-\sslash \bullet\,}{\CR}$ of $\laxlimdag \Fun^{\rR}(\Spc_\bullet^{\op}, \CAlg)$.
\end{construction}

Before we can explain how to view $\Gamma\decatcoh{\bullet}{-}{\CR}$ as an object of the partially lax limit we must reinterpret the global section functor $\Gamma$.

\begin{construction}\label{lem-functor-glsec}
Recall that $\Sp$ is initial in $\PrL$, and so there exists a unique natural transformation $i\colon \Delta(\Sp)\Rightarrow \id_{\PrL}$ from the constant functor at $\Sp$ to the identity on $\PrL$. We write $\Gamma\colon \CC\rightarrow \Sp$ for the right adjoint of the component of $i$ at $\CC$, and call this \emph{the global sections functor} of $\CC$. The global section functors form a natural transformation $\Gamma\colon \PrL^{\op}\rightarrow \Fun([1],\PrR)$ by passing to right adjoints. By uniqueness of adjoints one finds that $\Gamma$ agrees with the natural transformation of functors $\map(\1,-)\colon \CC\rightarrow \Sp$. Postcomposing with the functor $\CAlg(-)$, we obtain another natural transformation $\CAlg(\Gamma)\colon \PrL^{\op}\rightarrow \Fun([1],\Cat)$ which sends $\CC$ to $(\Gamma \colon \CAlg(\CC) \to \CAlg)$. 
\end{construction}
    
\begin{remark}
Pasting the natural transformation $\CAlg(\Gamma)$ with the oplax cocone $\1_\bullet$ of units constructed in \Cref{ex:unit-sections} we obtain an $(\PrL)^{\op}$-shaped oplax cocone $\Gamma\1\colon \ast \Rightarrow \Delta(\CAlg)$. By \Cref{ex:functors_are_lax_cones}, this is equivalent to a functor in $\Fun(\PrL,\CAlg)$. As such it agrees with $\mathrm{End}_\bullet(\1)\colon \PrL\rightarrow \CAlg$ of \Cref{ex:decat}. This follows from the identification in the previous construction.
\end{remark}

\begin{construction}\label{secondway}
Let $\CR$ be a naive global 2-ring and consider the functor $\CR \colon \Spcgl{\CE} \to \PrR$ given by passing to right adjoints. We may paste the lax transformation $\decatcoh{\bullet}{-}{\CR}$ from \Cref{thm-E} with the natural transformation $\CAlg(\Gamma)$ to obtain a new lax natural transformation
    \[
    \decatcohsp{\bullet}{-}{\CR} \colon {\Glo{\CE}}\rightarrow \Fun^{\mathrm{lax}}([1],\tCat)
    \]
which sends $\CB G\in \Glo{\CE}$ to the composite
    \begin{equation}\label{eq-G-coh}
        \decatcohsp{G}{-}{\CR} \colon \Spc_G^{\op}\xrightarrow{	\decatcoh{G}{-}{\,\CR}}\CAlg(\CR_{\CB G})\xrightarrow{\Gamma} \CAlg.
    \end{equation}
     Because the target of this lax natural transformation is constant at $\CAlg$, $\decatcohsp{\bullet}{-}{\CR}$ defines a functor ${\Glo{\CE}}\rightarrow {\tCat\laxslice {\CAlg}}$ into the lax slice category. As a consequence of the dual of \Cref{const:functors_into_lax_slice}, we may equivalently view this as an object of $\laxlim \Fun(\Spc_\bullet^{\op},\CAlg)$ as before. Furthermore by \Cref{rem:Talpha-equivalence} and by definition of naive global $2$-ring, we find it is contained in the full subcategory $\laxlimdag \Fun^{\rR}(\Spc_\bullet^{\op},\CAlg)$.
\end{construction}

Having constructed our two objects of $\laxlimdag \Fun(\Spc_\bullet^{\op},\CAlg)$, we may now compare them.

\begin{proposition}\label{obser:two_unravelings_agree}
    The objects $\decatcohsp{\gl}{-\sslash \bullet}{\CR}$ and $\decatcohsp{\bullet}{-}{\CR}$ of $\laxlimdag \Fun(\Spc_\bullet^{\op},\CAlg)$ are equivalent.
\end{proposition}

\begin{proof}
This is essentially true by definition. However since the definitions are slightly involved, this deserves elaboration. On the one hand, the object $\decatcohsp{\bullet}{-}{\CR}$ is given by the partially lax diagram of composite oplax natural transformations  
\[
\ast \xRightarrow{\1_{- \sslash G}} \CAlg(\CR_{- \sslash G}) \xRightarrow{\Tr_\emptyset} \CAlg(\CR_{\CB G}) \xRightarrow{\CAlg(\Gamma)} \CAlg.
\]
Note that the second two oplax natural transformations compose to $\CAlg(\CR_{-\sslash G})\xRightarrow{\CAlg(\Gamma)}\CAlg$ by uniqueness of the transformation $\CAlg(\Gamma)$. On the other hand $\decatcohsp{\gl}{-\sslash \bullet}{\CR}$ is given by applying the equivalence of \Cref{prop:family-of-eq-vs-global}(2) to the composite oplax natural transformation
\[
\ast\xRightarrow{\1_-} \CAlg(\CR) \xRightarrow{\CAlg(\Gamma)} \CAlg.
\]
Therefore it is also equivalent to the partially lax diagram of oplax natural transformations 
\[
\ast\xRightarrow{\1_{-\sslash G}} \CAlg(\CR_{-\sslash G})  \xRightarrow{\CAlg(\Gamma)} \CAlg.
\]
We conclude that the two objects agree.
\end{proof}
\part{Genuine global 2-rings}\label{part:genuine}
We now introduce a refinement of the notion of a naive global 2-ring, which we call \emph{(genuine) global 2-rings}. We will show that for such global 2-rings, the family of equivariant cohomology theories constructed before are strongly related and computable. We then explain how to use this to construct a canonical genuine refinement $\Gamma(\CR)_{\gl}\in \CAlg(\Sp_{\gl})$ for the associated naive global ring $\decatcohsp{\gl}{-}{\CR}$, thus accomplishing the main goal of this article. Finally given a global $2$-ring $\CR$ and a group $G$, we construct a genuine global section functor $\bGamma_G\colon \CR_G\rightarrow \Sp_G$. These are best behaved in the case of a \emph{rigid} global 2-ring, as we explain in the final section of this part.

\section{The Ginzburg-Kapranov-Vasserot axioms}\label{sec:cat_to_global}

So far we have associated to any naive global 2-ring $\CR\colon \Spcgl{\CE}^{\op}\rightarrow \PrL$ a family \[\decatcoh{G}{-}{\CR}\colon \Spc_G^{\op}\rightarrow \CAlg(\CR_{\CB G})\] of multiplicative $G$-equivariant $\CR_{\CB G}$-valued cohomology theories. These cohomology theories are related by two \emph{change-of-group transformations} 
\[
Q_\alpha\colon \decatcoh{H}{\alpha^*(-)}{\CR} \Rightarrow \alpha^*\decatcoh{G}{-}{\CR} \quad \text{and} \quad T_\alpha\colon \decatcoh{G}{\alpha_!(-)}{\CR} \Rightarrow \alpha_*\decatcoh{H}{-}{\CR}
\]
which are associated to any group homomorphism $\alpha\colon H\rightarrow G$. Moreover this data is coherently functorial in the group homomorphism, as expressed by the existence of a functor
\[
\decatcoh{\bullet}{-}{\CR}\colon \Glo{\CE}^{\op}\rightarrow \Fun^{\mathrm{oplax}}([1],\tCat^{\otimes,\lax}).
\]

Such data is reminiscent of the axiomatics for equivariant elliptic cohomology, as expressed by \cite{GKV95}. In particular our choice of the name $T_\alpha$ for the first transformation above is inspired by the notation there. In this section we introduce a condition on $\CR$ such that this data satisfies the change-of-group axioms of \emph{op.~cit}.

\begin{definition}\label{def:oriented_glo_cat}
Let $\CE$ be a multiplicative global family of compact Lie groups and let $\CT\subset\Glo{\CE}$ be a subcategory of enough injective objects in the sense of \Cref{def:enough-injectives}. We say a naive global 2-ring $\CR $ is \emph{$\CT$-pregenuine} if the following conditions hold:
    \begin{enumerate}
        \item\label{item:gen-adjoint} Given a pullback square in $\Spcgl{\CE}$ 
        \[
        \begin{tikzcd}
            {\sX} & \CB K\ \\
            {\CB H} & {\CB G}
            \arrow["g'", from=2-1, to=2-2]
            \arrow["f'", hook, from=1-2, to=2-2]
            \arrow["f"', hook, from=1-1, to=2-1]
            \arrow["\lrcorner"{anchor=center, pos=0.125}, draw=none, from=1-1, to=2-2]
            \arrow["g", from=1-1, to=1-2]
        \end{tikzcd}
        \]
        where $f'$ is faithful and 
        $\CB G\in\CT$,  
        the Beck--Chevalley transformation filling the square 
        \[
        \begin{tikzcd}
        {\CR_{\CB K}} & {\CR_{\sX}} \\    
        {\CR_{\CB G}} & {\CR_{\CB H}}
        \arrow["{f'_*}"', from=1-1, to=2-1]
        \arrow["{g^*}", from=2-1, to=2-2]
        \arrow["{(g')^*}", from=1-1, to=1-2]
        \arrow["{f_*}", from=1-2, to=2-2]
        \arrow[shorten <=8pt, shorten >=8pt, Rightarrow, from=2-1, to=1-2]
        \end{tikzcd}
        \]is an equivalence;\\
        \item\label{item:gen-proj} Given a faithful morphism $\CB \alpha\colon\CB H\to\CB G$ in $\Orb{\CE}$, the canonical map 
        \[
        \CF \otimes \alpha_* \CG \rightarrow \alpha_* (\alpha^* \CF \otimes \CG)
        \] adjoint to the composite 
        \[
        \alpha^*(\CF \otimes \alpha_* \CG) \simeq \alpha^*\CE \otimes \alpha^*\alpha_* \CG \xrightarrow{\alpha^* \CF \otimes \epsilon} \alpha^* \CF \otimes \CG 
        \] is an equivalence for all $\CF\in \CR_{\CB G}$ and $\CG\in \CR_{\CB H}$. In other words, we require that the adjunction $(\alpha^*,\alpha_*)$ satisfies the right projection formula.
    \end{enumerate}
\end{definition}

\begin{remark}\label{rem:para-perspective}
If $\CT = \CE$ is the whole family, then the previous definition admits a recasting in the language of global $\infty$-categories developed in \cite{CLL_Global} and \cite{CLL_Partial}. Namely a $\CE$-pregenuine global 2-ring $\CR$ is equivalently a fiberwise presentable, fiberwise stable, equivariantly complete $\CE$-global $\infty$-category equipped with a symmetric monoidal structure which commutes with fiberwise colimits and finite equivariant limits in each variable. We will not use this perspective in this article, except in \Cref{lem-cohomotopy}.
\end{remark}

Given these assumptions, we prove the following theorem.

\begin{theorem}\label{thm:GKV-axioms}
Let $\CR$ be a $\CT$-pregenuine global 2-ring. Then
\begin{enumerate}[itemsep = 5pt]
    \item \emph{Induction:} Let $\alpha\colon G\rightarrow G/N$ be a surjective group homomorphism with kernel $N$ and let $X$ be a $G$-space such that the action of $N$ on $X$ is free. Write $p\colon X\rightarrow X/N$ for the canonical projection map, which is $G$ equivariant. Then the composite 
    \[\decatcoh{G/N}{\alpha_!(X/N)}{\CR}\xrightarrow{T_\alpha} \alpha_*\decatcoh{G}{X/N}{\CR} \xrightarrow{\alpha_*\decatcoh{G}{p}{\,\CR}} \alpha_*\decatcoh{G}{X}{\CR}\] is an equivalence;
    \item \emph{Base change:} Let $\CB \alpha \colon \CB H \to \CB G$ be a map in $\Glo{\CE}$. If $\CB G\in \CT$, then the natural transformation
    \[
    Q_\alpha\colon \alpha^* \decatcoh{G}{X}{\CR}\to \decatcoh{H}{\alpha^*X}{\CR}
    \] of \Cref{diag-cEll} is an equivalence for all compact $G$-spaces $X$;
    \item \emph{K\"unneth:} Let $G$ and $H$ be two groups in $\CT$, $X$ a compact $G$-space and $Y$ a compact $H$-space. Then there is an equivalence 
    \[
    \pi_G^*\decatcoh{G}{X}{\CR}\otimes \pi_H^*\decatcoh{H}{Y}{\CR} \simeq \decatcoh{G\times H}{X \times Y}{\CR},
    \]
    where $\pi_H$ and $\pi_G$ denote the two projections $G\times H\rightarrow H,G$. Moreover, the functor $\decatcoh{G}{-}{\CR}\colon (\Spc_G^\omega)^{\op} \to \CR_{\CB G}$ is strong monoidal.
\end{enumerate}
\end{theorem}

We will consider each point in turn, we begin with the induction axiom. This axiom in fact holds for the unraveling of an arbitrary naive global 2-ring.

\begin{proposition}\label{prop:induction_axiom}
    Consider a naive global 2-ring $\CR\colon \Spcgl{\CE}^{\op}\rightarrow \PrL$. Let $\alpha\colon G\rightarrow G/N$ be a surjective group homomorphism with kernel $N$ and let $X$ be a $G$-space such that the action of $N$ on $X$ is free. Write $p\colon X\rightarrow X/N$ for the canonical projection map, which is $G$-equivariant. Then the composite 
    \[\decatcoh{G/N}{\alpha_!(X/N)}{\CR}\xrightarrow{T_\alpha} \alpha_*\decatcoh{G}{X/N}{\CR} \xrightarrow{\alpha_*\decatcoh{G}{p}{\,\CR}} \alpha_*\decatcoh{G}{X}{\CR}\] is an equivalence.
\end{proposition}	

\begin{proof}
    We note that the $G/N$-space $X/N$ is equivalent to $\alpha_!(X/N)$ and so the $G$-space $\alpha^*(X/N)$ is equivalent to $\alpha^*\alpha_!X$. Furthermore the map $p\colon X\rightarrow X/N$ identifies with the unit of the adjunction $\alpha_!\dashv \alpha^*$. Therefore, applying \Cref{rem:restriction_functoriality}, we obtain the following diagram
\[\begin{tikzcd}
	X\sslash G & {\alpha^*(X/N)}\sslash G & {X/N}\sslash G/N \\
	& {\CB G} & {\CB(G/N)}
	\arrow["p", from=1-1, to=1-2]
	\arrow["q", curve={height=-24pt}, two heads, from=1-1, to=1-3]
	\arrow["f"', from=1-1, to=2-2]
	\arrow["h", two heads, from=1-2, to=1-3]
	\arrow[hook, from=1-2, to=2-2]
	\arrow["\lrcorner"{anchor=center, pos=0.125}, draw=none, from=1-2, to=2-3]
	\arrow["g", hook, from=1-3, to=2-3]
	\arrow["\alpha", two heads, from=2-2, to=2-3]
\end{tikzcd}\]
    of global spaces. Recall that the maps $T_\alpha$ and $\alpha_*\decatcoh{G}{p}{\CR}$ are both given by the image of the lax unit comparisons of $h$ and $p\sslash G$ in $\CR_{\CB (G/N)}$, see \Cref{def-unravel-incoh}(3) and (2). Since lax unit maps compose we find that the composite in the statement is equivalent to the map
    \[(f/N)_*\1_{\CR_{X/N\sslash G/N}} \xrightarrow{(f/N)_*(\lax_q)}  (f/N)_* q_* \1_{\CR_{X\sslash G}} \simeq \alpha_*f_* \1_{\CR_{X\sslash G}}.\]
    We have to show that this map is an equivalence when $X$ is a $G$-space on which $N$ acts freely. Note that the subcategory of $G$-spaces such that the action of $N$ on $X$ is free is generated under colimits by the $G$-orbits $G/H$ such that $H\cap N = e$.  Because the condition of the proposition is closed under limits it suffices to consider such a $G$-orbit.
    In this case $X\sslash G \simeq \CB H$ and $f\simeq \CB\iota$ for an inclusion $\iota \colon H \to G$. Note that the composite $\CB \alpha \circ \CB\iota\colon \CB H\rightarrow \CB (G/N)$ is clearly faithful since $H \cap N=e$, and so $q$ is an equivalence. From this the conclusion follows immediately.
\end{proof}

Next we consider the base-change axiom.

\begin{lemma}\label{lem:Q_equivalence}
Let $\CR \colon \Spcgl{\CE}^{\op} \to \PrL$ be a $\CT$-pregenuine global 2-ring, and let $\alpha \colon  H \to G$ be a group homomorphism such that $\CB G\in \CT$. The natural transformation $Q_\alpha\colon \alpha^* \decatcoh{G}{X}{\CR}\Rightarrow \decatcoh{H}{\alpha^*X}{\CR}$ from \Cref{diag-cEll} is an equivalence on all compact $G$-spaces. 
\end{lemma}

\begin{proof}
    Note that because $\decatcoh{G}{-}{\CR}$ and $\decatcoh{H}{-}{\CR}$ preserve all limits by definition of a naive global $2$-ring, the condition that $Q_\alpha$ is an equivalence is closed under finite limits and retracts of $G$-spaces. Therefore it suffices to check that $Q_\alpha$ is an equivalence at every orbit. Therefore we fix a $G$-orbit $X \simeq G/K$. As discussed in \Cref{rem:restriction_functoriality}, $\alpha^* X$ fits into the following pullback:
    \[\begin{tikzcd}
        {\alpha^*X\sslash H} & {\alpha_!\alpha^* X\sslash G} & {X\sslash G} \\
        {\CB H} && {\CB G,}
        \arrow["\CB \alpha", from=2-1, to=2-3]
        \arrow["f", hook, from=1-3, to=2-3]
        \arrow["{\alpha^*f}", hook, from=1-1, to=2-1]
        \arrow["\lrcorner"{anchor=center, pos=0.125}, draw=none, from=1-1, to=2-3]
        \arrow["\epsilon", hook, from=1-2, to=1-3]
        \arrow["\gamma", two heads, from=1-1, to=1-2]
    \end{tikzcd}\]
    where we have furthermore fixed a factorization of the map $h\colon\alpha^* X\sslash H \rightarrow X\sslash G$ into a \quotientmap{} followed by a faithful morphism. The natural transformation $Q_\alpha$ is by definition given by the composite
   \[\alpha^*\decatcoh{G}{X}{\CR}\xrightarrow{\alpha^*\decatcoh{G}{\epsilon}{\CR}} \alpha^*\decatcoh{G}{\alpha_!\alpha^* X}{\CR} \xrightarrow{T_\alpha} \alpha^*\alpha_*\decatcoh{H}{\alpha^*X}{\CR}\xrightarrow{\epsilon'} \decatcoh{H}{\alpha^*X}{\CR},\] 
   where $\epsilon'$ is the counit of the adjunction $\alpha^*\dashv \alpha_*$. Recall from \Cref{def-unravel-incoh} that the map $\decatcoh{G}{\epsilon}{\CR}$ is given by applying $f_*$ to the lax unit map of $\epsilon$ while $T_\alpha$ is given by applying $f_*\epsilon_*$ to the lax unit map of $\gamma$. Therefore the composite is given by applying $f$ to the lax unit map of the composite $h = \epsilon \circ \gamma$. Recall that the lax unit map of $h$ is induced by $h_*$ being a right adjoint of $h^*$ and therefore is equivalent to the composite
    \[\1_{\CR_{X\sslash G}} \xrightarrow{\eta} h_*h^* \1_{\CR_{X\sslash G}} \simeq h_* \1_{\CR_{\alpha^* X\sslash H}}.\]
    We find that $Q_\alpha$ is given by going the right-most way around the following diagram:
    \[\begin{tikzcd}
        {\alpha^*f_*\1_{\CR_{X\sslash G}}} \\
        {\alpha^*f_*h_* h^*\1_{\CR_{X\sslash G}}} & {\alpha^*f_*h_* \1_{\CR_{\alpha^* X\sslash H}}} \\
        {\alpha^*\alpha_*(\alpha^* f)_* h^*\1_{\CR_{X\sslash G}}} & {\alpha^*\alpha_*(\alpha^* f)_* \1_{\CR_{\alpha^* X\sslash H}}} \\
        {(\alpha^* f)_* h^*\1_{\CR_{X\sslash G}}} & {(\alpha^* f)_* \1_{\CR_{\alpha^* X\sslash H}}.}
        \arrow["\eta"', from=1-1, to=2-1]
        \arrow["\sim", from=2-1, to=2-2]
        \arrow["\sim", from=3-1, to=3-2]
        \arrow["\sim", from=2-2, to=3-2]
        \arrow["{\epsilon'}"', from=3-1, to=4-1]
        \arrow["\sim"', from=2-1, to=3-1]
        \arrow["{\epsilon'}", from=3-2, to=4-2]
        \arrow["\sim"', from=4-1, to=4-2]
    \end{tikzcd}\]
    Note that both squares commute by naturality. In particular $Q_\alpha$ is an equivalence if and only if the left composite in the diagram above is an equivalence. However this is the Beck--Chevalley transformation associated to the square
    \[\begin{tikzcd}
        {\CR_{\alpha^* X\sslash H}} & {\CR_{X\sslash G}} \\
        {\CR_{\CB H}} & {\CR_{\CB G}.}
        \arrow["{(\alpha^* f)_*}"', from=1-1, to=2-1]
        \arrow["{\alpha_*}", from=2-1, to=2-2]
        \arrow["{f_*}", from=1-2, to=2-2]
        \arrow["{h_*}", from=1-1, to=1-2]
    \end{tikzcd}\] 
    This is an equivalence by (\ref{item:gen-adjoint}) of \Cref{def:oriented_glo_cat}, finishing the proof.
\end{proof}

We now turn to the K\"unneth axiom. We find it simpler to first prove that $\decatcoh{G}{-}{\CR}$ is strong monoidal on compact $G$-spaces.

\begin{lemma}\label{gamma_G-strong-mon}
    Let $\CR \colon \Spcgl{\CE}^{\op} \to \PrL$ be a $\CT$-pregenuine global 2-ring. Then for all $G\in \CT$, the functor $\decatcoh{G}{-}{\CR}\colon (\Spc_G^{\omega})^{\op}\rightarrow \CR_{\CB G}$ is strong monoidal. 
\end{lemma}

\begin{proof}
    Once again the result of the lemma is closed under finite limits and retracts in each variable, and so it suffices to prove that the lax symmetric monoidal structure map 
    \[\decatcoh{G}{X}{\CR} \otimes \decatcoh{G}{Y}{\CR}\rightarrow  \decatcoh{G}{X\times Y}{\CR}\] is an equivalence when $X$ and $Y$ are both orbits, that is $X=G/K$ and $Y=G/H$. Note as well that under the equivalence of \Cref{Rezk}, the cartesian product of $G$-spaces corresponds to pullback of global spaces over $\CB G$. Therefore we consider the pullback square
    \[\begin{tikzcd}
        \sX & {\CB H} \\
        {\CB K} & {\CB G}
        \arrow["f", hook,from=2-1, to=2-2]
        \arrow["g", hook, from=1-2, to=2-2]
        \arrow["{\pi_1}"', from=1-1, to=2-1]
        \arrow["{\pi_2}", from=1-1, to=1-2]
        \arrow["\lrcorner"{anchor=center, pos=0.125}, draw=none, from=1-1, to=2-2]
    \end{tikzcd}\] 
    of global spaces and compute
    \begin{align*}
        \decatcoh{G}{G/K \times G/H}{\CR} = (f\pi_1)_* \1_{\CR_\sX} &\simeq f_*(\pi_1)_* (\pi_1^* \1_{\CR_{\CB K}}\otimes \pi_2^* \1_{\CR_{\CB H}}) \\
        &\simeq f_* (\1_{\CR_{\CB K}} \otimes (\pi_1)_*\pi_2^* \1_{\CR_{\CB H}}) \\
        &\simeq f_* (\1_{\CR_{\CB K}} \otimes f^* g_* \1_{\CR_{\CB H}}) \\
        &\simeq  f_*\1_{\CR_{\CB K}} \otimes g_* \1_{\CR_{\CB H}} \\
        &\simeq \decatcoh{G}{G/K}{\CR} \otimes \decatcoh{G}{G/H}{\CR}.
    \end{align*}
    The second and fourth equivalence are an application of the right projection formula for $f$ and $\pi_1$, and so guaranteed by (\ref{item:gen-proj}) of \Cref{def:oriented_glo_cat}. The third equivalence is base-change for the pullback square defining $X$, and so follows from (\ref{item:gen-adjoint}) of \Cref{def:oriented_glo_cat}. The remaining equivalences are clear. A diagram chase shows that the equivalence constructed agrees with the lax monoidal structure map of $\decatcoh{G}{-}{\CR}$.
\end{proof}

We now show that this internal K\"unneth formula implies the external analog.

\begin{corollary}
Consider two groups $H,G\in \CT$ and $X$ and $Y$ compact $H$- and $G$-spaces respectively. Let $\CR$ be a $\CT$-pregenuine global 2-ring. Then there is an equivalence 
\[
\pi_H^*\decatcoh{H}{X}{\CR}\otimes \pi_G^*\decatcoh{G}{Y}{\CR} \simeq \decatcoh{H\times G}{X \times Y}{\CR},
\] where $\pi_H$ and $\pi_G$ denote the two projections $G\times H \rightarrow H,G$.
\end{corollary}

\begin{proof}
The equivalence is given by the composite
\[
\pi_H^*\decatcoh{H}{X}{\CR}\otimes \pi_G^*\decatcoh{G}{Y}{\CR} \xrightarrow{Q_{\pi_H}\otimes Q_{\pi_G}} \decatcoh{H\times G}{\pi_H^* X}{\CR} \otimes \decatcoh{H\times G}{\pi_G^* Y}{\CR} \simeq \decatcoh{H\times G}{X\times Y}{\CR},
\]
where we have applied both \Cref{lem:Q_equivalence} and \Cref{gamma_G-strong-mon}.
\end{proof}

In total we have proven \Cref{thm:GKV-axioms}.

\begin{proof}[Proof of \Cref{thm:GKV-axioms}]
This follows immediately by combining all of the results of this section.
\end{proof}

For later use we observe the following consequence of the theorem for the functor $\decatcoh{\bullet}{-}{\CR}$. 

\begin{corollary}\label{thm:Gamma_strong_strict}
    Suppose $\CR \colon \Spcgl{\CE}^{\op}\rightarrow \PrL$ is a $\CT$-pregenuine global 2-ring. Then the functor
    \[\decatcoh{\bullet}{-}{\CR}\colon \CT^{\op} \rightarrow \Fun^{\oplax}([1],\Cat^{\otimes,\lax}), \quad \CB G \mapsto [\decatcoh{G}{-}{\CR}\colon (\Spc_G^{\omega})^{\op}\rightarrow \CR_{\CB G}],\]obtained from \Cref{mainfunctor} factors through the subcategory
    \[
    \Fun([1],\Cat_\lex^{\otimes})\subset \Fun^{\oplax}([1],\Cat^{\otimes,\lax})
    \]
    of limit preserving strong symmetric monoidal functors and strict natural transformations.
\end{corollary}

\begin{proof}
Point (2) of \Cref{thm:GKV-axioms} shows that $Q_\alpha$ is always an equivalence, while point (3) shows that each functor is strong monoidal. Finally each functor is limit preserving since $\CR$ is in particular a naive global $2$-ring.
\end{proof}

Let us once again consider the functoriality of the unraveling construction on $\CT$-pregenuine rings.

\begin{definition}
We say that a morphism $F \colon \CR \to\CR'$ of $\CT$-pregenuine global $2$-rings is $\CT$-\emph{pregenuine} if for every faithful morphism $\iota \colon \CB H \to \CB G$ such that $\CB G \in \CT$, the Beck-Chevalley transformation  $F_{\CB G} \circ \iota_*(X) \to\iota_* \circ F_{\CB H}(X)$, associated to the natural equivalence $F_{\CB H}\iota^*\simeq \iota^* F_{\CB G}$, is an equivalence on  $X=\1_{\CB H}$.

We write $\TwoGlRingTpregen{\CE}{\CT}$ for the $\infty$-category of $\CT$-genuine global 2-rings and $\CT$-genuine morphisms.
\end{definition}

\begin{proposition}\label{prop:unravel_funct_pregen}
Consider the functor 
\[
\mathbb{H}\colon \NaiveTwoGlRing{\CE} \rightarrow \Fun(\Glo{\CE}^{\op},\Fun^{\oplax}([1],\Cat^{\lax,\otimes}))
\]
of \Cref{prop:funct_unfurl}. After restricting the source to $\TwoGlRingTpregen{\CE}{\CT}$, and restricting the result in the target to $\CT$ and compact equivariant spaces (as in \Cref{thm:Gamma_strong_strict}), $\mathbb{H}$ restricts to a functor 
\[
\mathbb{H}\colon \TwoGlRingTpregen{\CE}{\CT}\rightarrow \Fun(\CT^{\op},\Fun([1],\Cat^{\otimes}_{\lex}))
\]
\end{proposition}

\begin{proof}
Because 
\[\Fun(\CT^{\op},\Fun([1],\Cat^{\otimes}_{\lex}))\subset \Fun(\CT^{\op},\Fun^{\oplax}([1],\Cat^{\lax,\otimes}))\] is a subcategory, it suffices to prove the statement on objects and morphisms. The claim on objects is precisely \Cref{thm:Gamma_strong_strict}. On morphisms it follows immediately from the assumptions placed on $\CT$-pregenuine morphisms, by \Cref{rmk:unwind_funct_unraveling}.
\end{proof}

\section{Cohomology theories on equivariant spectra}\label{sec:inverting}

In this section we impose a further condition on a pregenuine global 2-ring $\CR$ such that $\decatcoh{G}{-}{\CR}$ extends to a cohomology theory on $\Sp_G$, the $\infty$-category of genuine $G$-spectra. First we extend our cohomology theories to pointed spaces. Namely observe that because the categories $\CR_{\CB G}$ are all pointed, the lax symmetric monoidal functors $\decatcoh{G}{-}{\CR}\colon \Spc_G^{\op}\rightarrow \CR_{\CB G}$ canonically lift to symmetric monoidal functors out of the $\infty$-category of pointed $G$-spaces, in such a way that $\decatcoh{G}{X_+}{\CR}:= \decatcoh{G}{X}{\CR}$ for any $G$-space $X$. In the following proposition we ensure such an assignment can be made coherent.

\begin{proposition}
Let $\CR$ be a naive global 2-ring. The diagram 
\[
\decatcoh{\bullet}{-}{\CR}\colon \Glo{\CE}^{\op}\rightarrow \Fun^{\mathrm{oplax}}([1],\tCat^{\otimes,\lax}), \quad \CB G \mapsto [\decatcoh{G}{-}{\CR}\colon \Spc_G^{\op}\rightarrow \CR_{\CB G}]
\]
    extends to a functor
    \[
    \decatcoh{\bullet}{-}{\CR}\colon \Glo{\CE}^{\op}\rightarrow \Fun^{\mathrm{oplax}}([1],\tCat^{\otimes,\lax}), \quad \CB G \mapsto [\decatcoh{G}{-}{\CR}\colon \Spc_{G,\ast}^{\op}\rightarrow \CR_{\CB G}]
    \] such that precomposing with $(-)_+\colon\Spc_{\bullet} \to \Spc_{\bullet,\ast}$ one recovers the original functor. 
\end{proposition}

\begin{proof}
    For this proof we instead view $\decatcoh{G}{-}{\CR}$ as a functor from $\Spc_{G}$ to $\CR_G^{\op}$, in which case it is a colimit preserving functor of cocomplete categories. Consider $[1]$ as a symmetric monoidal $\infty$-category via the coproduct, also known as $\mathrm{max}$. Given a cocomplete $\infty$-category $\CC$, one may identify $\CC_\ast$ as a symmetric monoidal subcategory of $\Fun([1],\CC)$, equipped with the Day convolution symmetric monoidal structure. Since Day convolution is a 2-functor 
    \[\Fun([1],-)\colon \tCat_\mathrm{L}^{\otimes,\lax} \rightarrow \tCat_\mathrm{L}^{\otimes,\lax},\] 
    we conclude that $(-)_\ast\colon \tCat^{\otimes,\lax}_\mathrm{L}\rightarrow \tCat^{\otimes,\lax}_\mathrm{L}$ is also 2-functorial. The composite 
    \[\Glo{\CE}^{\op}\xrightarrow{\decatcoh{\bullet}{-}{\,\CR}} \Fun^{\mathrm{oplax}}([1],\tCat^{\otimes,\lax}_\mathrm{L})\xrightarrow{\Fun^{\mathrm{oplax}}([1],(-)_\ast)} \Fun^{\mathrm{oplax}}([1],\tCat^{\otimes,\lax}_\mathrm{L})\] is our desired extension.
\end{proof}

We would like to extend our cohomology theory further to take values in $G$-spectra. This involves constructing representation sphere deloopings for the functor $\decatcoh{G}{-}{\CR}$. To explain how we will accomplish this, we recall the process of inverting objects in a symmetric monoidal $\infty$-category and its connection to genuine $G$-spectra. 

\subsection{Inverting representation spheres}

\begin{definition}
We define the $\infty$-category $\Cat^{\otimes,\aug}$ to be the $\infty$-category of symmetric monoidal $\infty$-categories $\CC$ equipped with an augmentation, i.e.~a set $S$ of objects of $\CC$. The morphisms of $\Cat^{\otimes,\aug}$ are strong monoidal functors which preserve the augmentations. Formally we may define $\Cat^{\otimes, \aug}$ as the following pullback:
    \[\begin{tikzcd}
        {\Cat^{\otimes,\aug}} & {\Ar(\Cat)} \\
        {\mathrm{Set}\times \Cat^\otimes} & {\Cat\times \Cat.}
        \arrow[from=1-1, to=1-2]
        \arrow["{\ev_0\times\ev_1}", from=1-2, to=2-2]
        \arrow["{\mathrm{incl} \times \fgt}"', from=2-1, to=2-2]
        \arrow[from=1-1, to=2-1]
    \end{tikzcd}\]
\end{definition}
Recall that an object $X$ in a symmetric monoidal $\infty$-category is called invertible if $X\otimes (-)\colon \CC\rightarrow \CC$ is an equivalence. 
\begin{definition}
    We define $\Cat^{\otimes,\aug^{-1}}$ to be the full subcategory of $\Cat^{\otimes,\aug}$ spanned by those pairs $(\CC,S)$ such that every object $X\in S$ is an invertible object of $\CC$.
\end{definition}

\begin{proposition}\label{prop:invert_augmentation}
The inclusion $\Cat^{\otimes,\aug^{-1}}\subset \Cat^{\otimes,\aug}$ admits a left adjoint $\mathbb{I}$.
\end{proposition}

\begin{proof}
Note that $\Cat^{\otimes,\aug}$ is a pullback in $\mathrm{Pr}_{\mathrm{L}}$ and so is itself a presentable $\infty$-category. Observe that the subcategory $\Cat^{\otimes,\aug^{-1}}$ is equivalent to the subcategory of objects $(\CC,S)$ which are right orthogonal to the group completion map $\FinSet^{\simeq}\rightarrow \Omega^{\infty} \mathbb{S}$, viewed as a map of augmented symmetric monoidal $\infty$-categories by augmenting every object of both $\FinSet^{\simeq}$ and $\Omega^{\infty}\mathbb{S}$. Therefore the result follows from \cite{HTT}*{Proposition 5.5.4.15}.
\end{proof}

\begin{definition}
We define $\Cat^{\otimes,\aug}_{\rex,+}$ to be the (non-full) subcategory of $\Cat^{\otimes,\aug}$ spanned on objects by those pairs $(\CC,S)$ such that $\CC$ is idempotent complete, admits finite colimits and the tensor product commutes with finite colimits in each variable. On morphisms it is spanned by those strong monoidal functors which preserve finite colimits. We define $\Cat_{\rex,+}^{\otimes,\aug^{-1}}$ to be the full subcategory spanned by those objects $(\CC,S)$ in $\Cat^{\otimes,\aug^{-1}}$. Similarly we may define $\mathrm{Pr}_{\mathrm{L}}^{\otimes, \aug}$ and $\mathrm{Pr}_{\mathrm{L}}^{\otimes, \aug^{-1}}$, as well as $\Cat^{\otimes,\aug}_{\rex}$ and $\Cat^{\otimes,\aug^{-1}}_{\rex}$.  
\end{definition}

Using \Cref{prop:invert_augmentation} one can easily prove:

\begin{proposition}
    The inclusions \[\Cat^{\otimes,\aug^{-1}}_{\rex,+}\rightarrow \Cat^{\otimes,\aug}_{\rex,+} \text{,}\quad \Cat^{\otimes,\aug^{-1}}_{\rex}\rightarrow \Cat^{\otimes,\aug}_\rex \quad \text{and} \quad \mathrm{Pr}_{\mathrm{L}}^{\otimes, \aug}\rightarrow \mathrm{Pr}_{\mathrm{L}}^{\otimes, \aug^{-1}}\] admits left adjoints, which we denote by $\mathbb{I}^{\rex,+}$, $\mathbb{I}^{\rex}$, and $\mathbb{I}^{\mathrm{L}}$ respectively.\qed
\end{proposition}

\begin{proof}
We first argue the presentable case. Write $\mathbb{P}(\CC)$ for the free symmetric monoidal $\infty$-category on an $\infty$-category $\CC$. We then define 
\[
\mathbb{I}^{\mathrm{L}}(\CC,S) \coloneqq \CC \coprod_{\mathrm{PSh}(\mathbb{P} (S))} \mathrm{PSh}(\mathbb{I}(\mathbb{P}(S),S))\] as a pushout in $\PrL$. By the symmetric monoidal universal property of taking presheafs, see \cite{HA}*{Proposition 4.8.1.10(3)}, this clearly has the required universal property.  For the other cases, the same proof applies by replacing the presheaf category by the relevant cocompletion functor. 
\end{proof}

\begin{definition}
    Let $G$ be a compact Lie group. Given a $G$-representation $V$, we define $S^V = V\cup \{\infty\}$ to be the one point compactification of $V$, viewed as a pointed $G$-space by choosing $\{\infty\}$ as a basepoint. We enhance the $\infty$-category $\Spc_{G,\ast}^{\omega}$ of compact pointed $G$-spaces to an augmented symmetric monoidal $\infty$-category by the set $\mathrm{RepSph}_G \subset \Spc_{G,\ast}^{\omega}$ of pointed $G$-spaces which are weak homotopy equivalent to a representation sphere $S^V$.
\end{definition}

Key for us will be the following calculation:

\begin{proposition}\label{thm:Sp_fin_univ}    
    Let $G$ be a compact Lie group. Then $\Sp_{G}^{\omega} \simeq \mathbb{I}^{\rex,+}(\Spc_{G,\ast}^{\omega},\mathrm{RepSph}_G)$.
\end{proposition}

\begin{proof}
    By \cite{GM20}*{Appendix C}, the large $\infty$-category $\Sp_{G}$ is the initial presentably symmetric monoidal $\infty$-category under $\Spc_{G,\ast}$ such that the representation spheres are invertible in $\Sp_G$, that is $\Sp_G \simeq \mathbb{I}^{\mathrm{L}}(\Spc_{G,\ast},\mathrm{RepSph}_G)$. However $\Sp_G$ is compactly generated, and so is equivalent to $\mathrm{Ind}(\Sp^{\omega}_{G})$ as a symmetric monoidal $\infty$-category. Similarly $\Spc_{G,\ast}\simeq \mathrm{Ind}(\Spc^{\omega}_{G,\ast})$. Using the symmetric monoidal universal property of the Ind-functor we conclude that $\mathbb{I}^{\rex,+}(\Spc_{G,\ast}^{\omega},\mathrm{RepSph}_G)$ agrees with $\Sp_G^{\omega}$.
\end{proof}

We can apply the previous proposition levelwise to extend families of equivariant cohomology theories.

\begin{corollary}\label{cor:univ_prop_suspension}
    Let $\CT\subset \Glo{\CE}$ be some full subcategory and let $\Lambda_\bullet \colon \Spc_{\bullet,\ast}^{\omega}\Rightarrow \CC_\bullet$ be a natural transformation of functors $\CT^{\op}\rightarrow \Cat_{\rex,+}^\otimes$ such that $\Lambda_G(S^V)$ is invertible for every group $G\in \CT$ and representation sphere $S^V\in \Spc_{G}^{\omega}$. Then there exists a unique dotted natural transformation in $\Cat^\otimes_{\rex, +}$ such that the triangle 
    \[
    \begin{tikzcd}
        {\Spc_{\bullet,\ast}^{\omega}} \\
        {\Sp_{\bullet}^\omega} & {\CC_\bullet.}
        \arrow["{\Lambda_\bullet}", Rightarrow, from=1-1, to=2-2]
        \arrow["{\exists!}"', Rightarrow, dashed, from=2-1, to=2-2]
        \arrow["{\Sigma^\infty_\bullet}"', Rightarrow, from=1-1, to=2-1]
    \end{tikzcd}
    \] commutes.
\end{corollary}

\begin{proof}
Note that for every group homomorphism $\alpha\colon H\rightarrow G$, the restriction functor $\alpha^*\colon \Spc_{G,\ast}\rightarrow \Spc_{H,\ast}$ preserves representation spheres. Therefore $\Spc^{\omega}_{\bullet,\ast}\colon \Glo{}^{\op}\rightarrow \Cat^{\otimes}_{\rex,+}$ lifts to a functor into $\Cat^{\otimes,\aug}_{\rex,+}$. Therefore the corollary follows immediately from the previous proposition by viewing $\CC_\bullet$ as a functor into $\Cat^{\otimes,\aug^{-1}}_{\rex,+}$ by augmenting $\CC_{\CB G}$ by $\mathrm{Pic}(\CC_{\CB G})$, the set of invertible objects in $\CC_{\CB G}$.
\end{proof}

\subsection{Extending to equivariant spectra}

As mentioned, we would like to apply the previous results to extend the cohomology theories $\decatcoh{G}{-}{\CR}$ to cohomology theories defined on genuine $G$-spectra. For this we have to assume that the representation spheres are inverted by our cohomology theories. 

\begin{definition}\label{def:T-genuine}
We say that a naive $2$-ring $\CR$ is $\CT$-\emph{genuine} if it is $\CT$-pregenuine and
\begin{enumerate}
\setcounter{enumi}{2}
 \item\label{item:gen-rep-sph} For every group $G\in \CT$ and every irreducible $G$-representation $V$, the object $\decatcoh{G}{S^V}{\CR} \in \CR_{\CB G}$ is invertible.
\end{enumerate}
\end{definition}

For certain statements, such as \Cref{thm:Represented-theorem}, it is in fact not relevant exactly which subcategory $\CT$ of enough injective objects a naive global 2-ring is $\CT$-genuine for, and so we make the following:

\begin{definition}\label{def:genuine}
We say a naive global 2-ring $\CR$ is \emph{genuine} if it there is a subcategory $\CT$ of enough injective objects such that $\CR$ is $\CT$-genuine. We will typically refer to a genuine global 2-ring simply as a \emph{global 2-ring}.
\end{definition}

We now make the object $\decatcoh{G}{S^V}{\CR}$, and so the condition of the previous definition, more explicit.

\begin{lemma}\label{gamma-sphere}
Let $V$ be a $G$-representation and let $p \colon S(V)\sslash G \to \CB G$ be the canonical faithful map of global spaces exhibiting $S(V)$, the unit sphere in $V$, as a $G$-space. Then the fiber of the unit map $\eta \colon \1_{\CR_{\CB G}} \to p_*p^* \1_{\CR_{\CB G}}$ is equivalent to $\decatcoh{G}{S^V}{\CR}$.
\end{lemma} 

\begin{proof}
We may consider the cofiber sequence $S(V)_+ \to S^0 \to S^V$ of pointed $G$-spaces. Because the extension of $\decatcoh{G}{-}{\CR}$ to pointed spaces is again limit preserving, we obtain a fiber sequence
\[
\decatcoh{G}{S^V}{\CR}\coloneqq \mathrm{fib}(\decatcoh{G}{S^0}{\CR} \to \decatcoh{G}{S(V)_+}{\CR}) \in \CR_{\CB G}.
\]
unraveling the definition, we find that $\decatcoh{G}{\pt}{\CR}= \1_{\CR_{\CB G}}$ and $\decatcoh{G}{S(V)}{\CR} =p_* \1_{\CR_{S(V)\sslash G}}=p_*p^* \1_{\CR_{\CB G}}$. Moreover the map $\decatcoh{G}{\pt}{\CR}\to \decatcoh{G}{S(V)}{\CR}$ is induced by the lax unit map of the lax monoidal functor $p_*$, which is in turn induced by $p_*$ being a right adjoint of $p^*$. It then follows that $\decatcoh{G}{\pt}{\CR}\to \decatcoh{G}{S(V)}{\CR}$ can be identified with the unit map $\eta \colon \1_{\CR_{\CB G}} \to p_*p^* \1_{\CR_{\CB G}}$.
\end{proof}

The key result of this section is the following corollary, which by a representability argument that we carry out in the following section will allow us to deduce \Cref{introthm:gen_ref}.

\begin{corollary}\label{prop-Fsp}
Let $\CR$ be a $\CT$-genuine global 2-ring. There exists a unique functor
\[
\decatcoh{\bullet}{-}{\CR}\colon \CT^{\op}\rightarrow \Fun([1],\Cat^{\otimes}_{\lex,+}), \quad \CB G\mapsto [\decatcoh{G}{-}{\CR}\colon (\Sp^{\omega}_{G})^{\op}\rightarrow \CR_{\CB G}]
\] 
which agrees with the functor $\decatcoh{\bullet}{-}{\CR}$ of \Cref{thm:Gamma_strong_strict} when restricted to suspension spectra. 
\end{corollary}

\begin{proof}
Let us show that all representation spheres are inverted by $\decatcoh{\bullet}{-}{\CR}$. First we note that the set of invertible objects in $\CR_{\CB G}$ is closed under two out of three,~i.e. given two objects $X,Y\in \CR_{\CB G}$, if two of $X$, $Y$ and $X\otimes Y$ are invertible then so is the third. If $V$ and $W$ are irreducible $G$-representations, then we know by axiom (3) that $\decatcoh{G}{S^V}{\CR}$ and $\decatcoh{G}{S^W}{\CR}$ are invertible objects. Using \Cref{gamma_G-strong-mon} we conclude that 
\[
\decatcoh{G}{S^{V \oplus W}}{\CR}=\decatcoh{G}{S^V \otimes S^W}{\CR}\simeq \decatcoh{G}{S^V}{\CR} \otimes \decatcoh{G}{S^W}{\CR} 
\]
is also an invertible object. Because every $G$-representation is a direct sum of irreducible representations, the result follows. Given this, the result follows immediately from \Cref{cor:univ_prop_suspension} applied to (the pointwise opposite of) the functor of \Cref{thm:Gamma_strong_strict}. 
\end{proof}

\begin{remark}\label{rem:extend_to_groups_not_in_T}
Let $\CR$ be a $\CT$-genuine global 2-ring. It is unlikely that the cohomology theory $\decatcoh{G}{-}{\CR}\colon (\Spc_G^\omega)^{\op}\rightarrow \CR_{\CB G}$ extends to a cohomology theory on $\Sp_G^\omega$ when $G\notin \CT$ in general. 

However, there is a specific case where this is possible. Recall that equivariant elliptic cohomology defines an $\ab$-global naive 2-ring $\CQ_\bullet$. We will see that $\CQ_\bullet$ is $\Tori$-genuine, but \emph{not} $\Glo{\ab}$-genuine, in both cases in the sense of \Cref{def:T-genuine}. Therefore the previous theorem only guarantees the lift of equivariant elliptic cohomology to a cohomology theory on equivariant spectra at tori. Nevertheless, it is in fact possible to define such extensions for any compact abelian Lie group. The general construction would follow the strategy of \cite{Chau2021}*{Lemma 2.1} for the group $C_n$. In particular it would in principle depend on a choice of embedding of $G$ into a torus, however we expect that one can show that it is coherently independent of this choice.
\end{remark}

Of course we may also observe the effect of this construction on morphisms of $\CT$-genuine global 2-rings. 

\begin{definition}
We define $\TwoGlRingTgen{\CE}{\CT} \subset \TwoGlRingTpregen{\CE}{\CT}$ to be the full subcategory spanned by the $\CT$-genuine 2-rings.
\end{definition}

The following result is immediate.

\begin{proposition}\label{prop:funct-unravel-gen}
The functor
\[
\mathbb{H}\colon \TwoGlRingTpregen{\CE}{\CT}\rightarrow \Fun(\CT^{\op},\Fun([1],\Cat^{\otimes}_{\lex}))
\] of \Cref{prop:unravel_funct_pregen} induces a functor
\[
\mathbb{H}\colon \TwoGlRingTgen{\CE}{\CT}\rightarrow \Fun(\CT^{\op},\Fun([1],\Cat^{\otimes}_{\lex,+}))
\]
which agrees on objects with the construction of \Cref{prop-Fsp}.
\qednow
\end{proposition}

Let us finally discuss various simplifications of (3) above which apply in specific instances. One case is when $\CE$ is the family of abelian compact Lie groups. For what follows we denote by $\T$ the torus of rank one and $\tau$ for the tautological representation, which is represented by the identity group homomorphism $\mathrm{id}\colon\T\to\T$.

\begin{proposition}\label{prop-abelian-compact}
Let $\mathrm{Ab}$ be the global family of \emph{abelian} compact Lie groups, let $\Tori \subset \CE$ be the full subcategory of enough injective objects spanned by the tori, and let $\CR$ be a naive global 2-ring which is $\Tori$-pregenuine. Then axiom (\ref{item:gen-rep-sph}) is equivalent to:
\begin{itemize}
    \item[(3')] The object $\decatcoh{\T}{S^\tau}{\CR}$ is invertible in $\CR_{\CB \T}$.
\end{itemize}
\end{proposition}

\begin{proof}
Clearly (\ref{item:gen-rep-sph}) implies (3') so let us prove the converse. First we note that the set of invertible objects in $\CR_{\CB G}$ is closed under two out of three,~i.e. given two objects $X,Y\in \CR_{\CB G}$, if two of $X$, $Y$ and $X\otimes Y$ are invertible then so is the third. Because every real representation is a direct summand of a complex representation it therefore suffices to prove that $\decatcoh{G}{S^V}{\CR}$ is invertible for $V$ an irreducible complex $G$-representation. Because $G$ is abelian, every irreducible complex $G$-representation is one dimensional, i.e., given by a character $\alpha\colon G \to \T$. By \Cref{lem:Q_equivalence} we have equivalences
    \[\decatcoh{G}{S^\alpha}{\CR} = \decatcoh{G}{S^{\alpha^*\tau}}{\CR} \simeq \decatcoh{G}{(\alpha^* S^{\tau}}{\CR} \xrightarrow{Q_\alpha} \alpha^*\decatcoh{\T}{S^\tau}{\CR}.\] 
    Because $\alpha^*\colon \CR_{\CB \T}\rightarrow \CR_{\CB G}$ is strong monoidal (see \Cref{gamma_G-strong-mon}), it preserves invertible objects. Therefore it suffices to prove that $\decatcoh{\T}{S^\tau}{\CR}$ is an invertible object, which is precisely (3').
\end{proof}

\begin{remark}
    Let $\T$ and $\tau$ be as above; then $S(\tau)\simeq \T$ and the map $p\colon S(\tau)\sslash \T \to \CB \T$ identifies with the map $i\colon \pt \to \CB \T$ induced by the inclusion $e \to \T$. 
    Because tensoring in $\CR_{\CB \T}$ is exact, we find that tensoring with $\decatcoh{\T}{S^\tau}{\CR}$ is equivalent to the functor
    \[
    \mathrm{fib}\big( (-) \otimes \1_{\CR_{\CB \T}} \rightarrow (-)\otimes i_*\1_{\CR_{\pt}}\big),
    \] 
    which is in turn equivalent to the functor $\mathrm{fib}\big(\eta\colon (-) \rightarrow i_*i^*(-)\big)$ by an application of the right projection formula for $p$ (which holds by axiom (2) of a global 2-ring). By assumption $\decatcoh{\T}{S^\tau}{\CR}$ is invertible, and so this functor is an equivalence. We conclude that $i^*\dashv i_*$ is a spherical adjunction, in the sense of \cite{Spherical21}.
\end{remark}

Similarly if $\CE$ is a family of finite abelian groups and $\CT = \CE$, then we can again simplify axiom (\ref{item:gen-rep-sph}). We write $\iota_n\colon C_n\rightarrow \T$ for the canonical inclusion, which defines a 1-dimensional complex $C_n$-representation.

\begin{proposition}\label{prop:finite-abelian}
    Let $\CE$ be a global family of \emph{finite abelian} groups and let $\CR$ be a naive global 2-ring which is $\CE$-pregenuine. Then axiom (\ref{item:gen-rep-sph}) is equivalent to:
    \begin{itemize}
        \item[(3')] The object $\decatcoh{C_n}{S^{\iota_n}}{\CR}$ is invertible in $\CR_{\CB C_n}$ for all $C_n\in \CE$.
    \end{itemize}
\end{proposition}
\begin{proof}
    Clearly (\ref{item:gen-rep-sph}) implies (3'). Conversely, as before we can reduce to showing that for all $G\in \CE$, $\decatcoh{G}{-}{\CR}$ inverts spheres associated to irreducible complex representations $V\colon G\rightarrow \T$ of $G$. We factor $V$ into a surjection $p$ followed by an injection, which is necessarily of the form $\iota_n\colon C_n\rightarrow \T$ for some $n$. As before we obtain an equivalence:
    \[
    \decatcoh{G}{S^V}{\CR} = \decatcoh{G}{S^{p^*\iota_n}}{\CR} \simeq \decatcoh{G}{p^* S^{\iota_n}}{\CR} \xrightarrow{Q_\alpha} p^*\decatcoh{C_n}{S^{\iota_n}}{\CR}.
    \] Because $p^*$ is strong monoidal we conclude that $\decatcoh{G}{S^V}{\CR}$ is invertible.
\end{proof}

\subsection{A canonical class of examples}

Finally, before we discuss more interesting examples in \Cref{examples}, let us first discuss a canonical class of examples of global 2-rings.

\begin{lemma}\label{lem-sp-example}
Let $\CE$ be any  multiplicative global family. Then the naive global 2-ring $\Sp_\bullet$ from \Cref{ex:genuine_spectra} is $\CE$-pregenuine.
\end{lemma}

\begin{proof}
The fact that $\Sp_\bullet$ is $\CE$-pregenuine expresses a collection of rather standard properties of the diagram of genuine equivariant spectra; nevertheless, let us give references and arguments here. Restricting the diagram $\Sp_\bullet$ to the subcategory $\Orb{\CE}$ we obtain a diagram which is equivalent to $\underline{\mathrm{OrbSp}}$, as defined in \cite{Cnossen22}. It follows from Proposition 4.17 of \emph{op.~cit.}~that $\Sp_\bullet$ satisfies condition (\ref{item:gen-proj}), and that it satisfies (\ref{item:gen-adjoint}) for faithful $g$. Let us argue the case when $\CB g\colon \CB G \rightarrow \CB K$ is a quotient map. By applying the Wirthm\"uller isomorphism we reduce to the analogous Beck--Chevalley statement for induction, which is discussed in \cite{LinskensGlobalization}*{Example 2.18}.
\end{proof}

From \Cref{lem:Q_equivalence} we obtain a functor $\decatcoh{\bullet}{-}{\Sp_\bullet}\colon \Glo{\CE}^{\op} \rightarrow \Fun([1],\Cat_\infty)$.

\begin{lemma}\label{lem-cohomotopy}
The functor $\decatcoh{\bullet}{-}{\Sp_\bullet}$ agrees with the $\Glo{\CE}^{\op}$-indexed natural transformation 
\[
(\Spc_G^\omega)^{\op} \xRightarrow{\Sigma^\infty_G} \Sp_G^{\op} \xRightarrow{\ul{\Hom}(-,S^0_G)} \Sp_G.
\]
\end{lemma}

\begin{remark}
This lemma shows that the unraveling associated to $\Sp_\bullet$ agrees with the family of equivariant Spanier--Whitehead duality functors.
\end{remark}

\begin{proof}
To prove this statement it is easiest to apply the parameterized perspective. As noted in \Cref{rem:para-perspective}, $\Sp_\bullet$ admits finite equivariant limits. Therefore, by the dual of \cite{CLL_Partial}*{Corollary 4.27} applied to $T = \Glo{\CE}$ and $S = \Orb{\CE}$, a natural transformation $F_\bullet\colon (\Spc_\bullet^\omega)^{\op}\Rightarrow \Sp_\bullet$ such that
\begin{itemize}
\item each functor $F_G$ preserves finite limits;
\item and for each faithful morphism $f\colon \CB H\rightarrow \CB G$ the Beck--Chevalley transformation $F_G f_! \Rightarrow f_* F_H$ is an equivalence;
\end{itemize}
is determined by the object $F_e(\ast)\in \Sp$. By \Cref{rem:Talpha-equivalence} and the definition of naive global $2$-ring the transformation $\decatcoh{\bullet}{-}{\Sp_\bullet}$ satisfies both properties, while a simple calculation shows that $\underline{\Hom}(\Sigma^\infty_\bullet(-),S^0_G)$ does as well. In both cases $\ast$ is sent to the sphere spectrum $S^0\in \Sp$, and so the two natural transformations must agree. 
\end{proof}

\begin{proposition}
The global 2-ring $\Sp_\bullet$ is $\CE$-genuine.
\end{proposition}

\begin{proof}
By \Cref{lem-sp-example}, $\Sp_\bullet$ is $\CE$-pregenuine. Moreover by the previous lemma, $\decatcoh{G}{-}{\Sp_\bullet}$ agrees with $\underline{\Hom}(\Sigma^{\infty}-, S^0_G)$. In particular $\decatcoh{G}{S^V}{\Sp_\bullet}$ is equivalent to $S^{-V}$ and so clearly invertible, showing that $\Sp_\bullet$ is $\CE$-genuine.
\end{proof}

\begin{proposition}\label{prop-genuine-from-global}
Suppose $E\in \CAlg(\Spgl{\CE})$ is a global ring. Then the naive global 2-ring $\Mod_{E_\bullet}$ defined in \Cref{ex:cat_gl_cohom_of_gl_spectrum} is $\CE$-genuine and the unraveling $\decatcoh{\bullet}{-}{\Mod_{E_\bullet}}$ agrees with the natural transformation
\[
(\Spc_\bullet^\omega)^{\op}\xRightarrow{\Sigma^\infty_\bullet} \Sp_\bullet^{\op}\xRightarrow{\underline{\Hom}(-,E)} \Sp_\bullet.
\]
\end{proposition}

\begin{proof}
Observe that axioms (\ref{item:gen-adjoint}) and (\ref{item:gen-proj}) of a global 2-ring reduce to showing that certain canonical maps are equivalences, and therefore we can check these axioms forgetting the module structures. However we note that for any continuous group homomorphism 
$\alpha \colon H\to G $, there are commutative diagrams 
\[
\begin{tikzcd}
        \Mod_{E_G}(\Sp_G) \arrow[r,"\alpha^*"] \arrow[d,"\fgt"']& \Mod_{E_H}(\Sp_H) \arrow[d, "\fgt"]& & \Mod_{E_G}(\Sp_G) \arrow[r,"\alpha_*"] \arrow[d,"\fgt"']& \Mod_{E_H}(\Sp_H) \arrow[d, "\fgt"] \\
        \Sp_G \arrow[r,"\alpha^*"] & \Sp_H & & \Sp_G \arrow[r,"\alpha_*"] & \Sp_H	
    \end{tikzcd}
    \]
    see for instance~\cite{BCHNP}*{Recollection 2.29}. Therefore after forgetting the resulting maps agree with those of~\Cref{lem-sp-example}, and are in particular equivalences. The remaining arguments are exactly as for $\Sp_\bullet$.
\end{proof}

\section{Genuine refinements of naive global rings}

Recall that given a naive global 2-ring $\CR$, its decategorification was a naive global ring which represented the multiplicative cohomology theory $\decatcohsp{\gl}{-}{\CR}\colon \Spcgl{\CE}^{\op}\rightarrow \CAlg$. In this section we will show that if $\CR$ is genuine, then this naive global ring admits a canonical genuine refinement. In other words, we will explain how to decategorify a global 2-ring to a global ring. Recall from \Cref{prop:family-of-eq-vs-global}(1) that a genuine refinement was a global ring $\genref{\gl}{\CR}\in \CAlg(\Spgl{\CE})$ together with an equivalence
\[
\decatcohsp{\gl}{-}{\CR} \simeq \Hom_{\Spgl{\CE}}(\Sigma^\infty_\gl(-), \genref{\gl}{\CR})
\]
between $\decatcohsp{\gl}{-}{\CR}$ and the multiplicative cohomology theory represented by $\genref{\gl}{\CR}$. To begin we will first construct the global ring $\genref{\gl}{\CR}$. Therefore let us now assume that $\CR$ is $\CT$-genuine for $\CT\subset \CE$ some family of enough injective objects.

\begin{construction}\label{con:a_object_in_oplaxlim_of_fun}
    Passing the natural transformation of Lemma~\ref{lem-functor-glsec} along the second mate equivalence (\ref{second-mate-equivalence}) gives an oplax natural transformation 
   \[\begin{tikzcd}
	{\CR_{\CB G}} & \Sp \\
	{\CR_{\CB H}} & \Sp.
	\arrow["{\alpha^*}"', from=1-1, to=2-1]
	\arrow["{\Gamma}", from=1-1, to=1-2]
	\arrow["{\Gamma}"', from=2-1, to=2-2]
	\arrow[Rightarrow, no head, from=1-2, to=2-2]
	\arrow[shorten <=7pt, shorten >=7pt, Rightarrow, from=1-2, to=2-1]
    \end{tikzcd}\]
    In fact, applying the generalized mate equivalence of \cite{Mate2024}*{Theorem 5.3.6} to the $(\infty,2)$-category $\tCat^{\otimes,\lax}$, we obtain $\Gamma$ as an oplax natural transformation of lax monoidal functors. Pasting this with the natural transformation $\decatcoh{\bullet}{-}{\CR}$ of strong monoidal functors from \Cref{prop-Fsp} we obtain a new oplax natural transformation
    \[\begin{tikzcd}[column sep =large]
	{(\Sp_G^\omega)^{\op}} & {\CR_{\CB G}} & \Sp \\
	{(\Sp_H^\omega)^{\op}} & {\CR_{\CB H}} & \Sp
	\arrow["{\alpha^*}"', from=1-2, to=2-2]
	\arrow["{\Gamma}", from=1-2, to=1-3]
	\arrow["{\Gamma}"', from=2-2, to=2-3]
	\arrow[Rightarrow, no head, from=1-3, to=2-3]
	\arrow[shorten <=9pt, shorten >=9pt, Rightarrow, from=1-3, to=2-2]
	\arrow["{\alpha^*}"', from=1-1, to=2-1]
	\arrow["{\decatcoh{G}{-}{\,\CR}}", from=1-1, to=1-2]
	\arrow["{\decatcoh{H}{-}{\,\CR}}"', from=2-1, to=2-2]
\end{tikzcd}\]
    in $\tCat^{\otimes,\lax}$ which we denote by $\decatcohsp{\bullet}{-}{\CR}$. The target of this oplax natural transformation is constant at $\Sp$ and so $\decatcohsp{\bullet}{-}{\CR}$ is equivalently an object of $\oplaxlim_{\CT}(\Fun^{\lex,\otimes-\lax}((\Sp_\bullet^\omega)^{\op},\Sp))$, see \Cref{const:functors_into_lax_slice}. Finally, of we equip the functor categories with the localized Day convolution symmetric monoidal structure of \Cref{lem:day_conv_localization}, then this equivalently defines an object $\oplaxlim_{\CT}(\CAlg(\Fun^{\lex}((\Sp_\bullet^\omega)^{\op},\Sp))$ via the universal property of Day convolution, see~\cite{HA}*{Example 2.2.6.9}.
\end{construction}

We now reinterpret this construction by applying higher Brown representability, in the following form

\begin{proposition}\label{lem-eq}
    The Yoneda embedding induces an equivalence $\Sp_\bullet \simeq\Fun^{\lex}((\Sp_\bullet^\omega)^{\op}, \Sp)$, which is moreover symmetric monoidal if the right hand side is given the localized Day convolution symmetric monoidal structure of \Cref{lem:day_conv_localization}. 
    Moreover this equivalence is natural: for $\CB\alpha \colon\CB H \to \CB G$, the following diagram commutes
    \[
    \begin{tikzcd}
        \Sp_H \arrow[d,"\alpha_*"] \arrow[r,"y","\sim"'] & \Fun^{\lex}((\Sp_H^\omega)^{\op}, \Sp) \arrow[d,"(\alpha^*)^*"]\\
        \Sp_G \arrow[r,"y","\sim"'] & \Fun^{\lex}((\Sp_G^\omega)^{\op}, \Sp).
    \end{tikzcd}
    \]
\end{proposition}
\begin{proof}
    This is the content of \Cref{lem:Brown_rep}, restricted to the diagram $\Sp_\bullet$.
\end{proof}

\begin{construction}
    The natural equivalence $\Sp_\bullet \simeq\Fun^{\lex}((\Sp_\bullet^\omega)^{\op}, \Sp)$ induces an equivalence \[\oplaxlim_{\CT} \CAlg(\Sp_\bullet) \simeq \oplaxlim_{\CT} \CAlg(\Fun^{\lex}((\Sp_
    \bullet^{\omega})^{\op},\Sp)).\] We pass the object $\decatcohsp{\bullet}{-}{\CR}$ along this equivalence to obtain an object of $\oplaxlim_{\CT} \CAlg(\Sp_\bullet)$. We denote the resulting object by $\{\genref{G}{\CR}\}$.
\end{construction}

Next we may pass from the oplax limit of $\Sp_\bullet$ along coinduction to the lax limit of $\Sp_\bullet$ along restriction.

\begin{construction}
    Applying \Cref{prop:adjoint_diag_lax_oplax} to the diagram $\CAlg(\Sp_\bullet)\colon \CT\rightarrow \Cat$ we obtain an equivalence
    \[
    \oplaxlim_{\CT} \CAlg(\Sp_\bullet) \simeq \laxlim_{\CT^{\op}} \CAlg(\Sp_\bullet),
    \]
    where the functoriality on the right hand side is given by restriction. Passing $\{\genref{G}{\CR}\}$ along this equivalence gives an object in $\laxlim \CAlg(\Sp_\bullet)$, which we again denote by $\{\genref{G}{\CR}\}$. 
\end{construction}

\begin{proposition}\label{prop-lies-partlaxlim}
    The object $\genref{\bullet}{\CR}$ lies in the full subcategory $\laxlimdag_{\CT^{\op}}\CAlg(\Sp_\bullet)$, where we mark the faithful maps in $\CT^{\op}$.
\end{proposition}

\begin{proof}
    Let $\alpha \colon \CB H\to \CB G$ be a faithful map in $\CT$. We need to verify that the lax structure map $f_\alpha$ associated to $\alpha$ is an equivalence. After passing to represented functors the adjoint of the lax structure map (i.e., the structure map in the oplax limit) is given by the composite of the natural transformations
    \[
    \begin{tikzcd}[column sep= large]
        {(\Sp_H^{\omega})^{\op}} & {\CR_{\CB H}} & \Sp. \\
        {(\Sp_G^{\omega})^{\op}} & {\CR_{\CB G}}
        \arrow["{\alpha^*}", from=2-1, to=1-1]
        \arrow["{\decatcoh{H}{-}{\,\CR}}", from=1-1, to=1-2]
        \arrow["{\decatcoh{G}{-}{\,\CR}}"', from=2-1, to=2-2]
        \arrow["\Gamma", from=1-2, to=1-3]
        \arrow[""{name=0, anchor=center, inner sep=0}, "\Gamma"', from=2-2, to=1-3]
        \arrow["{\alpha^*}", from=2-2, to=1-2]
        \arrow["{Q_\alpha}"', shorten <=9pt, shorten >=9pt, Rightarrow, from=2-2, to=1-1]
        \arrow[shorten <=2pt, shorten >=1pt, Rightarrow, from=0, to=1-2]
    \end{tikzcd}
    \] Note that when $\alpha$ is faithful, $\alpha^*$ admits a left adjoint $\alpha_!$. Some yoga with adjoint functors implies that the lax structure map is equivalent to the natural transformation obtained by passing to adjoints in the vertical direction. Doing so, we obtain the diagram
    \[
    \begin{tikzcd}[column sep= large]
        (\Sp_H^\omega)^{\op} \arrow[d,"\alpha_!"'] \arrow[r,"\decatcoh{H}{-}{\,\CR}"] & \CR_{\CB H}\arrow[d,"\alpha_*"] \arrow[r, "\Gamma"] & \Sp \\
        (\Sp_G^\omega)^{\op} \arrow[r,"\decatcoh{G}{-}{\,\CR}"'] & \CR_{\CB G}, \arrow[ur, "\Gamma"'] & 
        \arrow["T_\alpha", shorten <=9pt, shorten >=9pt, Rightarrow, from=2-1, to=1-2]
    \end{tikzcd}
    \]
where the right triangle commutes by \Cref{lem-functor-glsec}. The commutativity of the left square follows from \Cref{rem:Talpha-equivalence}, using that every functor in the diagram is exact and that $\Sp_H^\omega$ is generated as a stable $\infty$-category by the image of $\Sigma_+^{\infty}\colon \Spc_H^\omega \to \Sp_H^\omega.$ We conclude that the composite natural transformation is an equivalence, finishing the proof.
\end{proof}

\begin{definition}
Recall that \Cref{thm:P_gl_sp_from_tori} gave a symmetric monoidal equivalence
\[
\CAlg(\Spgl{\CE}) \simeq \laxlimdag_{\CT^{\op}} \CAlg(\Sp_\bullet).
\]
Passing $\genref{\bullet}{\CR}$ through this equivalence we obtain a global ring $\genref{\gl}{\CR}\in\CAlg(\Sp_{\CE\text{-}\gl})$.
\end{definition}

We can now conclude the main theorem of the paper:

\begin{theorem}\label{thm:Represented-theorem}
Let $\CR$ be a global 2-ring. Then the naive global ring $\decatcohsp{\gl}{-}{\CR}\colon \Spcgl{\CE}^{\op} \to \CAlg$ admits a canonical genuine refinement $\genref{\gl}{\CR}\in \CAlg(\Spgl{\CE})$.
\end{theorem}

\begin{proof}
We have to exhibit an equivalence between \[\decatcohsp{\gl}{-}{\CR} \simeq \Map_{\Spgl{\CE}}(\Sigma^{\infty}(-),\genref{\gl}{\CR}).\] However by \Cref{prop:family-of-eq-vs-global} it suffices to exhibit an equivalence between the family of equivariant cohomology theories associated to these global theories, which we note are $\{\decatcohsp{\gl}{-\sslash \bullet}{\CR}\}$  and $\{\decatcohsp{\bullet}{-}{\CR}\}$ respectively. The second of these identifications again uses \Cref{prop:family-of-eq-vs-global} and the definition of $\genref{\gl}{\CR}$. However these families agree by \Cref{obser:two_unravelings_agree}, and so we conclude. 
\end{proof}

For the final time, let us note the functoriality of this construction. 

\begin{proposition}\label{prop:funct_gen_ref}
There exists a functor $\Gamma_{\gl}(-)\colon \TwoGlRingTgen{\CE}{\CT} \rightarrow \CAlg(\Spgl{\CE})$, which agrees with the construction of \Cref{thm:Represented-theorem} on objects.
\end{proposition}

\begin{proof}
The constructions of this section are evidently functorial, and take as input the family of cohomology theories $\decatcoh{\bullet}{-}{G}$ on genuine $G$-spectra. This is a functor in $\TwoGlRingTgen{\CE}{\CT}$ by \Cref{prop:funct-unravel-gen}. 
\end{proof}

\begin{example}
Recall that given a global ring $R \in\CAlg(\Spgl{\CE})$, there was a canonical genuine categorification of $R$ given by 
\[
\CR :=\Mod_{R_\bullet}, \quad \CB G \mapsto \Mod_{\mathrm{res}_G R}(\Sp_G),
\] see \Cref{prop-genuine-from-global}. One can easily show that $\genref{\gl}{\Mod_{R_\bullet}} \simeq R$. In fact the assignment $R\mapsto \Mod_{R_\bullet}$ defines a functor $\CAlg(\Spgl{\CE}) \rightarrow \TwoGlRingTgen{\CE}{\CT}$, which is a section of the decategorification $\Gamma_{\gl}(-)$.
\end{example}

\section{Genuine global section functors}\label{sec:genuine_sections}

Suppose $\CR$ is a global 2-ring with respect to a family $\CT$ of enough injective objects. We have associated to $\CR$ a family of equivariant spectra $\genref{G}{\CR}$ for all $G\in \CT$. In this section we will more generally construct a genuine $G$-spectrum $\bGamma_G(\CF)$ associated to any object $\CF \in \CR_{\CB G}$. 

\begin{construction}\label{const:global_sect}
Suppose $\CR$ is a global $2$-ring which is genuine with respect to $\CT$. Then by \Cref{prop-Fsp} we have a natural transformation 
    \[
    \decatcoh{\bullet}{-}{\CR}\colon (\Sp_\bullet^\omega)^{\op}\rightarrow \CR
    \] 
    of functors $\CT^{\op}\rightarrow \Cat^{\otimes}_\lex$. We may precompose this natural transformation with equivariant Spanier-White duality to obtain a natural transformation of finite limit preserving functors
    \[
    \decatcoh{\bullet}{\CBD(-)}{\CR}\colon \Sp_\bullet^\omega \rightarrow \CR.
    \] However the source and target are stable, and so they are also finite colimit preserving. Passing to Ind-categories in the source gives a natural transformation of colimit preserving functors $F_\bullet \colon \Sp_\bullet\Rightarrow \CR$, which one may interpret as $\CR$-valued homology. We may further pass to right adjoints pointwise (via \Cref{second-mate-equivalence}) to obtain a lax natural transformation
    \[
    \bGamma_\bullet \colon \CR \rightarrow \Sp_\bullet.
    \]
    We write $R_f\colon f^*\bGamma_G(-)\Rightarrow \bGamma_H(f^*(-))$ for the natural transformation associated to $f\colon \CB H \rightarrow \CB G$:
    \[\begin{tikzcd}
        {\CR_{\CB G}} & {\Sp_G} \\
        {\CR_{\CB H}} & {\Sp_H.}
        \arrow["{f^*}"', from=1-1, to=2-1]
        \arrow["{\bGamma_G}", from=1-1, to=1-2]
        \arrow["{f^*}", from=1-2, to=2-2]
        \arrow["{\bGamma_H}"', from=2-1, to=2-2]
        \arrow["{R_f}", shorten <=6pt, shorten >=6pt, Rightarrow, from=1-2, to=2-1]
    \end{tikzcd}\]
\end{construction}
We record a few observations about this construction. 

\begin{remark}
    Consider $\CF\in \CR_{\CB G}$. We note that 
    \[
    \bGamma_G(\CF)^G\simeq\map_{\Sp_G}(\1, \bGamma_G(\CF))\simeq \Map_{\CR_{\CB G}}(\decatcoh{G}{\mathbb{D}\1}{\CR}),\CF) \simeq\Map_{\CR_{\CB G}}(\1,\CF)\simeq  \Gamma(\CF).
    \]
    In other words, the genuine fixed points of $\bGamma_G(\CF)$ agree with the global sections of $\CF$.
\end{remark}

\begin{remark}\label{rem:global_sec_of_unit}
    Let $G\in \CT$. Then we claim that $\bGamma_G(\1_{\CB G})\simeq \genref{G}{\CR}$. This follows from the following sequence of equivalences, where $X$ is a compact $G$-spectrum,
    \begin{align*}
        \map_{\Sp_G}(X,\bGamma_G(\1_{\CB G})) \simeq \map_{\CR_{\CB G}}(\decatcoh{G}{\CBD X}{\CR},\1_{\CB G}) &\simeq \map_{\CR_{\CB G}}(\CBD\decatcoh{G}{X}{\CR},\1_{\CB G}) \\ 
        &\simeq \Gamma(\CBD\CBD\decatcoh{G}{X}{\CR}) \simeq \Gamma(\decatcoh{G}{X}{\CR}).
    \end{align*}
    The second equivalence follows from the fact that $\decatcoh{G}{-}{\CR}$ is strong monoidal, and the fourth from the fact that $\decatcoh{G}{X}{\CR}$ is therefore a dualizable object of $\CR_{\CB G}$. 
\end{remark}

Furthermore, one can show that under these equivalences, the map \[R_f\colon f^*\bGamma_G(\1_{\CB G}) \rightarrow \bGamma_H(f^*\1_{\CB G}) \simeq \bGamma_H(\1_{\CB H})\] agrees with the lax structure map of the object $\genref{\gl}{\CR} \in \Spgl{\CE}\simeq \laxlimdag \Sp_\bullet$. In particular we conclude that $R_f$ is an equivalence on the unit when $f$ is faithful, but not generally. We would like to generalize this observation to arbitrary objects $X\in \CR_G$, but it is not clear to the authors if we should always expect this. However, under stronger assumptions, we can show that $R_f$ is a natural equivalence when $f$ is faithful. We introduce these stronger assumptions in the following:

\begin{definition}
    A naive global 2-ring $\CR$ is \emph{rigid} if:
    \begin{enumerate}
    \item each $\infty$-category $\CR_{\CB G}$ is rigidly-compactly generated, which is to say that $\CR_{\CB G}$ is compactly generated $\infty$-category in which an object is compact if and only if it is dualizable;
    \item for every faithful map $f\colon \CB H \rightarrow \CB G$, the functor $f_*\colon \CR_{\CB G}\rightarrow \CR_{\CB H}$ preserves compact objects.
    \end{enumerate}
\end{definition}

\begin{remark}
    If $\CR$ is rigid, then axiom (2) of \Cref{def:oriented_glo_cat} is automatically satisfied, see \cite{GrothendieckNeeman16}*{Theorem 1.3}.
\end{remark}

\begin{example}
    The global 2-ring $\Sp_\bullet$, and more generally $\Mod_{E_\bullet}$, is rigid.
\end{example}

\begin{example}
Let $\bG$ be an oriented $\rP$-divisible group. Then the naive global 2-ring $\LocSys_{\bG}(-)$ is rigid, see \Cref{megaprop}(6).
\end{example}

In the following lemma we introduce the crucial additional structure afforded to a rigid naive global 2-ring.

\begin{lemma}\label{lem:grothendieck_duality}
    Suppose $\CR$ is rigid, and let $f\colon \CB H \rightarrow \CB G$ be a faithful map. Then there exists an object $\omega_f\in \CR_{\CB H}$ such that $f^*$ is right adjoint to $f_\sharp \coloneqq f_*(-\otimes \omega_f)$. Moreover $\omega_f$ is uniquely characterized by the existence of a natural equivalence
    \[
    \underline{\Hom}_{\CR_{\CB G}}(f_*(-),\1_{\CB G}) \simeq f_*\underline{\Hom}_{\CR_{\CB H}}(-,\omega_f).
    \] 
\end{lemma}

\begin{proof}
    This is \cite{GrothendieckNeeman16}*{Theorem 1.7}.
\end{proof}

With this we can prove that $R_f$ is an equivalence whenever $f$ is faithful.

\begin{theorem}\label{thm-rigid}
    Suppose $\CR$ is a rigid $\CT$-genuine global 2-ring and let $f\colon \CB H\rightarrow \CB G$ be a faithful map in $\CT$. Then the natural transformation
    \[
    R_f\colon f^*\bGamma_G(-)\Rightarrow \bGamma_H(f^*(-))
    \]
    is an equivalence. 
\end{theorem}	

\begin{proof}
    Because $\CR$ is rigid all of the functors involved have left adjoints. Therefore we may instead show that the total mate 
    \[
    U_f\colon f_\sharp F_H(X) \Rightarrow   F_G f_!(X)
    \] of $R_f$ is an equivalence for all $H$-spectra. In fact because all of the functors preserve colimits, it suffices to prove this for $X$ compact. In this case we may compute:
    \begin{align*}
        f_\sharp F_H(X) \coloneqq f_\sharp \decatcoh{H}{\CBD X}{\CR} & \simeq f_\sharp \CBD\decatcoh{H}{ X}{\CR}\\
        & \simeq f_*(\underline{\Hom}_{\CR_{\CB H}}(\decatcoh{H}{X}{\CR}), \1_{\CB H})\otimes \omega_f) \\
        &\simeq f_*\underline{\Hom}_{\CR_{\CB H}}(\decatcoh{H}{X}{\CR}, \omega_f) \\
        &\simeq \underline{\Hom}_{\CR_{\CB G}}(f_*\decatcoh{H}{X}{\CR}, \1_{\CB G}) \\
        &\xrightarrow[\;\smash{\raisebox{0.5ex}{\ensuremath{\scriptstyle\sim}}}\;]{\hspace{.3em}\mathbb{D}(T_f)\hspace{.3em}}\CBD \decatcoh{G}{f_! X}{\CR} \\
        & \simeq \decatcoh{G}{\CBD f_! X}{\CR} \eqqcolon F_G(f_!X).
    \end{align*}
    The second and fourth equivalence use \Cref{lem:grothendieck_duality}, the third that $\decatcoh{H}{X}{\CR}$ is dualizable, and the first and fifth that $\decatcoh{H}{-}{\CR}$ and $\decatcoh{G}{-}{\CR}$ are strong monoidal functors. Finally $T_f$ is an equivalence by \Cref{rem:Talpha-equivalence}. A tedious diagram chase shows that this equivalence agrees with $U_f$.
\end{proof}	

\begin{remark}\label{rem:genuine_sect_lax}
    Recall that when $\CR$ is $\CT$-genuine, the functors $\decatcoh{G}{-}{\CR}$, and therefore also $F_G$, are strong monoidal for all $G \in \CT$. We conclude that the functors $\bGamma_G$ are therefore canonically lax monoidal, and that if $\CR$ is rigid then $\bGamma_\bullet$ is a natural transformation of lax monoidal functors.
\end{remark}

\begin{remark}
Rigidity is a sufficient but not necessary condition for the result of the previous theorem to hold. For example, the naive global 2-ring $\Spgl{\bullet}$ of \Cref{ex:global-spectra-cat-cohom} is \emph{not} rigid. Indeed, in the $\infty$-category $\Sp_{\gl}$ the object $\Sigma_{\gl}^\infty \CB G$ is compact but not dualizable. Nevertheless, one can show that $R_f$ is still an equivalence when $f$ is faithful. The reason for this is that $\Spgl{\bullet}$ does nevertheless admit Wirthm\"uller isomorphisms, and so in particular satisfies the first part of \Cref{lem:grothendieck_duality}. This is in fact the only input for the proof of the previous theorem which does not hold in an arbitrary $\CT$-genuine global 2-ring.
\end{remark}

\part{Examples}\label{examples}

We will now apply the theory developed in previous parts to two families of examples. The first are oriented spectral abelian group objects. Specializing to the universal example of an oriented elliptic curve we obtain a global spectrum enhancing the spectrum of topological modular forms. The second are $\rP$-divisible groups over a commutative ring spectrum $R$, as defined in \cite{Ell3}. We conclude that tempered cohomology is canonically represented by a global spectrum $\ul{R}^{\gl}$. Finally, using the notion of genuine global sections, we show that for an oriented $\rP$-divisible group, the $\infty$-category $\LocSys_{\bG}(\CB A)$ of tempered local systems on $\CB A$ is equivalent to modules over the restriction of $\ul{R}^{\gl}$ to a $A$-spectrum.

\section{Globally equivariant elliptic cohomology}\label{sec:oriented_ab_groups}
We start this section by giving a brief recollection on (pre)oriented strict abelian group objects in a general $\infty$-category. As we will see, these are very closely connected to diagrams out of the global orbit $\infty$-category. We then fix a nonconnective spectral Deligne-Mumford stack $\sfS$ and focus attention on oriented elliptic curves $\bE$ in spectral Deligne-Mumford stacks over $\sfS$, and explain how this data give rise to a genuine $\ab$-global $2$-ring $\CQ_{\bullet}^{\bE}$, see \Cref{ell-satisfies-cond}. Finally, we discuss how to obtain global refinements for the spectra of topological $K$-theory, elliptic cohomology and topological modular forms. 

Let us start by recalling some material on (pre)oriented strict abelian group objects. For more detail we refer the reader to \cite{GM20}*{Section 3}, see also \cite{Ell1}*{Section 1}. Let $\mathrm{Lat}$ denote the full subcategory of the 1-category of abelian groups spanned by the finitely generated free abelian groups and let $\mathcal{C}$ be a complete and cocomplete category.

\begin{definition}
We let $\Aball(\CC)$ denote the $\infty$-category of \textit{abelian group objects in $\CC$}, that is the $\infty$-category of product preserving functors from $\mathrm{Lat}^{\op}$ to $\CC$. 
\end{definition}

\begin{remark}
    Objects of $\Aball(\CC)$ are often called \textit{strict abelian group objects} in $\CC$ to distinguish them from group-like commutative monoids in $\CC$. For example $\mathrm{CGrp}(\Spc) \simeq \Sp^{\geq 0}$ while by \cite{Ell1}*{Remark 1.2.10} $\Aball(\Spc) \simeq \Mod_{\Z}^{\geq 0}$.
\end{remark}

\begin{example}
        Note that every compact abelian Lie group is an abelian group object of the topological category $\mathrm{AbCptLie}$ by \cite{GM20}*{Example 3.3}. Because the functor $\CB (-)\colon \mathrm{AbCptLie} \rightarrow \Glo{\ab}\rightarrow \Spcgl{\ab}$ preserves products, see \Cref{rem:Glo_fin_prods}, we conclude that every orbit stack $\CB A$ is canonically an abelian group object in $\Spcgl{\ab}$. Furthermore for every group homomorphism $\alpha\colon H\rightarrow G$ of abelian groups, $\CB\alpha \colon \CB H\rightarrow \CB G$ is canonically a morphism of abelian group objects.
\end{example}

\begin{example}
    Let $\sfS$ be a nonconnective spectral Deligne–Mumford stack, keeping in mind our convention (item (7) in \Cref{convention}). By definition any \emph{(strict) elliptic curve} over $\sfS$ is an abelian group object in the $\infty$-category of nonconnective spectral
Deligne–Mumford stacks over $\sfS$, see \cite{GM20}*{Definition 5.6}. 
\end{example}

\begin{example}
Every object $A \in \Aball(\Spc)$ canonically defines an object $A\coloneqq A\otimes {\pt}$ in $\Aball(\CC)$ by taking the constant $A$-shaped colimit of the point. For example, $B\T$ is canonically an abelian group object in spaces (corresponding to $\Sigma^2(\Z)$ in $\Mod_{\Z}^{\geq 0}$), and so $B\T$ is also an object of $\Aball(\CC)$ for any presentable $\CC$.
\end{example} 

\begin{definition}
A preorientation of an abelian group object $X$ in $\CC$ is a map $B\T\rightarrow X$ of abelian group objects in $\CC$. We define the $\infty$-category of \textit{preoriented abelian group objects} in $\CC$ as $\mathrm{PreAb}(\CC) \coloneqq \Aball(\CC)_{B\T/ }$.
\end{definition}

\begin{construction}[\cite{GM20}*{Construction 3.8}]\label{const:Pont-duality}
Consider the strict abelian group object $\CB \T$ in $\Spcgl{\ab}$. Mapping into $\CB \T$ defines a functor $\Spcgl{\ab}^{\op}\rightarrow \Aball(\Spc)$. Restricting this along the subcategory $\Glo{\ab}$ we obtain a functor $(\widehat{-})\colon \Glo{\ab}^{\op}\rightarrow \Aball(\Spc)$. The final object $\pt$ of $\Glo{\ab}$ is sent to $\Map_{\Glo{\ab}}(\pt,\CB \T) \simeq B\T$ by $(\widehat{-})$, and so the previous functor lifts to a functor
    \[
    (\widehat{-})\colon \Glo{\ab}^{\op}\rightarrow \mathrm{PreAb}(\Spc).
    \] 
    We call this construction \emph{shifted Pontryagin duality}. Note that $\widehat{\CB G} \simeq \widehat{G}\times B\T$, where $\widehat{G}$ is the ordinary Pontryagin dual of $G$, explaining the name.
\end{construction}

\begin{construction}\label{const:geom_ellipt_cohom}
    By \cite[Prop 3.10]{GM20}, there is an equivalence 
    \[
    \mathrm{PreAb}(\CC) \simeq \Fun^{\mathrm{R}}(\mathrm{PreAb}(\Spc)^{\op},\CC).
    \]
    Given $\bG\in \mathrm{PreAb}(\CC)$, we restrict the associated functor $\mathrm{PreAb}(\Spc)^{\op}\to \CC$ along shifted Pontryagin duality to obtain
    \[
    \bG[\widehat{-}]\colon \Glo{\ab}\rightarrow \CC, \quad \CB A\mapsto \bG[\hat{A}].
    \]
    Clearly this defines a functor 
    \begin{equation}\label{shifted-pontr-dual}
    \mathrm{PreAb}(\CC) \to \Fun(\Glo{\ab},\CC), \quad \bG\mapsto \bG[\widehat{-}]. 
    \end{equation}
\end{construction}

\begin{remark}\label{rem-G-bullet-pres-limits}
As observed in the proof of~\cite{GM20}*{Proposition 3.15}, the functor $\bG[\widehat{-}]\colon \Glo{\ab}\rightarrow \CC$ preserves finite products and pullbacks in which at least one of the maps is a quotient map. Moreover if $\CC$ is cartesian closed, it also follows that $\bG_{\bullet}\colon \Spcgl{\ab} \to \CC$ preserves finite products.  
\end{remark}

\begin{example}
    Unraveling the definitions one finds that $\bG_{\CB \T}\simeq \bG$. Moreover, using the previous remark, one sees that $\bG_{\CB C_n}$ sits in the following pullback square 
    \[
    \begin{tikzcd}
    \bG_{\CB C_n} \arrow[r]\arrow[d] & \pt \arrow[d] \\
    \bG \arrow[r, "n\cdot -"] &   \bG,
    \arrow["\lrcorner"{anchor=center, pos=0.125}, draw=none, from=1-1, to=2-2]
    \end{tikzcd}
    \]
    and so can be identified with the $n$-torsion points of $\bG$. As mentioned, The functor $\bG[\widehat{-}]$ preserves finite products, and so we conclude that the value of $\bG[\widehat{-}]$ on all abelian compact Lie groups is precisely the $\hat{A}$-valued points of $\bG$ as the notation suggests. 
\end{example}


\newcommand{\Aff}{\mathrm{Aff}}
\newcommand{\LRS}{\mathrm{Top}^\mathrm{loc}_\mathrm{CAlg}}
\newcommand{\Shv}{\mathrm{Shv}}

We will now explain how to apply this construction to define the global 2-ring enhancing elliptic cohomology. We begin with the affine case, where it will be induced by a ``geometric" incarnation of elliptic cohomology
\[ 
\mathbf{E}_\bullet\colon \Spcgl{\ab}^{\omega}\rightarrow\mathrm{Shv}(\Aff_\bS)
\]
by passing to $\infty$-categories of quasicoherent sheaves. We first make precise what we mean by quasi-coherent sheaves.

\begin{definition}
Consider the functor $\Mod\colon  \Aff_\bS \rightarrow \PrL$ which assigns to every affine scheme its category of modules. This satisfies descent with respect to Zariski covers, and so we obtain a unique limit preserving extension
\[
\QCoh\colon \mathrm{Shv}(\Aff_{\bS})^{\op}\rightarrow \PrL.
\]
\end{definition}

\begin{construction}[\cite{GM20}*{Construction 6.1}]

Let $\LRS$ denote the $\infty$-category of locally spectrally ringed spaces, as defined in \cite{SAG}.
Observe that fully faithful inclusion $\Aff_{\bS}\to\LRS$ induces a restricted Yoneda embedding functor
\[
\LRS\to\Shv(\Aff_{\bS}),
\]
where we endow $\Aff_{\bS}$ with the Zariski topology.
This is because pushouts of affine spectral schemes along Zariski open immersions can be computed in the $\infty$-category of spectrally locally ringed spaces. Then given an oriented elliptic curve $\bE\to\Spec R$ over an affine base, we may left Kan extend the functor $\bE[\widehat{-}]\colon\Glo{\ab}\to\LRS$ along the fully faithful inclusion $\Glo{\ab}\to\Spcgl{\ab}^\omega$ in order to obtain a functor
\[
\bE_\bullet\colon\Spcgl{\ab}^\omega\to\LRS.
\]
We then define $\CQ^\bE$ to be right Kan extension to all global spaces of the composite
\[
(\Spcgl{\ab}^\omega)^{\op}\to(\mathrm{Top}_{\CAlg}^{\mathrm{loc}})^{\op}\to\Shv(\Aff_\bS)^{\op} \xrightarrow{\QCoh(-)} \PrL.
\]
By definition this functor preserves cofiltered limits, but it will typically not preserve finite limits, as finite colimits are not respected by the functor of points $\LRS\to\Shv(\Aff_\bS)$. 
\end{construction}

\begin{remark}
Observe that by \Cref{const:Pont-duality}, we may identify $\bE_{\CB A}$ with the $\widehat{A}$-torsion points of $\bG$. By definition the value of $\bE_\bullet$ on any finite $\ab$-global space is the Zariski sheaf represented by a locally spectrally ringed space glued together from copies of $\bE[\widehat{A}]$ for various compact abelian Lie groups $A$.
We note that if $A$ is a finite abelian global quotient, then this locally ringed space is in fact a spectral scheme, but we will not make use of this fact. Nevertheless it provides a justification for taking finite colimits of the $\bE[\widehat{A}]$ as spectrally locally ringed spaces instead of Zariski sheaves, since the latter would not be represented by schemes in general.
\end{remark}

As we will prove below, the previous construction provides the global 2-ring associated to an elliptic curve over an affine base. We now explain how to globalize this construction. 

\begin{construction}
    The functor $\CQ^\bE\colon\Spcgl{\ab}^{\op}\to \PrL$ constructed above took as input an oriented elliptic curve $\bE\to\Spec R$ over an affine scheme. Observe that this is evidently contravariantly functorial in affine spectral schemes equipped with an oriented elliptic curve, or equivalently, with a map to $\CM_{\mathrm{ell}}^{\mathrm{or}}$, the moduli stack of oriented elliptic curves. Thus the previous construction refines to a functor
    \[
\Aff_{/\CM_{\mathrm{ell}}^{\mathrm{or}}}^{\op}\to\Fun(\Spcgl{\ab}^{\op},\PrL).
    \]
    The global 2-ring associated to the universal oriented elliptic curve $\bE^{\mathrm{univ}} \to \CM_{\mathrm{ell}}^{\mathrm{or}}$ over the moduli stack itself is obtained by taking the limit of this functor.

    More generally, an arbitrary nonconnective spectral Deligne--Mumford equipped $\sfS$ with an oriented spectral elliptic curve $\bE \to \sfS$ determines a functor $\Aff_{/\sfS}\to\Aff_{/\CM_{\mathrm{ell}}^{\mathrm{or}}}$ and hence by precomposition a functor
    \[\Aff_{/\sfS}^{\op}\to\Fun(\Spcgl{\ab}^{\op},\PrL).
    \]
    Taking the limit of this functor, we obtain the elliptic cohomology global $2$-ring $\CQ^\bE\colon \Spcgl{\ab}^{\op}\to \PrL$
    associated to an oriented elliptic curve $\bE\to \sfS$. Observe that, if $\sfS\simeq\Spec R$ is affine, then the $\infty$-category $\Aff_{/\sfS}$ admits a final object, recovering the original definition in the affine case.
\end{construction}

\begin{lemma}
Let $\mathbf{E}$ be an oriented elliptic curve in spectral Deligne--Mumford stacks over $\sfS$. Then $\CQ^{\bE}$ is a naive global 2-ring.
\end{lemma}

\begin{proof}
Since $\CQ^\bE$ is extended via limits from the affine case, we may immediately reduce to the case that $\sfS$ is affine and locally 2-periodic. Consider the unraveling of $\CQ^{\bE}$ at a compact abelian Lie group $A$:
\[
\decatcoh{G}{-}{\CQ^{\bE}}\colon \Spc_A^{\op} \to \QCoh(\bE_{\CB A}), \quad (X\sslash A \xrightarrow{f} \CB A) \mapsto f_* \1_{\QCoh(\bE_{X\sslash A})}.
\]
If we restrict to compact $A$-spaces, the resulting functor is limit-preserving by the argument of \cite[Lemma 6.5]{GM20}. For a general $A$-space $X$, we may write is as a filtered colimit of compact $A$-spaces $X_i$. Then $\QCoh(\bE_X)=\lim_i \QCoh(\bE_{X_i})$, and so $f_*\1_{X}$ is isomorphic to  $\lim_i (f_i)_*\1_{X_i\sslash A}$ by \cite[Theorem B]{HY}. It follows that $\decatcoh{G}{-}{\CQ^{\bE}}$ is Kan extended from the orbits, and so preserves all limits. 

We finally want to verify that the decategorification is limit-preserving. Since we have right Kan extended from compact global spaces, it suffices to check that the restricted functor 
\[
(\Spcgl{\ab}^\omega)^{\op} \to \PrL \xrightarrow{\mathrm{End}(\1)} \CAlg
\]
is limit preserving. Since taking endomorphisms of the unit in quasicoherent sheaves is naturally equivalent to taking global sections, this composite is equivalent to the functor 
\[
(\Spcgl{\ab}^{\omega})^{\op} \to \mathrm{LRS}^{\op} \xrightarrow{\Gamma} \Sp.
\]
The first functor is limit preserving by definition while the second is a right adjoint, and so we conclude that the composite preserves limits.
\end{proof}

We will now proceed to show that $\CQ^{\bE}$ is a genuine global 2-ring. We begin with a lemma.

\begin{lemma}\label{lem:G_preserve_conn_pullback}
Let $\mathbf{E}$ be an oriented elliptic curve in spectral Deligne--Mumford stacks over $\sfS$. Consider a pullback square 
\[
\begin{tikzcd}
{\sX} & \CB K\ \\
{\CB H} & {\CB G}
\arrow["f", from=2-1, to=2-2]
\arrow[hook, from=1-2, to=2-2]
\arrow[hook, from=1-1, to=2-1]
\arrow["\lrcorner"{anchor=center, pos=0.125}, draw=none, from=1-1, to=2-2]
\arrow[from=1-1, to=1-2]
\end{tikzcd}
\] in $\Spcgl{\ab}$ such that $G$ is a torus. Then 
\[
\begin{tikzcd}
    {\bE_\sX} & \bE_{\CB K}\ \\
    {\bE_{\CB H}} & {\bE_{\CB G}}
    \arrow["f", from=2-1, to=2-2]
    \arrow[hook, from=1-2, to=2-2]
    \arrow[hook, from=1-1, to=2-1]
    \arrow["\lrcorner"{anchor=center, pos=0.125}, draw=none, from=1-1, to=2-2]
    \arrow[from=1-1, to=1-2]
\end{tikzcd}
\]
is a pullback square in $\mathrm{Shv}(\Aff_{\bS})$.
\end{lemma}

The previous lemma is adapted from forthcoming work of the first author and Lennart Meier, to appear in~\cite{GM2}.
The argument below is a sketch of the key points in the proof.
\begin{proof}
We immediately reduce to the case where $\sfS$ is an affine scheme. We observe that up to homotopy the diagram $\CB H \to \CB G \leftarrow \CB K$ lifts to a diagram of pointed objects, and so by \Cref{rem-pointed_glo} to a diagram of abelian compact Lie groups.
It follows that the diagram lifts to a diagram in $\mathrm{Ab}(\Spcgl{\ab})$. Since the forgetful functor $\mathrm{Ab}(\Spcgl{\ab})\to \Spcgl{\ab}$ preserves limits, it suffices to check that the diagram is a pullback in abelian group objects. Since we are now working in an additive category and $\bE_\bullet$ preserves finite products, it suffices to show that $\bE_\bullet$ preserves fibers of maps of the form $f\colon\CB H\to\CB G$, where $G$ is a torus and $H$ is an abelian compact Lie group (but not necessary a subgroup).

First consider the case in which the map $f\colon\CB H\to\CB G$ is faithful. Then the fiber of $f$ is the loop space $\pt\times_{\CB(G/H)}\pt$, which is $G/H$. Because $G$ is a torus, $G/H$ is again a torus. Computing the colimit in locally ringed spaces defining $\bE_{\pt\times_{\CB(G/H)}\pt}$, one easily sees that it is affine over $\sfS$. It follows that the map
\[
\bE_{\pt\times_{\CB(G/H)}\pt}\to\bE_{\pt}\times_{\bE_{\CB(G/H)}}\bE_{\pt}
\]
is an equivalence by \cite{GM20}*{Theorem 8.1}.

Now consider a general map $\CB H\to\CB G$ such that $G$ is a torus. We may factor this as a composite
\[
\CB H\twoheadrightarrow \CB G'\hookrightarrow \CB G
\]
such that $H\to G'$ is surjective and $G'\to G$ injective.
Let $H'\subset H$ denote the kernel of the map $H\to G$.

The fiber $F$ of the map $\CB H\to\CB G$ is then a $\CB H'$-bundle over $G/G'$ in $\Aball(\Spcgl{\ab})$.
Using that $G/G'$ is also a torus and $BH'$ is connected, one computes that in fact $F\simeq G/G' \times \CB H'$.
Since $\bE_\bullet$ preserves finite products and pullbacks along quotient maps, see \Cref{rem-G-bullet-pres-limits}, we deduce that the map $\bE_Q\to\bE_{\CB H}\times_{\bE_{\CB G}}\bE_{\pt}$ is an equivalence.
\end{proof}

\begin{theorem}\label{ell-satisfies-cond}
Suppose that $\bE\rightarrow \sfS$ is an oriented elliptic curve in spectral Deligne-Mumford stacks over $\sfS$. Then the naive $\ab$-global $2$-ring
\[
\CQ^{\bE} \colon \Spc_{\ab}^{\op}\rightarrow \PrL
\]
is genuine with respect to the subcategory of tori.
\end{theorem}

    
\begin{proof}
We will show that $\CQ^{\bE}$ is genuine with respect to the family of tori in $\mathrm{Glo}_{\ab}$. By \Cref{lem:G_preserve_conn_pullback} given a pullback square as in condition (1) of \Cref{def:oriented_glo_cat}, the resulting square 
    \[
    \begin{tikzcd}
        {\bE_\sX} & \bE_{\CB K}\ \\
        {\bE_{\CB H}} & {\bE_{\CB G}}
        \arrow["\alpha", from=2-1, to=2-2]
        \arrow[hook, from=1-2, to=2-2]
        \arrow[hook, from=1-1, to=2-1]
        \arrow["\lrcorner"{anchor=center, pos=0.125}, draw=none, from=1-1, to=2-2]
        \arrow[from=1-1, to=1-2]
    \end{tikzcd}
    \]
    is a pullback square in $\mathrm{Shv}(\Aff_{\bS})$. Furthermore the map $\bE_{\CB K}\hookrightarrow \bE_{\CB G}$ is equivalent to the homomorphism of abelian group schemes $\bE[\hat{K}]\to\bE[\hat{G}]$ obtained by taking torsion points indexed by the Pontryagin dual groups, and so affine and proper by \cite{GM20}*{Proposition 3.15} and the proof of Proposition 6.3 of \emph{op.~cit.} Therefore the resulting square given by applying $\QCoh(-)$ is right adjointable by the push-pull formula for quasi-coherent sheaves, see \cite{DAG_QCoh}*{Corollary 3.2.6}. Condition (2) also follows immediately from the previous observations and \cite{DAG_QCoh}*{Proposition 3.2.11}. Finally condition (3) follows from \cite{GM20}*{Lemma 9.1}, which computes that $\decatcoh{\T}{S^\tau}{\CQ^{\bE}}$ is equivalent to the canonical line bundle $\CO_{\bE}(-e_1)$ of the elliptic curve $\bE$, by an application of \Cref{prop-abelian-compact}.
\end{proof}

\begin{remark}\label{rmk:failure-of-basechange}
It is \emph{not} true that the diagram $\CQ^{\bE}$ satisfies right base-change for arbitrary pullback squares 
\[
\begin{tikzcd}
\sX \arrow[r]\arrow[d, hook] & \CB H \arrow[d, hook,"\alpha"] \\ \CB K \arrow[r] & \CB G
\arrow["\lrcorner"{anchor=center, pos=0.125}, draw=none, from=1-1, to=2-2]
\end{tikzcd}
\] of global spaces. For example, as explained in \cite{GM2}*{Remark 2.10}, if $\alpha$ is injective and $\CB G$ is not connected then right base-change is not satisfied. Therefore it is crucial in this example that we have the flexibility to restrict to a family of enough injective objects, in this case the family of tori.
\end{remark}

Having shown that $\CQ^\bE$ is genuine, we obtain a family of equivariant cohomology theories which satisfy the axioms of Ginzburg--Kapranov--Vasserot.

\begin{theorem}
Let $\bE$ be an oriented elliptic curve in spectral Deligne-Mumford stacks over $\sfS$. Then there exists an equivariant elliptic cohomology theory
\[
\decatcoh{G}{-}{\bE}\colon \Spc_G^{\op}\rightarrow \CAlg(\QCoh(\bE[\hat{G}]))
\]
for every abelian compact Lie group $G$, as well as coherently functorial change-of-group transformations
\[
Q_\alpha\colon \decatcoh{H}{\alpha^*(-)}{\bE} \Rightarrow \alpha^*\decatcoh{G}{-}{\bE} 
\]
for any group homomorphism $\alpha\colon H\rightarrow G$. Furthermore this data satisfies the Ginzburg--Kapranov--Vasserot axioms:
\begin{enumerate}[itemsep = 5pt]
    \item \emph{Induction:} Let $\alpha\colon G\rightarrow G/N$ be a surjective group homomorphism with kernel $N$ and let $X$ be a $G$-space such that the action of $N$ on $X$ is free. Write $p\colon X\rightarrow X/N$ for the canonical projection map, which is $G$ equivariant. Then the composite 
    \[\decatcoh{G/N}{\alpha_!(X/N)}{\bE}\xrightarrow{T_\alpha} \alpha_*\decatcoh{G}{X/N}{\bE} \xrightarrow{\alpha_*\decatcoh{G}{p}{\,\bG}} \alpha_*\decatcoh{G}{X}{\bE}\] is an equivalence;
    \item \emph{Basechange:} Let $\CB \alpha \colon \CB H \to \CB G$ be a map in $\Glo{\CE}$ such that $G$ is a torus. Then the natural transformation 
    \[
    Q_\alpha\colon \alpha^* \decatcoh{G}{X}{\bE}\Rightarrow \decatcoh{H}{\alpha^*X}{\bE}
    \] is an equivalence for all compact $G$-spaces;
    \item \emph{K\"unneth:} Let $G$ and $H$ be two tori, $X$ a compact $G$-space and $Y$ a compact $H$-space. Then there is an equivalence 
    \[
    \pi_G^*\decatcoh{G}{X}{\bE}\otimes \pi_H^*\decatcoh{H}{Y}{\bE} \simeq \decatcoh{G\times H}{X \times Y}{\bE},
    \]
    where $\pi_H$ and $\pi_G$ denote the two projections $G\times H\rightarrow H,G$. Moreover the functor $\decatcoh{G}{-}{\bE}\colon (\Spc_G^{\omega})^{\op}\rightarrow \QCoh(\bE[\hat{G}])$ is strong monoidal.
\end{enumerate}
\end{theorem}

\begin{proof}
Applying \Cref{thm:Unraveling} to the naive global 2-ring $\CQ^\bE$ we obtain the data as in the theorem. Because $\CQ^\bE$ is genuine with respect to the family of tori, the Ginzburg--Kapranov--Vasserot axioms follow from \Cref{thm:GKV-axioms}.
\end{proof}

\begin{remark}
Let us more explicitly compare these axioms to those of Ginzburg--Kapranov--Vasserot. First we note that, as observed in \Cref{rmk:failure-of-basechange}, $\CQ^{\bE}$ is \emph{not} genuine with respect to the family of all groups. In particular, the base-change axiom only hold for group homomorphisms into a torus, and the K\"unneth axiom only holds for tori (see \cite{GM2} for details and counterexamples for abelian groups which are not tori).

Secondly let us address the periodicity axiom, which we have so far not discussed. We compute that 
\[
\decatcoh{\T}{X\wedge S^\tau}{\bE} \simeq \decatcoh{\T}{X}{\bE}\otimes \decatcoh{\T}{S^\tau}{\bE} \simeq \decatcoh{\T}{X}{\bE}\otimes \CO_{\bE}(-e_1).
\] We interpret this as stating that the periodicity axiom holds in $\mathrm{RO}(\T)$-grading. Obtaining the periodicity axiom precisely as stated in \cite{GKV95} would require the object $\CO_{\bE}(-e_1)$ to be equivalent to $\Sigma^{2} \CO_{\bE}$, which is typically not the case.
\end{remark}

From our results, we also obtain a global spectrum representing globally equivariant elliptic cohomology.

\begin{definition}
Let $\bE\rightarrow \sfS$ be an oriented elliptic curve in spectral Deligne-Mumford stacks over $\sfS$. Applying \Cref{thm:Represented-theorem}, we conclude that the naive $\ab$-global ring $\decatcohsp{\gl}{-}{\bE}\colon \Spcgl{\ab}^{\op}\rightarrow \Sp$ is represented by an $\ab$-global spectrum, which we denote by $\Gamma_{\bE}(\sfS;\CO_{\sfS})$.
\end{definition}

\begin{remark}
The underlying spectrum of the global spectrum $\Gamma_{\bE}(\sfS;\CO_{\sfS})$ is simply the global sections $\Gamma(\sfS)=\Gamma(\sfS;\CO_\sfS)$ of the spectral Deligne--Mumford stack $\sfS=(\sfS,\CO_\sfS)$. However, the globally equivariant structure depends on the oriented abelian group object $\bE$ over $\sfS$.
\end{remark}

\begin{example}
Most of the arguments of this section can be adapted and carried out for an arbitrary oriented abelian group object over a spectral Deligne--Mumford stack. In particular we may apply the construction for the multiplicative group scheme $\mathbb{G}_{m}$ over $\Spec(\mathrm{KU})$, oriented as in \cite{Survey}*{Section 3.1}, see also \cite{GM20}*{Section 4}. By our previous construction we obtain a global spectrum $\Gamma_{\mathbb{G}_m}(\mathrm{KU};\CO_{\mathrm{KU}})$, which is a global form of equivariant complex K-theory. 
\end{example}

\begin{remark}
We note that a global spectrum $\mathrm{KU}_{\gl}$ enhancing globally equivariant complex K-theory has previously been constructed by \cite{Schwede18}. In fact, using \cite{GM20}*{Section 4} one can show that the global cohomology theory associated to $\Gamma_{\mathbb{G}_m}(\mathrm{KU};\CO_{\mathrm{KU}})$ agrees with that associated to Schwede's global complex $K$-theory spectrum. However less clear is how to construct an equivalence of these objects as global spectra. In other words, a coherent identification of the representation sphere deloopings provided by both constructions does not seem to follow formally. While we certainly expect this to be possible, we do not attempt a rigorous comparison here.
\end{remark}

\begin{example}
    Suppose $\bE$ is the universal elliptic curve lying over the moduli stack of oriented elliptic curves $(\CM_{\mathrm{ell}}^{\mathrm{or}},\CO_{\CM_{\mathrm{ell}}^{\mathrm{or}}})$, see \cite{Ell2}*{Proposition 7.2.10} and \cite{Meier2022}*{Section 4.1}. Then $\mathrm{TMF}_{\gl}\coloneqq \Gamma_{\bE}(\CM_{\mathrm{ell}}^{\mathrm{or}},\CO_{\CM_{\mathrm{ell}}^{\mathrm{or}}})\in \CAlg(\Spgl{\ab})$ is a global refinement of the spectrum $\mathrm{TMF}$ of topological modular forms.   
\end{example}

\begin{example}
    Let $\bE$ be an oriented elliptic curve over $\sfS$. Restricting the $\ab$-global spectrum $\Gamma_{\bE}(\sfS;\CO_{\sfS})$ to any compact abelian Lie group we obtain a genuine $A$-spectrum representing $A$-equivariant elliptic cohomology. By construction, the restriction of $\Gamma_{\bE}(\sfS;\CO_{\sfS})$ to a $\T$-spectrum agrees with the $\T$-spectrum $R$ of \cite{GM20}*{Construction 9.3}.    
    Moreover, given an \emph{arbitrary} compact Lie group $G$, our construction also gives a definition of a genuine $G$-spectrum representing $G$-equivariant elliptic cohomology. Namely, we may first right induce $\Gamma_{\bE}(\sfS;\CO_{\sfS})$, in the sense of \cite{Schwede18}*{Theorem 4.5.1}, to a fully global spectrum and then restrict this to a $G$-spectrum. This definition is motivated by the case of equivariant K-theory, which is also right induced from abelian groups, and implements the suggestion of \cite{Survey}. 
\end{example}

\section{Tempered cohomology}\label{sec:tempered}

We start this section by recalling the theory of $\bG$-tempered cohomology and $\bG$-tempered local systems for an oriented $\rP$-divisible group $\bG$ over a commutative ring spectrum $R$ following~\cite{Ell3}. We then apply our main results to canonically refine tempered cohomology to a global spectrum, see \Cref{thm-tempered-is-genuine}. In the final part of this section we identify Lurie's $\infty$-category of tempered local systems in terms of equivariant stable homotopy theory, see \Cref{thm-localsym-as-spectra}

\subsection{Tempered cohomology and tempered local systems}
Recall that by \cite{Ell3}*{Remark 3.5.2} the data of a $\rP$-divisible group $\bG$ over a commutative ring spectrum $R$ is equivalent to the data of a functor 
\[
\bG[-]\colon \Aball_{\mathrm{fin}} \rightarrow \CAlg_R
\] from the subcategory of finite abelian groups to commutative $R$-algebras such that
\begin{enumerate}
    \item $\bG[-]$ preserves finite coproducts. In particular $\bG[e] \simeq R$.
    \item For every short exact sequence $A\rightarrow B\rightarrow C$ of finite abelian groups, the square 
    \[\begin{tikzcd}
        {\bG[A]} & {\bG[B]} \\
        \bG[e] & {\bG[C]}
        \arrow[from=2-1, to=2-2]
        \arrow[from=1-1, to=1-2]
        \arrow[from=1-2, to=2-2]
        \arrow[from=1-1, to=2-1]
    \end{tikzcd}\] is a pushout square.
    \item Applied to an injective group homomorphism, $\bG[-]$ is finite flat of positive degree.
\end{enumerate}
\begin{notation}
In this section we exclusively consider the global family of \emph{finite abelian groups}. To simplify notation we make the convention that in this section $\Glo{}$ denotes the global orbit category with isotropy in the family of finite abelian groups. We therefore write $\Spc_{\gl}$ for the associated $\infty$-category of global spaces and $\Orb{}$ for the wide subcategory of $\Glo{}$ spanned by the faithful morphisms.
\end{notation}

By \cite{Ell3}*{Theorem 3.5.5} a \emph{preorientation} for $\bG$ can equivalently be defined as a lift of $\bG[-]$ in the diagram
    \begin{equation}\label{eq-lift}
    \begin{tikzcd}
        {\Aball_{\fin}} & {\CAlg_R} \\
        {\Glo{}^{\op}}
        \arrow["{\CB {(\widehat{-})}}"', from=1-1, to=2-1]
        \arrow["{R_\bG^\bullet}"', dashed, from=2-1, to=1-2]
        \arrow["{\bG[-]}", from=1-1, to=1-2]
    \end{tikzcd}
    \end{equation}
where $(\widehat{-})$ denotes the Pontryagin duality functor. A preoriented $\rP$-divisible group $\bG$ is \emph{oriented} if, roughly, at every prime $p$ the canonical map from $R_{(p)}^{B(\Q_p/\Z_p)}$ to the formal part of $\bG_{(p)}$ is an equivalence of formal groups. See \cite{Ell3}*{Definition 2.6.12} for a precise definition.

\begin{definition}
Let $\bG$ be a preoriented $\rP$-divisible group over a commutative ring spectrum $R$. Limit extending the composite 
\[
\Glo{}^{\op}\xrightarrow{R_\bG^\bullet} \CAlg_R \xrightarrow{\Gamma} \Sp
\] from (\ref{eq-lift}), we obtain a global cohomology theory $R^\bullet_\bG$ which, following \cite{Ell3}, we call $\bG$-\emph{tempered cohomology}.
\end{definition}

To investigate the properties of $\bG$-tempered cohomology, Lurie introduced the notion of $\bG$-tempered local systems. We now recall the definition and list some of the most important properties that this construction satisfies. 

\begin{definition}[{\cite{Ell3}*{Construction 5.1.3}}]\label{def-pretempered}
    Let $R$ be a commutative ring spectrum and let $\bG$ be a preoriented $\rP$-divisible group. For any global space $\sX$, we let ${\Glo{}}_{/\sX}$ denote the fiber product $\Glo{}\times_{\Spc_{\gl}}{\Spc_{\gl}}_{/\sX}$. More informally, ${\Glo{}}_{/\sX}$, is the $\infty$-category whose objects are pairs $(\CB A, \eta)$ where $\CB A \in \Glo{}$ and $\eta \colon \CB A^{(-)} \to \sX$ is a map of global spaces, see \Cref{not-righ-adj-eval}. 
    
    We also let $\underline{R}_\sX$ denote the composite
    \[
    ({\Glo{}}_{/\sX})^{\op} \xrightarrow{\fgt} \Glo{}^{\op} \xrightarrow{R_{\bG}^{\bullet}} \CAlg_R
    \]
    which we can view as a commutative algebra in $\Fun(({\Glo{}}_{/\sX})^{\op}, \Sp)$, equipped with the pointwise tensor product. A $\bG$-\emph{pretempered local system} $\CF$ on $\sX$ is an $\underline{R}_\sX$-module object of the functor $\infty$-category $\Fun(({\Glo{}}_{/\sX})^{\op}, \Sp)$ satisfying:
    \begin{itemize}
        \item[(A)] For any morphism $\alpha \colon \CB A \to \CB A_0$ in ${\Glo{}}_{/\sX}$ which is represented by a surjective group homomorphism, the map $R_{\bG}^{\CB A} \otimes_{R_{\bG}^{\CB A_0}} \CF(\CB A_0) \to \CF(\CB A)$ induced by $\CF(\alpha)$ is an equivalence of $R_{\bG}^{\CB A}$-modules.
    \end{itemize}
    We denote by $\LocSys_{\bG}^{\mathrm{pre}}(\sX)$ the full subcategory of $\underline{R}_\sX$-modules spanned by the $\bG$-pretempered systems on $\sX$. 
\end{definition}

\begin{definition}[{\cite{Ell3}*{Definition 5.2.4}}]
    In the situation of \Cref{def-pretempered} let $\CF$ be a $\bG$-pretempered system on a global space $\sX$. We say that $\CF$ is a $\bG$-\emph{tempered local system} if it satisfies the additional condition:
    \begin{itemize}
        \item[(B)] For any $\CB A \in (\Glo{}^{\op})_{/\sX}$ and faithful morphism $\alpha \colon \CB A_0 \to \CB A$, the canonical map 
        \[
        \CF(\CB A) \to \CF(\CB A_0)^{hA/A_0}
        \]
        exhibits the target as a $I(A_0/A)$-completion of $\CF (\CB A)$ where $I(A_0/A):=\ker(R_{\bG}^{\CB A} \to R_{\bG}^{\CB A_0})$ is the relative augmentation ideal.
    \end{itemize}
    We denote by $\LocSys_{\bG}(\sX)$ the full subcategory spanned by the $\bG$-tempered local systems on $\sX$.
\end{definition}

\begin{example}\label{ex-local-system}
Let $X$ be a space, which we identify with a constant global space as in \Cref{not-constant-global-space}. By~\cite{Ell3}*{Variant 5.1.15} there is an equivalence of categories 
\[
\LocSys_{\bG}^{\mathrm{pre}}(X) \simeq \Fun(X,\Mod_R).
\]
If furthermore $\bG$ is oriented, then by \cite{Ell3}*{Corollary 5.4.3} every pretempered local system on $X$ is tempered, and so there is also an equivalence
\[
\LocSys_{\bG}(X) \simeq \Fun(X,\Mod_R).
\]
\end{example}
The next example shows that the $\infty$-category of (pre)-local system over $\CB A$ is controlled by the faithful maps into $\CB A$.
\begin{example}\label{ex-pre-local-system-over-classifying-space}
    If $\bG$ is a preoriented $\rP$-divisible group over a commutative ring spectrum $R$, we let $\ul{R}_{\bG, \mathrm{fth}}$ denote the composite functor 
    \[
    ({\Orb{}}_{/\CB A})^{\op} \hookrightarrow ({\Glo{}}_{/\CB A})^{\op} \to \Glo{}^{\op} \xrightarrow{R_{\bG}}\CAlg 
    \]
    which we can regard as a commutative algebra objects in $\Fun(({\Orb{}}_{/\CB A}), \Sp)$. Then by \cite{Ell3}*{Proposition 5.1.12}, there is an equivalence 
    \[
    \LocSys_{\bG}^{\mathrm{pre}}(\CB A) \simeq\Mod_{\ul{R}_{\bG,\mathrm{fth}}}(\Fun(({\Orb{}}_{/\CB A})^{\op}, \Sp)).
    \]
    If $\bG$ is oriented, then one can also describe the $\infty$-category $\LocSys_{\bG}(\CB A)$ as a full subcategory of $\Mod_{\ul{R}_{\bG, \mathrm{fth}}}(\Fun(({\Orb{}}_{/\CB A})^{\op}, \Sp))$ satisfying the analogue of condition (B), we refer the reader to \cite{Ell3}*{Proposition 5.4.2} for more details.
\end{example}

\begin{notation}\label{not-evaluation-functor}
    For any $A$ there is an evaluation functor 
    \[
    \LocSys_{\bG}(\CB A)\subseteq \LocSys_{\bG}^{\mathrm{pre}}(\CB A) \to \Mod_{R_{\bG}^{\CB A}}, \qquad \CF \mapsto \CF(A),
    \]
    which we often denote by $\CF^A$. By \Cref{ex-pre-local-system-over-classifying-space}, the collection of functors $\{\CF\mapsto \CF^{A_0}, A_0\subseteq A\}$ is jointly conservative on pretempered local systems.
\end{notation}

Given a morphism of global spaces $f\colon \sX \to \sY$, consider the functor $({\Glo{}}_{/\sX})^{\op} \to ({\Glo{}}_{/\sY})^{\op}$ given by composition with $f$. This induces a functor on $\bG$-pretempered local systems
\[
f^* \colon \LocSys_{\bG}^{\mathrm{pre}}(\sY) \to \LocSys_{\bG}^{\mathrm{pre}}(\sX). 
\]
Furthermore the pullback functor clearly preserves $\bG$-tempered local systems  and so restricts to a functor $f^* \colon \LocSys_{\bG}(\sY) \to \LocSys_{\bG}(\sX)$. In fact by~\cite{Ell3}*{Remark 5.2.11} this construction determines a limit preserving functor 
\begin{equation}\label{localsystemfunctor}
\LocSys_{\bG} \colon \Spc_{\gl}^{\op} \to \widehat{\Cat}_\infty.
\end{equation}

In the next result we record some important facts about this functor that we will use throughout the section. 

\begin{proposition}\label{megaprop}
    Let $\bG$ be an oriented $\rP$-divisible group over a commutative ring spectrum $R$. 
    \begin{itemize}
        \item[(1)] For any $\rP$-global space $\sX$, the $\infty$-category $\LocSys_{\bG}(\sX)$ is stable and presentably symmetric monoidal with unit object given by $\underline{R}_\sX$. 
        \item[(2)] For any morphism of global spaces $f\colon \sX \to \sY$, the functor $f^*\colon \LocSys_{\bG}(\sY)\rightarrow \LocSys_{\bG}(\sX)$ preserves limits and colimits and is symmetric monoidal. In particular, $f^*$ admits a left adjoint $f_!$ and a right adjoint $f_*$.
        \item[(3)] Let $f \colon \CB A_0 \to \CB A$ be a faithful map in $\Glo{}$. There is an equivalence of functors $f_! \simeq f_*$.
        \item[(4)] For any finite abelian group $A$, the evaluation functor $\LocSys_{\bG}(\CB A)\to \Mod_{R_{\bG}^{\CB A}}$ preserves all limits and colimits. 
        \item[(5)] For any finite abelian group $A$, the $\infty$-category $\LocSys_{\bG}(\CB A)$ is rigidly-compactly generated. A set of compact generators is given by $\{f_! \1 \mid f \in {\Orb{}}_{/\CB A}\}$.
        \item[(6)] The functor $\LocSys_{\bG}$ is a rigid naive global 2-ring.  
    \end{itemize}
\end{proposition}

\begin{proof}
    For part (1) combine ~\cite{Ell3}*{Proposition 5.2.12, Corollary 5.8.6 and Remark 5.8.8}. Part (2) follows by combining~\cite{Ell3}*{Corollary 5.2.13, Corollary 5.3.2 and Proposition 5.8.13}. Part (3) will follow from the fact that $f\colon \CB H \to \CB G$ is $v_{\bG}$-ambidextrous, see discussion before~\cite{Ell3}*{Definition 7.2.4}. We first observe that $f$ is relative $\pi$-finite by \cite{Ell3}*{Proposition 7.2.7} together with \Cref{ex-classifying-space-right-induced}. It then follows from \cite{Ell3}*{Theorem 7.2.10} that the map $f$ is $v_{\bG}$-ambidextrous as needed. Part (4) follows by combining~\cite{Ell3}*{Corollary 5.2.13 and Corollary 5.3.2}. The previous point implies that the unit of $\LocSys_{\bG}(\CB A)$ is compact for all $\CB A$. Note that $f_!$ preserves compact objects as $f^*$ preserves all colimits, and so the object $f_! \1$ associated to a faithful map $f\colon \CB A_0\hookrightarrow \CB A$ is compact and corepresents evaluation at $A_0$. \Cref{not-evaluation-functor} therefore implies that $\{f_! \1 \mid f \in {\Orb{}}_{/\CB A}\}$ is a set of compact generators for $\LocSys_{\bG}(\CB A)$. These objects are even self-dual. This is a consequence of (3), as is explained in \cite{Ell3}*{Proposition 7.3.15}. It then follows that $\LocSys_{\bG}(\CB A)$ is rigidly-compactly generated as the unit is compact and we have a set of compact and dualizable generators, this concludes (5). Finally, part (6) easily follows from (1), (2) and (4). 
\end{proof}

\begin{remark}
Part (3) of the theorem above shows that the categories of tempered local systems satisfy an analogue of the Wirthm\"uller isomorphisms, also known as ambidexterity with respect to faithful maps of orbits. In \cite{Ell3}, Lurie in fact shows that the categories of tempered local systems satisfy ambidexterity for a much larger class of maps, namely all relatively $\pi$-finite maps of global spaces. For the purposes of our article we only require the special case of Wirthm\"uller isomorphisms.
\end{remark}

Decategorifying $\LocSys_\bG$ we obtain a naive global ring with associated global cohomology theory
\[
\decatcohsp{\gl}{-}{\LocSys_{\bG}}\colon \Spc_{\gl}^{\op}\rightarrow \CAlg.
\]
By~\cite{Ell3}*{Example 7.1.5} it agrees with $R^\bullet_{\bG}$. We conclude that $\bG$-tempered local systems function as a categorification of tempered cohomology. We would like to exhibit $\LocSys$ as a genuine global 2-ring. To accomplish this we have to take a digression and first discuss a notion of geometric fixed points for tempered local systems.

\subsection{Geometric fixed points}

To begin, we follow \cite{Ell3}*{Notation 5.6.8} by making the following:
\begin{definition}\label{df-nilp-local-compl}
    Let $R$ be a commutative ring spectrum, $\bG$ be a preoriented $\rP$-divisible group over $R$ and $A$ a finite abelian group. We define $\Spec(R_\bG^{\CB A})^{\deg}\subset \Spec(R_\bG^{\CB A})$ to be the union of the images of all maps $\Spec(R_\bG^{\CB A_0})\rightarrow \Spec(R_\bG^{\CB A})$ associated to all proper subgroups $A_0$ of $A$. Then $\Spec(R_\bG^{\CB A})^{\deg}\subset \Spec(R_\bG^{\CB A})$ is the vanishing locus of some finitely generated ideal $I\subset \pi_0(R_{\bG}^{\CB A})$. We say an $R_{\bG}^{\CB A}$-module $M$ is $\Spec(R_\bG^{\CB A})^{\deg}$-nilpotent/local/complete if and only if it is $I$-nilpotent/local/complete in the sense of \cite{SAG}*{Definitions 7.1.1.1, 7.2.4.1 and 7.3.1.1} respectively. We write $\Mod_{R^{\CB A}_{\bG}}^{\deg\mathrm{-loc}}$ for the subcategory of $\Mod_{R^{\CB A}_{\bG}}$ spanned by the $\Spec(R^{\CB A}_{\bG})^{\deg}$-local modules and write $\iota_{\mathrm{loc}}$ for the inclusion of $\Mod_{R^{\CB A}_{\bG}}^{\deg\mathrm{-loc}}$ into $\Mod_{R^{\CB A}_{\bG}}$. By \cite{SAG}*{Proposition 7.2.4.4} this inclusion admits a left adjoint $L\colon \Mod_{R^{\CB A}_{\bG}}\rightarrow \Mod_{R^{\CB A}_{\bG}}^{\deg\mathrm{-loc}}$.
\end{definition}

\begin{definition}
    Let $\bG$ be an oriented $\rP$-divisible group over a commutative ring spectrum $R$. Let $A$ be a finite abelian group and let $\CF\in\LocSys_{\bG}(\CB A)$. We define $\Phi^{A}(\CF):= L(\CF^A)$, that is as the localization of $\CF^A\in\Mod_{R_{\bG}^{\CB A}}$ at $\Spec(R_\bG^{\CB A})^{\deg}$. We refer to this as the $A$-\emph{geometric fixed points} of $\CF$. This gives a functor 
    \[
    \LocSys_{\bG}(\CB A)\rightarrow \Mod_{A_{\bG}^{\CB A}}^{\deg\mathrm{-loc}}\simeq  \Mod_{\Phi^A(\1)}.
    \]
    For any subgroup $\iota\colon A_0\subset A$, we define $\Phi^{A_0}(\CF) \coloneqq \Phi^{A_0}(\iota^* \CF)$.
\end{definition}

We now show that our definition of geometric fixed points functor enjoys similar properties as the geometric fixed points functor in equivariant stable homotopy theory. We start with the following:

\begin{proposition}
    For any finite abelian group $A$, the functor $\Phi^A \colon \LocSys_{\bG}(\CB A)\rightarrow \Mod_{R_{\bG}^{\CB A}}^{\deg\mathrm{-loc}}$ is a smashing localization, with right adjoint given by $U_A \circ \iota_{\mathrm{loc}}$ where $U_A$ sends a $R_{\bG}^{\CB A}$-module $M$ to the tempered local system determined by the assignment 
    \[
A_0\subset A \mapsto \begin{cases}
M & A_0 = A \\
0 & \text{otherwise.}
\end{cases}
\] Moreover, the functor $\Phi^{A_0}\colon \LocSys_{\bG}(\CB A)\rightarrow \Mod_{R_{\bG}^{\CB A_0}}^{\deg\mathrm{-loc}}$ is symmetric monoidal for all $A_0\subset A$.
\end{proposition}

\begin{proof}
Consider the functor
\[
L \circ \ev_A\colon \LocSys_{\bG}^{\mathrm{pre}}(\CB A)\to \Mod_{R_\bG^{\CB A}} \to \Mod_{R_\bG^{\CB A}}^{\deg\mathrm{-loc}}
\]
where the first functor is the evaluation at $A$, and the second functor is localization with respect to $\Spec(R^{\CB A}_\bG)^{\deg}$. Note that the evaluation functor $\ev_A$ admits a fully faithful right adjoint $U_A$, given by relative right Kan extension, which is defined as in the proposition.  We claim that the right adjoint $U_A\circ \iota_{\mathrm{loc}}$ of $L\circ \ev_A$ takes values in \emph{tempered} local systems. Note that tempered local systems form the right class in an orthogonal decomposition on $\LocSys^{\mathrm{pre}}(\CB A)$ by \cite{SAG}*{Proposition 7.2.1.4}, and therefore it suffices to show that $L\circ \ev_A$ kills those pretempered local systems which are left orthogonal to tempered local systems. However by \cite{Ell3}*{Proposition 5.7.7} these are precisely the \emph{null} local systems, defined as those $\CF\in \LocSys^{\mathrm{pre}}(\CB A)$ such that $\CF^{A_0}$ is $\Spec(R^{\CB A_0}_\bG)^{\deg}$-nilpotent for all faithful maps $A_0\hookrightarrow A$. Such local systems are clearly sent to zero by $L\circ \ev_A$. We conclude that $U_A \circ \iota_{\mathrm{loc}}$ restricts to a fully faithful right adjoint to $\Phi^A$. The essential image of $U_A \circ \iota_{\mathrm{loc}}$ is closed under colimits by \Cref{megaprop}(5), and so we conclude that $\Phi^{A}$ is even a smashing localization, see~\cite{axiomatic}*{Definition 3.3.2} for example. In particular it is canonically strong monoidal. As the pullback functors are symmetric monoidal by \Cref{megaprop}(2), we conclude that $\Phi^B$ is also symmetric monoidal for all $B\subset A$.
\end{proof}

We next recognize geometric fixed points as a localization procedure.

\begin{remark}\label{rem-completion}
 Let $A$ be a finite abelian group and let ${\mathrm{Orb}^\circ_{/A}}$ denote the full subcategory of ${\mathrm{Orb}_{/A}}$ obtained by removing the terminal object. By \cite{Ell3}*{Theorem 5.6.9}, if $\CF \in \LocSys_{\bG}(\CB A)$ then
    \[
    \CF^A\rightarrow \lim\limits_{A_0\in(
{\mathrm{Orb}^\circ_{/A}})^{\op}} \CF^{A_0}
    \]
    exhibits the right hand side as the completion of $\CF^A$ at $\Spec(R_\bG^{\CB A})^{\deg}$. Since local and complete objects form an orthogonal decomposition, we conclude that the functor \[
    \Mod_{R_{\bG}^{\CB A}} \to \Mod_{R_\bG^{\CB A}}^{\deg\mathrm{-loc}} \times \Mod_{R_\bG^{\CB A}}^{\deg\mathrm{-cpl}}, \quad \CF^A \mapsto \Phi^A \CF \times\lim\limits_{A_0\in(
{\mathrm{Orb}^\circ_{/A}})^{\op}} \CF^{A_0}
\]
is conservative. 
\end{remark}

\begin{proposition}\label{lem:geom_kills_induced}
Fix a finite abelian group $A$ and let $I_A$ denote the localizing ideal of $\LocSys(\CB A)$ spanned by the objects of the form $\iota_! \1$, where $\iota\colon A_0\hookrightarrow A$ is an inclusion of a proper subgroup. Consider $\CF\in \LocSys_{\bG}(\CB A)$. Then $\Phi^A \CF = 0$ if and only if $\CF\in I_A$.
\end{proposition}

\begin{proof}
   The statement in the proposition will follow from the following three claims:
    \begin{itemize}
        \item[(i)] $I_A={}^{\perp}(I_A^{\perp})$;
        \item[(ii)] $I_A^\perp$ is precisely the essential image of $U_A \iota_{\mathrm{loc}} \colon \Mod_{R_\bG^{\CB A}}^{\deg\mathrm{-loc}} \to \LocSys_{\bG}(\CB A)$;
        \item[(iii)] $\Phi^A \CF=0$ if and only if $\CF$ lies in the left orthogonal of the essential image of $U_A \iota_{\mathrm{loc}}$. 
    \end{itemize}
      
    Claim (i) is \cite{stacks-project}*{Lemma 0CQS} (we use that $I_A$ is localizing so the canonical inclusion functor admits a right adjoint). 
    For claim (ii), note that if $\CF= U_A \iota_{\mathrm{loc}}(M)$ then $\iota^* \CF= \iota^* U_A\iota_{\mathrm{loc}}(M)\simeq 0$ for any proper inclusion $\iota\colon A_0\hookrightarrow A$. This together with a simple adjunction argument shows that $\CF\in I_A^\perp$. For the converse suppose $\CF$ is right orthogonal to the induced objects. Given a faithful map $\iota\colon A_0\rightarrow A$,  mapping out of $\iota_! \1$ corepresents evaluation at $A_0$ by adjunction. Therefore $\CF^{A_0}$ is necessarily zero for all proper subgroups $A_0$. However by \Cref{rem-completion}, the tempered local systems with this property are clearly all of the form $U_A(M)$, where $M$ is a $\Spec(R_{\bG}^{\CB A})^{\deg}$-local object. This proves claim (ii). Finally claim (iii) easily follows from the equivalence 
    \[
    \Map(\Phi^A \CF, M)\simeq \Map(\CF, U_A \iota_{\mathrm{loc}}(M)).
    \]
\end{proof}
\begin{remark}\label{rem-geom-fix-points}
We conclude from the previous result that the quotient $\LocSys_{\bG}(\CB A)/I_A$ is equivalent to $\Mod^{\deg-\mathrm{loc}}_{R_{\bG}^{\CB A}}$ and that the geometric fixed points functor $\Phi^A$ is equivalent to the canonical functor from $\LocSys_{\bG}(\CB A)$ to the quotient.
\end{remark}

We next prove that the geometric fixed points functor form a jointly conservative family.

\begin{lemma}
    Let $A$ be a finite abelian group. For all subgroups $A_0\subseteq A$, the functors 
    \[
    \Phi^{A_0} \colon \LocSys_{\bG}(\CB A)\to \Mod_{R_{\bG}^{\CB A_0}}^{\deg\mathrm{-loc}}
    \]  
    form a jointly conservative family.
\end{lemma}

\begin{proof}
    We argue by induction on the order of the group $A$. If $A$ is the trivial group, then $\Phi^A \CF=\CF^A$ and $(-)^{A}$ is an equivalence by~\cite{Ell3}*{Example 5.1.13} and so conservative. Now suppose $A$ is an arbitrary abelian group, and let $\CF \in \LocSys_{\bG}(\CB A)$ be a tempered local system such that $\Phi^{A_0}(\CF)\simeq 0$ for all $A_0\subseteq A$. 
    By induction we know that for every proper subgroup inclusion $\iota\colon A_0\subset A$, we have $\iota^*\CF\simeq 0$ and so also $\CF^{A_0}\simeq 0$. As noted in \Cref{rem-completion}, there is a conservative functor 
    \[
    \CF^A \mapsto \Phi^A \CF \times\lim\limits_{A_0\in(
{\mathrm{Orb}^\circ_{/A}})^{\op}} \CF^{A_0}.
    \]
    The right hand side is trivial by assumption and by induction. Therefore $\CF^A\simeq 0$. By \Cref{not-evaluation-functor}, we conclude that $\CF \simeq 0$.
\end{proof}

Given the above, we have all the pieces required to show that geometric fixed points detects invertible dualizable objects. 

\begin{proposition}\label{prop:invert_temp_loc_sys}
    An object $\CF$ is an invertible object of $\LocSys_{\bG}(\CB A)$ if and only if 
    \begin{enumerate}
        \item $\CF$ is a dualizable object of $\LocSys_{\bG}(\CB A)$ and 
        \item $\Phi^{A_0}\CF$ is an invertible object for all $A_0\subset A$.
    \end{enumerate}
\end{proposition}

\begin{proof}
    This follows immediately from the fact that the functors $\Phi^{A_0}$ form a jointly conservative family of strong monoidal functors.
\end{proof}

We would like to apply this to the object $\decatcoh{C_n}{S^{\iota_n}}{\bG}$. For this we will need to understand the geometric fixed points of ``Borel tempered local systems". To this end we introduce the proper Tate construction and explain its connection to the geometric fixed points functor.
\begin{notation}
Consider $R\in \CAlg$  and a finite group $A$. We let $J_A\subset \Fun(BA,\Mod_R)$ denote the thick ideal generated by the objects $i_! R$, for every proper subgroup inclusion $i\colon BA_0 \rightarrow BA$. Following \cite{AMR2021}, we define the \emph{proper Tate construction} as the composite 
\[
(-)^{\tau A}\colon \Fun(BA, \Mod_R) \xrightarrow{p} \Fun(BA,\Mod_R)/J_A \xrightarrow{\Map(\1, -)} \Sp. 
\]
where $p$ is the quotient functor.
\end{notation}

Recall from \Cref{ex-nu-map} that there is a canonical map $\nu\colon B A \rightarrow \CB A$ of global spaces. The next result relates the proper Tate construction to the geometric fixed points functor. 

\begin{theorem}\label{tatecommuteswithgeomfix}
The following triangle commutes:
\[
\begin{tikzcd}
	{\Fun(BA,\Mod_R)} && {\LocSys_{\bG}(\CB A)} \\
	& \Sp
	\arrow["{\Phi^A}", from=1-3, to=2-2]
	\arrow["{(-)^{\tau A}}"', from=1-1, to=2-2]
	\arrow["{\nu_*}", from=1-1, to=1-3]
\end{tikzcd}
\]
\end{theorem}

The proof is an adaptation of the argument given in \cite{AMR2021}*{Proposition 5.9} for genuine $G$-spectra. For the convenience of the reader we record a version of the argument here after some preparation.

\begin{remark}\label{rem-nu}
    Since $\nu^*$ preserves colimits and $\LocSys_{\bG}(\CB A)$ is compactly generated, we get an extension $\overline{\nu}^*$ as in the following diagram:
    \[
    \begin{tikzcd}
        \LocSys_{\bG}(\CB A)^\omega \arrow[r, hook] \arrow[d, hook] & \LocSys_{\bG}(\CB A) \arrow[r,"\nu^*"] &\Fun(BA, \Mod_R)\arrow[d, hook, "i"] \\
        \LocSys_{\bG}(\CB A) \arrow[rr, "\overline{\nu}^*"] && \mathrm{Ind}(\Fun(BA, \Mod_R)).
    \end{tikzcd}
    \]
    In other words, for any $\CF \in \LocSys_{\bG}(\CB A)^\omega$ we have the formula 
    \begin{equation}\label{eq-nu-compact}
        \overline{\nu}^*(\CF) =i \nu^*(\CF).
    \end{equation}
    Moreover by construction the functor $\overline{\nu}^*$ fits into the following commutative diagram
    \[
    \begin{tikzcd}
        \LocSys_{\bG}(\CB A) \arrow[d,"\overline{\nu}^*"'] \arrow[r,"\nu^*"] & \Fun(BA, \Mod_R)  \\
         \mathrm{Ind}(\Fun(BA, \Mod_R)). \arrow[ur, "\colim"'] &
    \end{tikzcd}
    \]
    Passing to right adjoints we also deduce that 
    \[
    \begin{tikzcd}
        \Fun(BA, \Mod_R) \arrow[dr,"\nu_*"'] \arrow[r, hook, "i"] & \mathrm{Ind}(\Fun(BA, \Mod_R)) \arrow[d, "\overline{\nu}_*"] \\
        & \LocSys_{\bG}(\CB A).
    \end{tikzcd}
    \]
\end{remark}

Let $\mathrm{Ind}(J_A)$ denote the localizing ideal in $\mathrm{Ind}(\Fun(BA, \Mod_R))$ generated by $i(J_A)$.

\begin{lemma}\label{lem-ideal}
The localizing ideal generated by $\overline{\nu}^*(I_A)$ is equal to $\mathrm{Ind}(J_A)$ in $\mathrm{Ind}(\Fun(BA, \Mod_R))$.
\end{lemma}

\begin{proof}
    We note that $\mathrm{Ind}(J_A)$ can equivalently be described as the localizing ideal generated by the objects $i\iota_! \1$, where $\iota \colon A_0 \to A$ is the inclusion of a subgroup, see \cite{AMR2021}*{Observation 3.8}.  Since $I_A$ is generated by the objects $\iota_! \1\in \LocSys(\CB A)$, it suffices to
    observe the following chain of equivalences
    \[
    \overline{\nu_A}^*(\iota_! \1)\simeq i \nu_A^*\iota_!\1\simeq i \iota_! \nu_{A_0}^* \1 \simeq i \iota_! \1. 
    \]
    The first equivalence follows from~\Cref{eq-nu-compact} and compactness of $\iota_!\1$. For the second we observe that there is a pullback square of global spaces
    \[
    \begin{tikzcd}
        BA_0 \arrow[r, "\nu_{A_0}"] \arrow[d,"B\iota"']& \CB A_0 \arrow[d, "\CB \iota"] \\
        BA \arrow[r, "\nu_A"] & \CB A.
    \end{tikzcd}
    \] Therefore we obtain by~\cite{Ell3}*{Corollary 7.1.7} an equivalence $\nu^*_A \circ \iota_! \simeq \iota_!\circ  \nu_{A_0}^*$.
\end{proof}

\begin{remark}\label{rem-verdier}
    There is a commutative diagram of Verdier quotients
    \[
    \begin{tikzcd}
        \Fun(BA,\Mod_R) \arrow[d,"p"]\arrow[r,hook, "i"] & \mathrm{Ind}(\Fun(BA, \Mod_R)) \arrow[d,shift left, "q"] \arrow[r,"\overline{\nu}_*"]&\LocSys_{\bG}(\CB A) \arrow[d,shift left ,"r"] \\
        \Fun(BA, \Mod_R) / J_A \arrow[r,hook] & \mathrm{Ind}(\Fun(BA, \Mod_R))/\mathrm{Ind}(J_A)\arrow[u, hook, shift left, "j"]\arrow[r,"\overline{\overline{\nu}}_*"] & \LocSys_{\bG}(\CB A)/ I_A \arrow[u, hook, shift left,"k"]
    \end{tikzcd}
    \]
    where $p,q$ and $r$ are symmetric monoidal. The existence of right adjoints $j$ and $k$ is a formal consequence of the fact that $\mathrm{Ind}(J_A)$ and $I_A$ are localizing ideals. It follows from \cite{AMR2019}*{Corollary 4.2.15} together with \Cref{lem-ideal} that 
    \begin{equation}\label{equation-unit}
        \overline{\nu}^*(kr \1)=jq \1.
    \end{equation}
\end{remark}

\begin{proof}[Proof of \Cref{tatecommuteswithgeomfix}]
Consider the following diagram
 \[
    \begin{tikzcd}
        \Fun(BA,\Mod_R) \arrow[d,"p"]\arrow[r,hook, "i"] & \mathrm{Ind}(\Fun(BA, \Mod_R)) \arrow[d,shift left, "q"] \arrow[r,"\overline{\nu}_*"]& \LocSys_{\bG}(\CB A) \arrow[d,shift left ,"r"] \\
        \Fun(BA, \Mod_R) / J_A \arrow[r,hook,"\overline{i}"] \arrow[d, "{\Map(\1, -)}"'] & \mathrm{Ind}(\Fun(BA, \Mod_R))/\mathrm{Ind}(J_A)\arrow[u, hook, shift left, "j"]\arrow[r,"\overline{\overline{\nu}}_*"] & \LocSys_{\bG}(\CB A)/ I_A \arrow[d,"{\Map(\1, -)}"]\arrow[u, hook, shift left,"k"] \\
        \Sp \arrow[rr,"="]&  & \Sp
    \end{tikzcd}
    \]
    where the left vertical composite agrees with $(-)^{\tau A}$ by definition, the right vertical composite agree with the forgetful functor composed with $\Phi^A$ by \Cref{rem-geom-fix-points} and the top horizontal composite agree with $\nu^*$ by \Cref{rem-nu}. 
    In other words, we can rephrase the proposition as claiming that the outer diagram commutes. We first note that the commutativity of the top part of the diagram follows from \Cref{rem-verdier}. Moreover since $p$ and $r$ are symmetric monoidal, we can rewrite the global section functor as $\Map(p \1, -)$ and $\Map(r \1, -)$. Therefore we ought to prove that for all $X\in \Fun(BA, \Mod_R)$ there is a natural equivalence
    \[
    \Map(r\1, r\overline{\nu}_*iX)\simeq \Map(p\1,pX).
    \]
   Applying various adjunction equivalences, together with \Cref{equation-unit} yields
    \[
      \Map(r\1, r\overline{\nu}_*iX) \simeq\Map(\overline{\nu}^*kr \1, iX)\simeq \Map(jq\1, iX)\simeq \Map(q\1, qi X).
    \]
    Finally, using that $i$ is symmetric monoidal, the commutativity of the top right part of the diagram, and the fact that $\overline{i}$ is fully faithful we see that 
    \[
    \Map(q\1, qi X) \simeq \Map(qi\1, qiX)= \Map(\overline{i}p \1, \overline{i}p X) \simeq \Map(p\1, pX)
    \]
    as required.
\end{proof}

\subsection{Tempered local systems is genuine}

We may unravel $\LocSys_\bG$ into a family of equivariant cohomology theories
\[
\decatcoh{A}{-}{\LocSys_{\bG}} \colon \Spc_A^{\op} \to \CAlg(\LocSys_{\bG}(\CB A))
\]
using \Cref{thm:Unraveling}.

\begin{notation}
We will abbreviate the functor $\decatcohsp{\gl}{-}{\LocSys_{\bG}}$ by $\decatcohsp{\gl}{-}{\bG}$ and write $\decatcoh{A}{-}{\bG}$ instead of $\decatcoh{A}{-}{\LocSys_{\bG}}$
\end{notation}

As shown in \cite{Ell3}, tempered local systems function as an incredibly rich and powerful categorification of tempered cohomology. Some of this richness is captured by the following theorem.

\begin{theorem}\label{thm-tempered-is-genuine}
Let $\bG$ be an oriented $\rP$-divisible group over a commutative ring spectrum $R$. Then $\LocSys_\bG$ is genuine with respect to all finite abelian groups, and so a rigid global 2-ring. In particular tempered cohomology is represented by a global spectrum $\ul{R}_{\bG}^{\gl}$.
\end{theorem}

\begin{proof}
We have already noted in \Cref{megaprop} that $\LocSys_{\bG}(-)$ is a rigid naive global 2-ring. We only need to show that $\LocSys_{\bG}(-)$ satisfies the three conditions of \Cref{def:oriented_glo_cat} for $\mathcal{T}=\Glo{}$. The first two conditions are already contained in \cite{Ell3}:
\begin{itemize}
 \item[(1)] This follows from~\cite{Ell3}*{Theorem 7.1.6};
 \item[(2)] This is a consequence of \cite{Ell3}*{Theorem 7.3.10}, where we note that any faithful map $f\colon \CB H \to \CB G$ is $v_{\bG}$-ambidextrous, see the proof of \Cref{megaprop}(3).
\end{itemize} 
    Condition (3) of a global 2-ring will require more work. By~\Cref{prop:finite-abelian} it suffices to show that the object 
    \[
    \CF_n:=\decatcoh{C_n}{S^{\iota_n}}{\bG} \in \LocSys_{\bG}(\CB C_n)
    \]
    is invertible for all $1\leq n <\infty$. To this end we show that $\CF_n$ verifies conditions (1) and (2) of \Cref{prop:invert_temp_loc_sys}.
    
    For (1) we note that $S^{\iota_n}$ is a finite colimits of $C_n$-orbits. It follows that $\CF_n$ is a finite limit of objects of the form $i_* \1$ associated to subgroup inclusions  $i\colon \CB H\rightarrow \CB G$. Since each object $i_* \1$ is dualizable by \Cref{megaprop}(3) and (4), we conclude that $\CF_n$ is dualizable. 
    
    For (2) we have to show that  $\Phi^{C_m}\CF_n$ is an invertible object for all $1\leq m\leq n$. However note that for all subgroups $\alpha\colon \CB C_m\rightarrow\CB C_n$, we have $\alpha^* S^{\iota_n} \simeq S^{\iota_m}$.     
    Therefore by \Cref{lem:Q_equivalence} with $\mathcal{T}=\Glo{}$, it suffices to prove that $\Phi^{C_n}\CF_n$ is invertible for all $n \geq 1$.  Recall from \Cref{gamma-sphere} that there is a fiber sequence 
    \[
    \CF_n\rightarrow \1_{\LocSys_\bG(\CB C_n)} \to p_* \1_{\LocSys_{\bG}(S(\iota_n)\sslash C_n)}=: \CG_n
    \]
    where $p\colon S(\iota_n)\sslash C_n\rightarrow \CB C_n$ is the structure map of the $C_n$-space $S(\iota_n)$.
    
    If $n=1$, then $\LocSys_{\bG}(\CB C_1)=\Mod_R$. Since $S(\iota_1)\sslash C_1$ is the constant global space $S^1$, we also know that $\LocSys_{\bG}(S(\iota_1)\sslash C_1)\simeq \Fun(S^1, \Mod_R)$. In this case $p_*$ is given right Kan extension. We can then calculate that $\Phi^{C_1}\CG_1 = R^{S^1}\simeq R\oplus \Sigma^{-1}R$, and so $\Phi^{C_1}\CF_1 \simeq \mathrm{fib}(R \to R\oplus \Sigma^{-1} R)= \Sigma^{-2} R$, which is invertible.
    
    Now suppose $n>1$; we will prove that $\Phi^{C_n}(\CG_n)$ is trivial. To do this we have to understand the tempered local system $\CG_n$ better. Note that the action of $C_n$ on $S(\iota_n)$ is free, and so by \Cref{prop:free_G_spaces_as_global_spaces}, $S(\iota_n)\sslash C_n$ is equivalent as a global space to the constant global space on the space $S(\iota_n)_{hC_n}$.Furthermore, using \Cref{prop:free_G_spaces_as_global_spaces} again, $S(\iota_n)\sslash C_n$ lives over $\CB C_n$ via the composite 
    \[
    p\colon S(\iota_n)_{h C_n} \xrightarrow{f} B C_n \xrightarrow{\nu} \CB C_n.
    \]
    So we may compute $p_*$ as the composite $\nu_*f_*$. Now we note that the source and target of $f$ are constant global spaces so 
    \[
    f_* \colon \LocSys_{\bG}(S(\iota_n)_{h C_n}) = \Fun(S(\iota_n)_{h C_n},\Mod_R) \to \LocSys_{\bG}(BC_n)=\Fun(BC_n,\Mod_R)
    \]
    is given by right Kan extension. Applying this to the unit in $\Fun(S(\iota_n)_{h C_n}, \Mod_R)$, we obtain $\mathbb{D}(S(\iota_n))\otimes R\in \Fun(BC_n,\Mod_R)$, the Spanier--Whitehead dual of $S(\iota_n)$ with its induced $C_n$-action tensored by the base $R$. Note that $\mathbb{D}(S(\iota_n)) \otimes R\in J_{C_n}$ since $S(\iota_n)$ is a finite free $C_n$-space. Therefore we conclude that \[\Phi^{C_n} \nu_*(\mathbb{D}(S(\iota_n))\otimes R) \simeq (\mathbb{D}(S(\iota_n))\otimes R)^{\tau C_n}\simeq 0\]
    by \Cref{tatecommuteswithgeomfix}.
\end{proof}
\subsection{Tempered local systems as genuine spectra}

Write $\ul{R}_\bG^{\CB A}$ for the underlying $A$-spectrum of $\ul{R}_{\bG}^{\gl}$. In this section we will identify $\LocSys_{\bG}(\CB A)$ with the $\infty$-category of $\ul{R}_\bG^{\CB A}$-modules in $A$-spectra. 

\begin{construction}
    To begin with we recall that the naive global 2-ring $\LocSys_{\bG}(-)$ is genuine and rigid by \Cref{thm-tempered-is-genuine}. Therefore we obtain a natural transformation 
    \[
    \bGamma_{\bullet}\colon \LocSys_{\bG}\Rightarrow \Sp_\bullet
    \]
    of functors $\Orb{}^{\op}\rightarrow \Cat$ by \Cref{thm-rigid}. As observed in \Cref{rem:genuine_sect_lax}, $\bGamma_\bullet$ is a natural transformation of lax monoidal functors, and so lifts to a natural transformation
    \[
    \bGamma_{\bullet}\colon \LocSys_{\bG}(-)\Rightarrow \Mod_{\ul{R}_{\bG}^\bullet},
    \] valued in modules over the image of the unit in $\LocSys_{\bG}(\bullet)$, which agrees with $\ul{R}_{\bG}^\bullet$ by \Cref{rem:global_sec_of_unit}. We note that this is again a transformation of lax monoidal right adjoints. 
\end{construction}

\begin{lemma}\label{lem:bGamma_gen_fixed_temp}  
    Let $A$ be a finite abelian group and let $\bG$ be an oriented $\rP$-divisible group. Then for every subgroup $B\subset A$ the triangle
    \[\begin{tikzcd}
        {\LocSys_{\bG}(\CB A)} && {\makebox[4em][l]{$\Sp_A$}} \\
        & {\Sp}
        \arrow["{\bGamma_A}", from=1-1, to=1-3]
        \arrow["(-)^B", from=1-3, to=2-2]
        \arrow["(-)^B"', from=1-1, to=2-2]
    \end{tikzcd}\] 
    commutes.
\end{lemma}

\begin{proof}
    Because $\bGamma_A$ commutes with restriction to subgroups it suffices to prove the lemma for all $A$ with $B=A$. In both cases the functor $(-)^A$ is given by mapping out of the unit. Therefore the result follows from the fact that the left adjoint $F_A$ of $\bGamma_A$ is strong monoidal, see \Cref{rem:genuine_sect_lax}.
\end{proof}

\begin{theorem}\label{thm-localsym-as-spectra}
    Let $A$ be a finite abelian group and let $\bG$ be an oriented $\rP$-divisible group. Then 
    \[
    \bGamma_{A}\colon \LocSys_{\bG}(\CB A)\rightarrow \Mod_{\ul R_{\bG}^{\CB A}}
    \] is a symmetric monoidal equivalence.
\end{theorem}

\begin{proof}
    Recall that $\bGamma_A\colon \LocSys_{\bG}(\CB A)\rightarrow \Sp_A$ admits a symmetric monoidal left adjoint, which we denote by $F_A$. We will show that the adjunction $F_A\dashv \bGamma_A$ satisfies the criteria of \cite{MNN17}*{Proposition 5.29} and so induces an equivalences as in the statement. For simplicity we recall the conditions of the cited proposition, which are 
    \begin{enumerate}
        \item $F_A\dashv \bGamma_A$ satisfies the left projection formula;
        \item $\bGamma_A$ commutes with colimits;
        \item $\bGamma_A$ is conservative.
    \end{enumerate} 
    For (1) we may appeal to \cite{GrothendieckNeeman16}*{Theorem 1.3}. To show (2) we observe that it suffices to show that the composite $(-)^B\circ \bGamma_A$ preserves colimits for every subgroup $B\subset A$, because the genuine fixed point functors jointly detect colimits. However by \Cref{lem:bGamma_gen_fixed_temp} this agrees with the functor $(-)^B\colon \LocSys_\bG(\CB A)\rightarrow \Sp$, which commutes with colimits by \Cref{megaprop}(5). Finally for (3) let $\CF$ be a tempered local system such that $\bGamma_A(\CF)$ is zero. Then $\CF^B \simeq \bGamma_A(\CF)^B$ is zero for all $B\subset A$, and we conclude that $\CF$ is zero because the evaluation functors $(-)^A$ are jointly conservative on tempered local systems, see \Cref{not-evaluation-functor}.
\end{proof}	

For completeness we record the fact that both notions of geometric fixed points agree.

\begin{proposition}\label{gfpandgamma}
Let $A$ be a finite abelian group and let $\bG$ be an oriented $\rP$-divisible group. Then the triangle
    \[\begin{tikzcd}
        {\LocSys_{\bG}(\CB A)} && {\Mod_{\ul R_{\bG}^{\CB A}}} \\
        & {\Mod_{\Phi^A R^{\CB A}_{\bG}}}
        \arrow["{\bGamma_A}", from=1-1, to=1-3]
        \arrow["{\Phi^A}"', from=1-1, to=2-2]
        \arrow["{\Phi^A}", from=1-3, to=2-2]
    \end{tikzcd}\]
    commutes.
\end{proposition}

\begin{proof}
It suffices to observe that in both cases geometric fixed points is given by localizing away from the localizing subcategory generated by objects induced from proper subgroups $B\subset A$. 
\end{proof}

Finally we also show that the equivalence $\bGamma_A\colon \LocSys_{\bG}(\CB A) \rightarrow \Mod_{\ul{R}_\bG^{\CB A}}$ is natural in $\Glo{}^{\op}$, in particular with inflation along a surjective group homomorphism. Recall that $\bGamma_\bullet \colon \LocSys_\bG(\bullet)\Rightarrow \Sp_\bullet$ was constructed in \Cref{const:global_sect} by passing a natural transformation $F_\bullet$ of functors $\Glo{}^{\op}\rightarrow \Cat$ through the mate equivalence, and so canonically laxly commutes. By the 2-functoriality of taking modules, the lax transformation $\bGamma_\bullet$ lifts to a lax natural transformation 
\[
\LocSys_\bG(\bullet)\Rightarrow \Mod_{\ul{R}_{\bG}^{\bullet}}
\] 
taking values in modules over the unit. We write $R_\alpha\colon \alpha^*\bGamma_C(-)\Rightarrow \bGamma_A(\alpha^*(-))$ for the natural transformation associated to a map $\CB \alpha\colon \CB A \rightarrow \CB C$ of global spaces.

\begin{proposition}\label{prop:temp_eq_natural}
Let $\CB \alpha\colon \CB A\rightarrow \CB C$ be an arbitrary morphism. Then the natural transformation 
\[R_\alpha\colon \alpha^*\bGamma_C(-)\Rightarrow \bGamma_A(\alpha^*(-))\]
is an equivalence. 
In other words, the square 
\[
    \begin{tikzcd}
        \LocSys_\bG(\CB C)\arrow[r, "\bGamma_C"] \arrow[d, "\alpha^*"'] & \Mod_{\ul R_\bG^{\CB{C}}}  \arrow[d, "\alpha^*"] \\
        \LocSys_\bG(\CB A) \arrow[r, "\bGamma_A"] & \Mod_{\ul R_\bG^{\CB{A}}} 
    \end{tikzcd}
    \] canonically commutes.
\end{proposition}

\begin{proof}
Recall that every map of global spaces factors into a faithful map followed by a \quotientmap{}, and therefore we can consider these two cases separately. The case of a faithful map is already shown in \Cref{thm-rigid}. Therefore we may assume that $\CB \alpha \colon \CB A \rightarrow \CB C$ is a \quotientmap{}, that is $\alpha\colon A\rightarrow C$ is a surjective group homomorphism. To show the result it suffices to show that $R_\alpha$ is an equivalence after taking geometric fixed points for all subgroups $A_0\subset A$. However because $\bGamma_\bullet$ commutes with faithful maps we may reduce to the case $A_0= A$. 

We compute that for all $\CF\in \LocSys_{\bG}(\CB C)$, 
\begin{align*}
 \Phi^A\alpha^*(\bGamma_C \CF ) = \Phi^A(\alpha^* \bGamma_C \CF\otimes_{\alpha^* \ul{R}_{\bG}^{\CB C}} \ul{R}_{\bG}^{\CB A}) &\simeq \Phi^A\alpha^* \bGamma_C \CF\otimes_{\Phi^A \alpha^* \ul{R}_{\bG}^{\CB C}} \Phi^A\ul{R}_{\bG}^{\CB A}\\
 &\simeq \Phi^C \bGamma_C \CF \otimes_{\Phi^C \ul{R}_{\bG}^{\CB C}} \Phi^A\ul{R}_{\bG}^{\CB A} \\ 
 &\simeq \Phi^C \CF \otimes_{\Phi^C \ul{R}_{\bG}^{\CB C}} \Phi^A\ul{R}_{\bG}^{\CB A}
\end{align*}
where in the last step we used \Cref{gfpandgamma}.
Next we note that because $\alpha$ is surjective, the degenerate ideal of $R_{\bG}^{\CB A}$ is generated by the image of the degenerate ideal of $R_{\bG}^{\CB C}$, see \cite{Ell3}*{Remark 5.2.3}. 
Writing $L_A$ and $L_C$ for the localization at $\Spec(R^{\CB A}_{\bG})^{\deg}$ and $\Spec(R^{\CB C}_{\bG})^{\deg}$ respectively, we can continue computing
 \[
 \Phi^C \CF \otimes_{\Phi^C \ul{R}_{\bG}^{\CB C}} \Phi^A\ul{R}_{\bG}^{\CB A} \simeq L_C \CF^C\otimes_{L_C R_{\bG}^{\CB C}} L_A R_{\bG}^{\CB A} \simeq L_A(\CF^C\otimes_{R_{\bG}^{\CB C}} R_{\bG}^{\CB A}).
 \] 
On the other hand, 
\[
\Phi^A\bGamma_A\alpha^*\CF \simeq \Phi^A \alpha^* \CF \simeq L_A (\alpha^* \CF)^A \simeq L_A(\CF(\CB A))\simeq L_A(\CF^C\otimes_{R_{\bG}^{\CB C}} R_{\bG}^{\CB A})
\]
where we used \Cref{gfpandgamma}, the definition of $A$-geometric fixed points, the definition of $\alpha^*$, and finally condition (A) of \Cref{def-pretempered}.
\end{proof}

\appendix
\section{Symmetric monoidal higher Brown representability}\label{sec:Brown}

Suppose that $\CC$ is a compactly generated stable $\infty$-category. Then the restricted spectral Yoneda embedding gives an equivalence $\CC\simeq \Fun^{\lex}((\CC^\omega)^{\op},\Sp)$. In this short appendix we show that this equivalence is symmetric monoidal, where the right hand side is endowed with a localized form of Day convolution. We begin by showing that this localization of Day convolution exists.

\begin{lemma}[\cite{bispans}*{Proposition 3.3.4}]\label{lem:day_conv_localization}
    Let $\CC$ be a presentable $\infty$-category and let $I$ be a symmetric monoidal small $\infty$-category with $\kappa$-small limits, such that the tensor product commutes with $\kappa$-small limits in each variable. Then the localization $L\colon \Fun(I,\CC) \rightarrow \Fun^{\rR_\kappa}(I,\CC)$ is compatible with Day convolution.
\end{lemma}

We call the induced symmetric monoidal structure on $\Fun^{\lex}(I,\CC)$ the \textit{localized Day convolution} symmetric monoidal structure.

\begin{definition}
Suppose $\CC$ and $\CD$ are presentably symmetric monoidal $\infty$-categories. Suppose furthermore that $\CC$ is $\kappa$-compactly generated for some $\kappa$. Consider the Day convolution operad $\Fun(\CC,\CD)^\otimes$. Applying the equivalence 
\[
\Fun^{\rR}(\CC^{\op},\CD) \simeq \Fun^{\rR_\kappa}((\CC^{\kappa})^{\op},\CD)
\] and the previous result, we find that $\Fun(\CC^{\op},\CD)^\otimes$ localizes to a symmetric monoidal structure on $\Fun^{\rR}(\CC^{\op},\CD)$. We again call this the localized Day convolution symmetric monoidal structure. We note that one can easily show that the resulting symmetric monoidal structure is independent of the cardinal chosen above.
\end{definition}
Recall that the Lurie tensor product defines a symmetric monoidal functor $\Pr^\rL \times \Pr^\rL \to \Pr^\rL$ and so passes to commutative algebra objects $\CAlg(\Pr^\rL) \times \CAlg(\Pr^\rL) \to \CAlg(\Pr^\rL)$. In particular we obtain a natural symmetric monoidal structure on $\CC \otimes \CD$ for any pair $(\CC, \CD) $ of presentably symmetric monoidal $\infty$-categories.
\begin{theorem}\label{thm:Day-convolution-is-mon-struc-on-PrL-tensor}
Suppose $\CC$ and $\CD$ are presentably symmetric monoidal $\infty$-categories. Then the natural equivalence 
\[
\CC\otimes \CD \xrightarrow{\sim} \Fun^{\rR}(\CC^{\op},\CD)
\]
of \cite{HA}*{Proposition 4.8.1.17} is symmetric monoidal, where the right hand side is equipped with the Day convolution symmetric monoidal structure.
\end{theorem}

\begin{proof}
    Let $\CC$ be $\kappa$-compactly generated. By definition there is an equivalence 
    \[
    \Fun^{\rR}(\CC^{\op},\CD) \simeq \Fun^{\rR_{\kappa}}((\CC^\kappa)^{\op},\CD)
    \] of symmetric monoidal categories, and so it suffices to prove that $\CC\otimes \CD$ is equivalent to the right hand side. We first do this in the case that $\CD \simeq \Spc$, the unit of $\mathrm{Pr}^{\rL}$. In this case the equivalence is given by the restricted Yoneda embedding. In particular it is defined by currying the hom functor, and so is lax symmetric monoidal by the universal property of Day convolution. It suffices to show that it is in fact strong monoidal. We first observe that on tensor products of objects in $\CC^\kappa$ this follows immediately from \cite{HA}*{Corollary 4.8.1.12}. However the symmetric monoidal structure on $\Fun^{\rR_{\kappa}}((\CC^\kappa)^{\op}, \Spc)$ commutes with colimits in each variable, and so we deduce the statement for arbitrary objects in $\CC$.
    
    Now consider the following diagram
\[\begin{tikzcd}
	{\CC\otimes \CD} & {\Fun^{\lex}((\CC^{\kappa})^{\op},\Spc)\otimes \CD} & {\Fun^{\lex}((\CC^{\kappa})^{\op},\CD)} \\
	& {\Fun((\CC^{\kappa})^{\op},\Spc)\otimes \CD} & {\Fun((\CC^{\kappa})^{\op},\CD)}
	\arrow["\sim", from=1-1, to=1-2]
	\arrow[from=1-2, to=1-3]
	\arrow["{L\otimes \Sp}", from=2-2, to=1-2]
	\arrow["\sim", from=2-2, to=2-3]
	\arrow["L"', from=2-3, to=1-3]
\end{tikzcd}\]
    The composite along the top is equivalent to the equivalence $\CC\otimes \CD \simeq \Fun((\CC^\omega)^{\op},\CD)$. Note that the vertical maps are symmetric monoidal localizations by \Cref{lem:day_conv_localization}. Therefore, because the bottom equivalence is symmetric monoidal by \cite{ben-moshe-schlank}*{Proposition~3.10}, the top rightmost horizontal map is again symmetric monoidal. We have previously shown that the first map is symmetric monoidal, and therefore the composite is also symmetric monoidal.
\end{proof}

\begin{corollary}\label{lem:Brown_rep}
Let $\CC$ be a compactly generated symmetric monoidal stable $\infty$-category. Then the restricted spectral Yoneda embedding
\[
y\colon \CC\xrightarrow{\sim} \Fun^{\lex}((\CC^\omega)^{\op}, \Sp)
\] is canonically symmetric monoidal, where the right hand side is given the localized Day convolution structure.
\end{corollary}

\begin{proof}
The equivalence of the corollary factors as the composite 
\[
\CC\xrightarrow{\sim} \CC\otimes \Sp \xrightarrow{\sim} \Fun^{\lex}((\CC^{\omega})^{\op},\Sp).
\]
The first map is symmetric monoidal by \cite{HA}*{Proposition 4.8.2.10} and the second by the previous theorem.
\end{proof}

\bibliographystyle{plain}
\bibliography{reference}

\begin{bibdiv}
\begin{biblist}

\bib{AG20}{article}{
      author={Abell\'{a}n~Garc\'{\i}a, Fernando},
       title={Marked colimits and higher cofinality},
        date={2022},
        ISSN={2193-8407},
     journal={J. Homotopy Relat. Struct.},
      volume={17},
      number={1},
       pages={1\ndash 22},
         url={https://doi.org/10.1007/s40062-021-00296-2},
}

\bib{Mate2024}{article}{
      author={Abell\'an, Fernando},
      author={Gagna, Andrea},
      author={Haugseng, Rune},
       title={Straightening for lax transformations and adjunctions of $(\infty,2)$-categories},
        date={2024},
     journal={{arXiv:2404.03971}},
}

\bib{AMR2019}{article}{
      author={Ayala, David},
      author={Mazel-Gee, Aaron},
      author={Rozenblyum, Nick},
       title={Stratified noncommutative geometry},
        date={2024},
        ISSN={0065-9266,1947-6221},
     journal={Mem. Amer. Math. Soc.},
      volume={297},
      number={1485},
       pages={iii+260},
      review={\MR{4760797}},
}

\bib{AMR2021}{article}{
      author={{Ayala}, David},
      author={{Mazel-Gee}, Aaron},
      author={{Rozenblyum}, Nick},
       title={{Derived Mackey functors and $C_{p^n}$-equivariant cohomology}},
        date={2021},
     journal={arXiv:2105.02456},
}

\bib{BCHNP}{article}{
      author={Barthel, Tobias},
      author={Castellana, Nat\`alia},
      author={Heard, Drew},
      author={Naumann, Niko},
      author={Pol, Luca},
       title={Quillen stratification in equivariant homotopy theory},
        date={2025},
        ISSN={0020-9910,1432-1297},
     journal={Invent. Math.},
      volume={239},
      number={1},
       pages={219\ndash 285},
         url={https://doi.org/10.1007/s00222-024-01301-0},
      review={\MR{4841779}},
}

\bib{GrothendieckNeeman16}{article}{
      author={Balmer, Paul},
      author={Dell'Ambrogio, Ivo},
      author={Sanders, Beren},
       title={Grothendieck-{Neeman} duality and the {Wirthm{\"u}ller} isomorphism},
        date={2016},
        ISSN={0010-437X},
     journal={Compos. Math.},
      volume={152},
      number={8},
       pages={1740\ndash 1776},
}

\bib{Berman}{article}{
      author={{Berman}, John~D.},
       title={On lax limits in infinity categories},
        date={2020},
     journal={arXiv:2006.10851},
}

\bib{ben-moshe-schlank}{article}{
      author={Ben-Moshe, Shay},
      author={Schlank, Tomer~M.},
       title={Higher semiadditive algebraic {K}-theory and redshift},
        date={2024},
        ISSN={0010-437X,1570-5846},
     journal={Compos. Math.},
      volume={160},
      number={2},
       pages={237\ndash 287},
         url={https://doi.org/10.1112/s0010437x23007595},
}

\bib{Chau2021}{article}{
      author={Chua, Dexter},
       title={{$C_2$}-equivariant topological modular forms},
        date={2022},
        ISSN={2193-8407,1512-2891},
     journal={J. Homotopy Relat. Struct.},
      volume={17},
      number={1},
       pages={23\ndash 75},
         url={https://doi.org/10.1007/s40062-021-00297-1},
}

\bib{CLL_Global}{article}{
      author={Cnossen, Bastiaan},
      author={Lenz, Tobias},
      author={Linskens, Sil},
       title={{Parametrized stability and the universal property of global spectra}},
        date={2023},
     journal={arXiv:2301.08240},
}

\bib{CLL_Partial}{article}{
      author={Cnossen, Bastiaan},
      author={Lenz, Tobias},
      author={Linskens, Sil},
       title={Partial parametrized presentability and the universal property of equivariant spectra},
        date={2023},
     journal={{arXiv:2307.11001}},
}

\bib{Cnossen22}{article}{
      author={{Cnossen}, Bastiaan},
       title={Twisted ambidexterity in equivariant homotopy theory},
        date={2023},
     journal={arXiv:2303.00736},
}

\bib{bispans}{article}{
      author={Cnossen, Bastiaan},
      author={Rune, Haugseng},
      author={Lenz, Tobias},
      author={Linskens, Sil},
       title={Homotopical commutative rings and bispans},
        date={2024},
     journal={arXiv:2403.06911},
}

\bib{Spherical21}{article}{
      author={Dyckerhoff, Tobias},
      author={Kapranov, Mikhail},
      author={Schechtman, Vadim},
      author={Soibelman, Yan},
       title={Spherical adjunctions of stable {$\infty $}-categories and the relative {S}-construction},
        date={2024},
        ISSN={0025-5874,1432-1823},
     journal={Math. Z.},
      volume={307},
      number={4},
       pages={Paper No. 73, 59},
         url={https://doi.org/10.1007/s00209-024-03549-x},
      review={\MR{4771790}},
}

\bib{GGN}{article}{
      author={Gepner, David},
      author={Groth, Moritz},
      author={Nikolaus, Thomas},
       title={Universality of multiplicative infinite loop space machines},
        date={2015},
        ISSN={1472-2747,1472-2739},
     journal={Algebr. Geom. Topol.},
      volume={15},
      number={6},
       pages={3107\ndash 3153},
         url={https://doi.org/10.2140/agt.2015.15.3107},
}

\bib{GH}{article}{
      author={Gepner, David},
      author={Henriques, Andre},
       title={Homotopy theory of orbispaces},
        date={2007},
     journal={arXiv:math\slash 0701916},
}

\bib{GHL21}{article}{
      author={Gagna, Andrea},
      author={Harpaz, Yonatan},
      author={Lanari, Edoardo},
       title={Gray tensor products and {L}ax functors of {$(\infty,2)$}-categories},
        date={2021},
        ISSN={0001-8708},
     journal={Adv. Math.},
      volume={391},
       pages={Paper No. 107986, 32},
         url={https://doi.org/10.1016/j.aim.2021.107986},
}

\bib{GHN}{article}{
      author={Gepner, David},
      author={Haugseng, Rune},
      author={Nikolaus, Thomas},
       title={Lax colimits and free fibrations in {$\infty$}-categories},
        date={2017},
        ISSN={1431-0635},
     journal={Doc. Math.},
      volume={22},
       pages={1225\ndash 1266},
}

\bib{GKV95}{article}{
      author={Ginzburg, Victor},
      author={Kapranov, Mikhail},
      author={Vasserot, Eric},
       title={Elliptic algebras and equivariant elliptic cohomology},
        date={1995},
     journal={arXiv:q-alg\slash 9505012},
}

\bib{GM20}{article}{
      author={Gepner, David},
      author={Meier, Lennart},
       title={On equivariant topological modular forms},
        date={2023},
        ISSN={0010-437X,1570-5846},
     journal={Compos. Math.},
      volume={159},
      number={12},
       pages={2638\ndash 2693},
         url={https://doi.org/10.1112/s0010437x23007509},
      review={\MR{4664814}},
}

\bib{GM2}{unpublished}{
      author={Gepner, David},
      author={Meier, Lennart},
       title={Equivariant elliptic cohomology with integral coefficients},
        date={2024},
        note={In preparation},
}

\bib{HHLN1}{article}{
      author={Haugseng, Rune},
      author={Hebestreit, Fabian},
      author={Linskens, Sil},
      author={Nuiten, Joost},
       title={Lax monoidal adjunctions, two-variable fibrations and the calculus of mates},
        date={2023},
     journal={Proceedings of the London Mathematical Society},
         url={https://arxiv.org/abs/2011.08808},
}

\bib{HHLN2}{article}{
      author={Haugseng, Rune},
      author={Hebestreit, Fabian},
      author={Linskens, Sil},
      author={Nuiten, Joost},
       title={Two-variable fibrations, factorisation systems and $\infty $-categories of spans},
        date={2023},
     journal={Forum of Mathematics, Sigma},
      volume={11},
       pages={e111},
}

\bib{axiomatic}{article}{
      author={Hovey, Mark},
      author={Palmieri, John~H.},
      author={Strickland, Neil~P.},
       title={Axiomatic stable homotopy theory},
        date={1997},
        ISSN={0065-9266,1947-6221},
     journal={Mem. Amer. Math. Soc.},
      volume={128},
      number={610},
       pages={x+114},
         url={https://doi.org/10.1090/memo/0610},
}

\bib{HY}{article}{
      author={Horev, Asaf},
      author={Yanovski, Lior},
       title={On conjugates and adjoint descent},
        date={2017},
        ISSN={0166-8641,1879-3207},
     journal={Topology Appl.},
      volume={232},
       pages={140\ndash 154},
         url={https://doi.org/10.1016/j.topol.2017.10.007},
      review={\MR{3720887}},
}

\bib{LenzGglobal}{article}{
      author={Lenz, Tobias},
       title={{$G$}-global homotopy theory and algebraic {$K$}-theory},
        date={2025},
        ISSN={0065-9266,1947-6221},
     journal={Mem. Amer. Math. Soc.},
      volume={306},
      number={1545},
       pages={v+246},
         url={https://doi.org/10.1090/memo/1545},
      review={\MR{4865560}},
}

\bib{LinskensGlobalization}{article}{
      author={Linskens, Sil},
       title={Globalizing and stabilizing global $\infty$-categories},
        date={2024},
     journal={arXiv:2401.02264},
}

\bib{LNP}{article}{
      author={Linskens, Sil},
      author={Nardin, Denis},
      author={Pol, Luca},
       title={Global homotopy theory via partially lax limits},
        date={2025},
        ISSN={1465-3060,1364-0380},
     journal={Geom. Topol.},
      volume={29},
      number={3},
       pages={1345\ndash 1440},
         url={https://doi.org/10.2140/gt.2025.29.1345},
      review={\MR{4918109}},
}

\bib{Survey}{incollection}{
      author={Lurie, J.},
       title={A survey of elliptic cohomology},
        date={2009},
   booktitle={Algebraic topology},
      series={Abel Symp.},
      volume={4},
   publisher={Springer, Berlin},
       pages={219\ndash 277},
         url={https://doi.org/10.1007/978-3-642-01200-6_9},
      review={\MR{2597740}},
}

\bib{HTT}{book}{
      author={Lurie, Jacob},
       title={Higher topos theory},
      series={Annals of Mathematics Studies},
   publisher={Princeton University Press, Princeton, NJ},
        date={2009},
      volume={170},
        ISBN={978-0-691-14049-0; 0-691-14049-9},
         url={https://doi.org/10.1515/9781400830558},
}

\bib{DAG_QCoh}{book}{
      author={Lurie, Jacob},
       title={{Derived algebraic geometry {VIII}: Quasi-coherent sheaves}},
        date={2011},
        note={Available from the author's webpage at \url{https://www.math.ias.edu/~lurie/papers/DAG-VIII.pdf}},
}

\bib{SAG}{unpublished}{
      author={Lurie, Jacob},
       title={Spectral algebraic geometry},
        date={2016},
        note={Available from the author's webpage at \url{https://www.math.ias.edu/~lurie/papers/SAG-rootfile.pdf}},
}

\bib{HA}{book}{
      author={Lurie, Jacob},
       title={{Higher algebra}},
        date={2017},
        note={Available from the author's webpage at \url{http://www.math.harvard.edu/~lurie/papers/HA.pdf}},
}

\bib{Ell1}{unpublished}{
      author={Lurie, Jacob},
       title={Elliptic cohomology {I}: Spectral abelian varieties},
        date={2018},
        note={\url{https://www.math.ias.edu/~lurie/papers/Elliptic-I.pdf}},
}

\bib{Ell2}{unpublished}{
      author={Lurie, Jacob},
       title={Elliptic cohomology {II}: Orientations},
        date={2018},
        note={\url{https://www.math.ias.edu/~lurie/papers/Elliptic-II.pdf}},
}

\bib{Ell3}{unpublished}{
      author={Lurie, Jacob},
       title={Elliptic cohomology {III}: Tempered cohomology},
        date={2019},
        note={\url{https://www.math.ias.edu/~lurie/papers/Elliptic-III.pdf}},
}

\bib{Ma16}{article}{
      author={Mathew, Akhil},
       title={The {Galois} group of a stable homotopy theory},
    language={English},
        date={2016},
        ISSN={0001-8708},
     journal={Adv. Math.},
      volume={291},
       pages={403\ndash 541},
}

\bib{Meier2022}{article}{
      author={Meier, Lennart},
       title={Topological modular forms with level structure: decompositions and duality},
        date={2022},
        ISSN={0002-9947,1088-6850},
     journal={Trans. Amer. Math. Soc.},
      volume={375},
      number={2},
       pages={1305\ndash 1355},
         url={https://doi.org/10.1090/tran/8514},
      review={\MR{4369249}},
}

\bib{MNN17}{article}{
      author={Mathew, Akhil},
      author={Naumann, Niko},
      author={Noel, Justin},
       title={Nilpotence and descent in equivariant stable homotopy theory},
        date={2017},
        ISSN={0001-8708,1090-2082},
     journal={Adv. Math.},
      volume={305},
       pages={994\ndash 1084},
         url={https://doi.org/10.1016/j.aim.2016.09.027},
}

\bib{Nikoluas}{article}{
      author={{Nikolaus}, Thomas},
       title={{Stable $\infty$-Operads and the multiplicative Yoneda lemma}},
        date={2016-08},
     journal={arXiv e-prints},
       pages={arXiv:1608.02901},
      eprint={1608.02901},
}

\bib{Rezk}{book}{
      author={Rezk, Charles},
       title={{Global homotopy theory and cohesion}},
        date={2014},
        note={Avaliable from the author's webpage at \url{https://faculty.math.illinois.edu/~rezk/global-cohesion.pdf}},
}

\bib{Schwede18}{book}{
      author={Schwede, Stefan},
       title={Global homotopy theory},
      series={New Mathematical Monographs},
   publisher={Cambridge University Press, Cambridge},
        date={2018},
      volume={34},
        ISBN={978-1-108-42581-0},
         url={https://doi.org/10.1017/9781108349161},
}

\bib{Schwede20}{article}{
      author={Schwede, Stefan},
       title={Orbispaces, orthogonal spaces, and the universal compact {L}ie group},
        date={2020},
        ISSN={0025-5874},
     journal={Math. Z.},
      volume={294},
      number={1-2},
       pages={71\ndash 107},
         url={https://doi.org/10.1007/s00209-019-02265-1},
}

\bib{stacks-project}{misc}{
      author={{Stacks project authors}, The},
       title={The stacks project},
         how={\url{https://stacks.math.columbia.edu}},
        date={2024},
}

\end{biblist}
\end{bibdiv}
\end{document}